\documentclass[11pt,reqno]{amsart}

\usepackage[a4paper,top=25mm,bottom=29mm,left=28mm,right=28mm]{geometry}
\usepackage[T1]{fontenc}
\usepackage{lmodern}
\usepackage{amsmath,amssymb,amsfonts,amsthm,mathtools,bm,mathrsfs}
\usepackage{float,graphicx,subcaption,placeins,adjustbox}
\usepackage{booktabs,tabularx,array,longtable,multirow}
\usepackage{comment,enumerate,enumitem,xcolor,microtype,siunitx,url}

\usepackage[hidelinks]{hyperref}

\usepackage[nameinlink,capitalise,noabbrev]{cleveref}
\graphicspath{{figures/}}

\theoremstyle{plain}
\newtheorem{theorem}{Theorem}[section]
\newtheorem*{theorem*}{Theorem}

\newtheorem{lemma}[theorem]{Lemma}
\newtheorem*{claim*}{Claim}
\newtheorem{proposition}[theorem]{Proposition}

\theoremstyle{definition}

\newtheorem{example}[theorem]{Example}
\theoremstyle{remark}

\newtheorem{remark}[theorem]{Remark}
\numberwithin{equation}{section}

\DeclarePairedDelimiter{\norm}{\lVert}{\rVert}
\newcommand{\Pent}{\mathbb{P}}
\newcommand{\SwitchAlg}{\mathcal A}
\newcommand{\Vpure}{V}

\newcommand{\Imag}{\operatorname{Im}}

\newcommand{\Span}{\operatorname{span}}

\newcommand{\R}{\mathbb{R}}

\newcommand{\C}{\mathbb{C}}

\newcommand{\Hom}{\mathrm{Hom}}

\newcommand{\Tr}{\mathrm{Tr}}

\newcommand{\diag}{\operatorname{diag}}
\newcommand{\Aut}{\operatorname{Aut}}

\newcommand{\SO}{\mathrm{SO}}
\newcommand{\ignore}[1]{}
\newcommand{\kg}{\kappa_\gamma}
\newcommand{\sg}{\sigma_\gamma}
\newcommand{\crossg}{\mathbin{\times_{\!\gamma}}}

\newcommand{\N}{N}

\newcommand{\Alt}{\operatorname{Alt}}

\newcommand{\ip}[2]{\left\langle #1,#2\right\rangle}
\newcommand{\dd}{\,\mathrm d}

\newcommand{\Frob}{\mathrm F}

\newcommand{\Nuc}{\operatorname{Nuc}}

\begin{document}
\raggedbottom

\title[From Quaternions to Pentarnions I: Algebra,
geometry, and a spin model]{
From Quaternions to Pentarnions I: Algebra,
geometry, and a spin model}

\author{Alireza Poshtkohi}
\address{School of Physics, Engineering and Computer Science, University of Hertfordshire, Hatfield, Hertfordshire, United Kingdom}
\email{a.poshtkohi@herts.ac.uk}

\author{Susanne Pumpl\"un}
\address{School of Mathematical Sciences, University of Nottingham,
Nottingham NG7 2RD, United Kingdom}
\email{Susanne.Pumpluen@nottingham.ac.uk}

\keywords{
Nonassociative algebras, quadratic algebras, Pentarnion algebras, automorphism groups, rotational geometry, zero divisors, classical spin, Hamiltonian systems.}

\subjclass[2020]{Primary: 17A45; Secondary: 17A36, 70H05, 37J06}

\date{\today}

\begin{abstract}
We introduce a family of five-dimensional non-flexible quadratic real algebras $\{\mathcal A_\gamma\}_{0\leq\gamma\leq1}$ with canonical basis $1,i,j,k,\ell$, in which $\gamma$ changes the multiplication but not their norm, and study how this change affects geometry and dynamics. The algebras are pairwise non-isomorphic, $\mathcal A_0$ contains Hamilton's quaternion algebra, and $\mathcal A_1$ is the Pentarnion algebra.
The alternating part of the trilinear form $\langle u\times_\gamma v,w\rangle$ induced by
the multiplication of $\mathcal{A}_\gamma$ defines a
three-dimensional subspace
$H_\gamma\subset\operatorname{span}_{\mathbb R}\{i,j,k,\ell\}$,
tilted by the angle $\arctan(\gamma/\sqrt3)$, whose rotation group
$G_\gamma\cong\mathrm{SO}(3)$ contains the automorphism group of
$\mathcal A_\gamma$. For $\gamma_1\ne\gamma_2$ the groups
$G_{\gamma_1}$ and $G_{\gamma_2}$ generate $\mathrm{SO}(V)\cong\mathrm{SO}(4)$. For the six ways of ordering and bracketing the product of $i,p,p$ in $\mathcal A_\gamma$, two map the unit circle in ${\rm span}\{j,k\}$ to a point and, for $\gamma>0$, four map it to ellipses.

As an application, we derive a Hamiltonian model of a classical spin on
$H_\gamma$ whose multiplication-induced quadratic energy is isotropic
for $\gamma=0$ and anisotropic for $\gamma>0$. It conserves spin length and, for fixed $\gamma$, energy,
and satisfies an exact work identity when $\gamma$ varies in time. Numerical experiments with periodically
varying $\gamma$ indicate both regular and irregular spin dynamics.
\end{abstract}

\maketitle

\section*{Introduction}

Quadratic algebras constitute a large class of nonassociative unital algebras; see, for example, \cite{Osborn1962,Pumpluen2008} and
\cite{Dieterich2000,DieterichRub2010}. Most of the existing literature is interested in division algebras, hence most known families have dimension four or eight. Non-flexible quadratic
algebras have been constructed for instance as mutation algebras of Hurwitz algebras
\cite{Pumpluen2014}.

We define a new class of five-dimensional non-flexible quadratic real
algebras $\mathcal A_\gamma(a,b,c,d)$ for suitable $a,b,c,d\in \mathbb{R}^\times$,  $\gamma\in \R$, and
 focus on  the family
$\{\mathcal{A}_\gamma\}_{0\leq\gamma\leq 1}$,  $\mathcal{A}_\gamma=\mathcal A_\gamma(1,1,1,1)$.

 The algebra $\mathcal{A}_\gamma$ has a canonical basis $1,i,j,k,\ell$ and  $\mathcal{A}_0$
contains Hamilton's quaternion algebra $\mathbb{H}$ as the subalgebra generated by $1,i,j,k$, while
$\mathcal{A}_1$ is the \emph{Pentarnion algebra} $\Pent$ from \cite{PoshtkohiPumplunSubmitted}. Our main contribution is to link  the change of multiplication of the real algebras $\mathcal{A}_\gamma$ as $\gamma$ varies, which we call  \emph{switching}, to changes in geometry and algebra-valued dynamical systems.

 After the preliminaries in \cref{sec:prelim}, we
define the real algebras $\mathcal A_\gamma(a,b,c,d)$
in \cref{sec:switch}. In \Cref{sec:switching-algebra},
we show that the algebras $\mathcal A_\gamma$, $0\leq\gamma\leq 1$, have the same
norm and are pairwise non-isomorphic, since the isomorphism invariants
$\|\mathrm{Assoc}_\gamma\|_F^2=3(41\gamma^2+31)/(1+\gamma^2)$ and $\beta_\gamma=\arctan(\gamma/\sqrt3)$ are strictly
increasing in $\gamma$, and we compute their automorphism groups and nuclei  (\cref{thm:pairwise-nonisomorphic}, \cref{thm:minimal-nuclei,thm:automorphism-groups}).

\Cref{sec:rotational-geometry} shows that the alternating part of the multiplication in $\mathcal{A}_\gamma$ defines a three-dimensional subspace
$H_\gamma\subset\operatorname{span}_{\mathbb R}\{i,j,k,\ell\}$, tilted by the angle
$\arctan(\gamma/\sqrt3)$, whose rotation group $G_\gamma\cong\mathrm{SO}(3)$ contains $\Aut(\mathcal A_\gamma)$; and that  $G_{\gamma_1}$ and
$G_{\gamma_2}$ generate $\mathrm{SO}(V)$ for $\gamma_1\ne\gamma_2$
(Theorems \ref{thm:Hgamma}, \ref{thm:exact-shadow}, \ref{thm:so4-generation}).

\Cref{sec:bracket-geometry} compares the six different homogeneous quadratic maps in
\cref{eq:six-maps-01,eq:six-maps-23,eq:six-maps-45} obtained by
ordering and bracketing  $i,p\in\mathcal A_\gamma$ in different ways (Theorem \ref{thm:associator-ellipse}).
\cref{sec:spin-dynamics} derives a Hamiltonian model of a unit
classical spin from the multiplication; it is integrable for fixed $\gamma$ and satisfies exact
conservation laws and a work identity
(\cref{thm:static-spin-integrability,thm:work-identity}).

Results that rely on computer algebra are marked as such;
the accompanying code and data are described in
Appendix~\ref{app:reproducibility}.
Appendix~\ref{app:coordinate-matrices} provides the coordinate
product, multiplication matrices and Jacobian formulas.
The numerical methods for the spin model are described in
Section~\ref{sec:spin-dynamics}.

 The main notation is collected in Table~\ref{tab:notation}.

\begin{table}[ht]\label{tab:notation}
\centering
\renewcommand{\arraystretch}{1.25}
\small
\begin{tabular}{@{}l p{0.56\textwidth} l@{}}
\hline
Notation & Meaning & Defined in\\
\hline
$\mathcal A_\gamma(a,b,c,d)$ & the five-dimensional real algebra with basis $1,i,j,k,\ell$ & Section~\ref{sec:switch}\\
$\mathcal A_\gamma$, $\mathbb P$ & $\mathcal A_\gamma=\mathcal A_\gamma(1,1,1,1)$, $\mathbb P=\mathcal A_1$; $\mathbb H\subset\mathcal A_0$ & Section~\ref{sec:switch}, \ref{sec:4.1}\\
$\star_\gamma$, $\times_\gamma$, $\Pi_\gamma$ & multiplication of $\mathcal A_\gamma$, $u\star_\gamma v=-\langle u,v\rangle+u\times_\gamma v$, $\Pi_\gamma(u\wedge v)=u\times_\gamma v$ & Section~\ref{sec:switch}, \eqref{eq:jordan-independent}\\
$N(p)$, $\bar p$ & norm $N(p)=p_0^2+\|u\|^2$ and conjugate $\bar p=p_0-u$ of $p=p_0+u$ & Theorem~\ref{thm:quadratic-structure}\\
$V$, $E_0$ & $V=\operatorname{span}\{i,j,k,\ell\}$, $E_0=\operatorname{span}\{i,j,k\}$ & Section \ref{sec:4.1}\\
$F^\gamma_u$ & $F^\gamma_u(v)=u\times_\gamma v$ for $u,v\in V$ & Section \ref{sec:4.1}\\
$\kappa_\gamma$, $\sigma_\gamma$ & $\kappa_\gamma=1/\sqrt{1+\gamma^2}$, $\sigma_\gamma=\gamma/\sqrt{1+\gamma^2}$ & \eqref{eq:switch-angle}, \eqref{eq:trB}\\
$\operatorname{tr}_\gamma$ & $\operatorname{tr}_\gamma(u,v)=\operatorname{Tr}(F^\gamma_uF^\gamma_v)$ & Lemma~\ref{lem:beta}\\
$a,b,c,d$ & $a=\tfrac{i+j+k}{\sqrt3}$, $b=\tfrac{i-j}{\sqrt2}$, $c=\tfrac{i+j-2k}{\sqrt6}$, $d=\ell$ & \eqref{eq:adapted-basis}\\
$[x,y,z]_\gamma$ & associator $(x\star_\gamma y)\star_\gamma z-x\star_\gamma(y\star_\gamma z)$ & Section~\ref{sec:4.3}\\
$T_\gamma$, $\varphi_\gamma$ & $T_\gamma(u,v,w)=\langle u\times_\gamma v,w\rangle$, $\varphi_\gamma=\operatorname{Alt}T_\gamma$ & Section~\ref{sec:5.1} \eqref{eq:rotational-three-form}\\
$\beta_\gamma$ & $\arctan(\gamma/\sqrt3)$ & \eqref{eq:tilt-angle}\\
$h_\gamma$, $n_\gamma$ & $h_\gamma=\cos\beta_\gamma\,a+\sin\beta_\gamma\,d$, $n_\gamma=-\sin\beta_\gamma\,a+\cos\beta_\gamma\,d$ & (5.5)\\
$H_\gamma$ & $n_\gamma^\perp=\operatorname{span}\{h_\gamma,b,c\}$ & Theorem~\ref{thm:Hgamma}\\
$\boxtimes_\gamma$ & cross product on $H_\gamma$ in the basis $h_\gamma,b,c$ & \eqref{eq:box-cross}\\
$G_\gamma$ & $\{Q\in\mathrm{SO}(V)\mid Qn_\gamma=n_\gamma\}\cong\mathrm{SO}(3)$ & Section \ref{sec:5.3}\\
$R_{\gamma,u}(\theta)$ & rotation of $H_\gamma$ by $\theta$ about $\mathbb Ru$, fixing $n_\gamma$ & \eqref{eq:rodrigues-shadow}\\
$L^\gamma_a$, $R^\gamma_a$ & $L^\gamma_a(q)=a\star_\gamma q$, $R^\gamma_a(q)=q\star_\gamma a$ & Section \ref{sec:7.1}\\
$M_{r,\gamma}$ & the six maps $i\star_\gamma(p\star_\gamma p),\dots,p\star_\gamma(p\star_\gamma i)$, $r=0,\dots,5$ & \eqref{eq:six-maps-01}--\eqref{eq:six-maps-45}\\
$Q_r(p;\gamma)$ & $\det DM_{r,\gamma}(p)$ & \eqref{eq:Qr-definition}\\
$K_\gamma$ & $\operatorname{diag}(A_\gamma,B_\gamma,B_\gamma)$, the energy matrix of the spin model & Theorem~\ref{thm:multiplication-anisotropy}\\
$\mathcal H_\gamma$ & Hamiltonian $\tfrac12y^TK_\gamma y-g_\gamma^Ty$ & \eqref{eq:spin-H-main}\\
\hline
\end{tabular}
\caption{Main notation.}
\label{tab:notation}
\end{table}

The family $\{\mathcal A_\gamma\}_{0\leq\gamma\leq1}$ arose from the first author's attempt to formulate dynamical
systems in which the multiplication, and not only the coefficients of an
equation, can be varied.

There is also a direct geometric motivation. Hamilton
introduced the new direction $k$ through $k=ij$,  so that in the pure quaternions $k$ spans the
orthogonal complement of ${\rm span}\{i,j\}$
\cite[Preface, pp.~43--47]{Hamilton1853}.
After adjoining
a fourth pure direction $\ell$, this complement  becomes the plane $\Span\{k,\ell\}$, which
leads naturally to the multiplications $\star_\gamma$  with
$i\star_\gamma j=\kappa_\gamma k+\sigma_\gamma\ell$ in $\{\mathcal A_\gamma\}_{0\leq\gamma\leq1}$.

Quadratic differential equations over $\mathcal A_\gamma$, including explicit elliptic solutions and a Riccati equation with a Hopf bifurcation, are studied in Part~II~\cite{PoshtkohiPumpluenII}.

\section{Preliminaries} \label{sec:prelim}

\subsection{Nonassociative algebras}
Let  $A$ be an $\R$-vector space. Then $A$ is called an \emph{algebra} over $\R$, if there exists an $\R$-bilinear
map $A\times A\to A$, $(x,y) \mapsto x \cdot y$, usually denoted
by juxtaposition, the  \emph{multiplication} of $A$. An algebra
$A$ is  \emph{unital} if there is an element in $A$, denoted by $1$ or $1_A$, such that $1\cdot x=x\cdot 1=x$ for all $x\in A$.

An algebra $A$ is called \emph{commutative} if $xy=yx$ for all $x,y\in A$, and  \emph{anti-commutative} if $xy=-yx$ for all $x,y\in A$; $A$ is called \emph{flexible} if $x(yx)=(xy)x$ for all  $x,y\in A$, and \emph{alternative} if   $(xx)y=x(xy)$  and  $y(xx)=(yx)x$ for all  $x,y\in A$. An algebra $A$ is called \emph{power-associative} if $x^k$ is well defined for all integers $k>1$, no matter the placement of the parentheses; equivalently, every subalgebra generated by one element and 1 is associative.

 The \emph{left, middle} and \emph{right nucleus} of
$A$ are defined as ${\rm Nuc}_l(A) = \{ x \in A \, \vert \, [x, A, A]
= 0 \}$, ${\rm Nuc}_m(A) = \{ x \in A \,
\vert \, [A, x, A]  = 0 \}$ and  ${\rm
Nuc}_r(A) = \{ x \in A \, \vert \, [A,A, x]  = 0 \}$, where $[x,y,z]=(xy)z-x(yz)$ is the \emph{associator}. All three sets are associative
subalgebras of $A$ and are used for measuring the associativity of $A$ \cite{Schafer1995}.

The set ${\rm Comm}(A)=\{x\in A\,|\,xy=yx \text{ for all }y\in A\}$  is called the \emph{commuter} of $A$ and ${\rm C}(A)=\text{Nuc}_l(A)\cap \text{Nuc}_m(A)\cap\text{Nuc}_r(A) \cap {\rm Comm}(A)$ is the {\it
center} of $A$.

A nonassociative algebra $A\not=0$ has no zero divisors if and only if
$L_a$ and $R_a$ (left and right multiplication with $a$ in $A$) are injective for all $0\not=a\in A$.

 Any (ring) anti-automorphism $\sigma:A\to A$ of period 2 is called an {\it involution} on $A$.
An involution on a unital algebra  is called {\it scalar} if all {\it norms} $\sigma(x)x$ are elements of $F1$.
For every scalar involution $\sigma$, $N:A\to F$ defined via $N_A(x)=\sigma(x)x$ is a quadratic form on $A$, respectively, the {\it trace} $\Tr:A\to F$, $\Tr(x)=\sigma(x)+x$ a linear form on $A$ \cite{McCrimmon1985}.

\subsection{Quadratic algebras}\label{subsec:quadr}
A unital real algebra $A$ is called {\it quadratic}, if there exists a quadratic form $N \colon A \to \R$ such that
\begin{equation}\label{equ:1}
x^2 - N(1_A, x)x + N(x)1_A = 0
\end{equation}
for all $x \in A$,
 where $N(x,y) =
N (x+y) - N(x) -N(y)$ denotes the
 symmetric bilinear form induced by $N$. This automatically implies that
$N(1_A) = 1$. The form $N$ is uniquely determined
and called the {\it norm} $N$ of the quadratic algebra $A$.
 The existence of a scalar involution  on a unital algebra $A$ implies that $A$ is quadratic.
 Moreover, the unital algebras with scalar involution are precisely those quadratic algebras with normal trace \cite[Theorem 1.1, Proposition 2.2]{McCrimmon1985}, since we work over base fields, all our quadratic algebras have a scalar involution and a trace $\Tr:A\to F$, $\Tr(x)=\sigma(x)+x$.

 Every quadratic algebra is power-associative.
Two isomorphic quadratic algebras over $\R$ have isometric norms. The converse is not true.

Every $x\in A$ generates an associative subalgebra of dimension $\leq 2$ in $A$ which is spanned by $1$ and $x$.

If $A$ is a quadratic algebra over $\R$, then $A=\R 1\oplus V$ with $V=\{ u\in A\,|\, \Tr(u)=0\}$ the pure part of $A$. Define $\times: V\times V\rightarrow V$,
$u\times v=pr(uv)$ with $pr:A \rightarrow V$ the canonical projection map. Then $(V,\times)$ is an anti-commutative algebra over $F$ with $$uv=-\frac{1}{2}N(u,v)+u\times v.$$
The algebra $(V,\times)$ is called  the  \emph{pure part} of $A$.

A quadratic form $N \colon A \to \R$ is called {\it nondegenerate} if its
 induced symmetric bilinear form $N(u,v)  =
N(u+v)-N(u)-N(v)$ determines an $\R$-vector space isomorphism
$A \overset{\sim}\to A^\vee=\Hom_\R(A,F)$. A quadratic form $N \colon A \to F$ on a unital algebra $A$ is {\it multiplicative}  if  $N(uv)=N(u)N(v)$ for all $u,v\in A$.
A unital algebra $A$ is called a {\it composition algebra} over $\R$
if it admits a nondegenerate multiplicative quadratic form $N \colon A \to F$. Composition algebras are quadratic alternative and $N(1_A)=1$.  A nondegenerate  multiplicative  quadratic form $N$ on  a unital composition algebra $A$ is its norm as a quadratic algebra, thus uniquely determined up to isometry \cite[p.~454 ff.]{KMRT}
 and called the {\it norm} of $A$.

 A real algebra $A$ is \emph{locally complex} if every element $x\in A \setminus \R$, together with 1, generates a subfield of $A$ that is isomorphic to $\mathbb{C}$ \cite{BresarSemrlSpenko2011}.

From now on we assume that  $F=\R$.

\section{The algebra $\mathcal{A}_\gamma(a,b,c,d)$}\label{sec:switch}

 Fix
$a,b,c,d\in F^\times$ and $\gamma\in \R$ such that
\[
 a+b,\quad a+c,\quad b+c,\quad a+\gamma^2d,\quad
 b+\gamma^2d,\quad c+\gamma^2d
\]
are nonzero and the six square roots
\[
 \sqrt{\frac{ab}{c+\gamma^2d}},\qquad
 \sqrt{\frac{ac}{b+\gamma^2d}},\qquad
 \sqrt{\frac{bc}{a+\gamma^2d}},\qquad
 \sqrt{\frac{ad}{b+c}},\qquad
 \sqrt{\frac{bd}{a+c}},\qquad
 \sqrt{\frac{cd}{a+b}}
\]
belong to $\R$.
Specialisations at $\gamma=0$ or $\gamma=1$ are considered only
when the above conditions hold.

We define the unital five-dimensional algebra
\[
 \mathcal A_\gamma(a,b,c,d)
 =
 \R1\oplus \R i\oplus \R j\oplus \R k\oplus \R\ell
\]
by
\[
 i^2=-a,\qquad
 j^2=-b,\qquad
 k^2=-c,\qquad
 \ell^2=-d,
\]
and by
\begin{align*}
 i\star_\gamma j
 &=
 \sqrt{\frac{ab}{c+\gamma^2d}}\,(k+\gamma\ell),
&
 i\star_\gamma k
 &=
 -\sqrt{\frac{ac}{b+\gamma^2d}}\,(j+\gamma\ell),
\\
 j\star_\gamma k
 &=
 \sqrt{\frac{bc}{a+\gamma^2d}}\,(i+\gamma\ell),
&
 i\star_\gamma\ell
 &=
 -\sqrt{\frac{ad}{b+c}}\,(j+k),
\\
 j\star_\gamma\ell
 &=
 -\sqrt{\frac{bd}{a+c}}\,(k+i),
&
 k\star_\gamma\ell
 &=
 -\sqrt{\frac{cd}{a+b}}\,(i+j).
\end{align*}
The products in the reverse order are determined by
requiring distinct pure basis elements to anti-commute.

The $\R$-basis $1,i,j,k,\ell$ is called the
\emph{canonical basis} of $\mathcal A_\gamma(a,b,c,d)$.

When $\gamma=1$, we write $\Pent(a,b,c,d)=\mathcal A_1(a,b,c,d)$ and call the algebra  $\Pent(a,b,c,d)$  a \emph{generalised
Pentarnion algebra}; its multiplication table is given below. At $\gamma=0$ the sub vector space of $\mathcal{A}_0(a,b,c,d)$ generated by
$1,i,j,k$  is a subalgebra of  $\mathcal{A}_0(a,b,c,d)$, and associative if and only if  $c\sqrt{ab/c} = b\sqrt{ac/b} = a\sqrt{bc/a}$. In that case, it is isomorphic to the quaternion algebra $(-a,-b)_\R$.

\begin{table}[H]
\centering
\small
\setlength{\tabcolsep}{3pt}
\renewcommand{\arraystretch}{1.8}
\resizebox{\textwidth}{!}{$
\begin{array}{c|ccccc}
\star_\gamma & 1 & i & j & k & \ell \\ \hline
1 & 1 & i & j & k & \ell \\
i & i & -a & \sqrt{\frac{ab}{c+\gamma^2d}}(k+\gamma\ell)
 & -\sqrt{\frac{ac}{b+\gamma^2d}}(j+\gamma\ell)
 & -\sqrt{\frac{ad}{b+c}}(j+k) \\
j & j & -\sqrt{\frac{ab}{c+\gamma^2d}}(k+\gamma\ell)
 & -b & \sqrt{\frac{bc}{a+\gamma^2d}}(i+\gamma\ell)
 & -\sqrt{\frac{bd}{a+c}}(k+i) \\
k & k & \sqrt{\frac{ac}{b+\gamma^2d}}(j+\gamma\ell)
 & -\sqrt{\frac{bc}{a+\gamma^2d}}(i+\gamma\ell)
 & -c & -\sqrt{\frac{cd}{a+b}}(i+j) \\
\ell & \ell & \sqrt{\frac{ad}{b+c}}(j+k)
 & \sqrt{\frac{bd}{a+c}}(k+i)
 & \sqrt{\frac{cd}{a+b}}(i+j) & -d
\end{array}$}
\caption{Multiplication table of $\mathcal A_\gamma(a,b,c,d)$.}
\label{tab:switchalgebra}
\end{table}

 Note also that the definition is chosen so that, for example,
$N(i\star_\gamma j)=N(i)N(j)=ab$.

When $\gamma$ is clear from the context we will denote multiplication in $\mathcal{A}_\gamma(a,b,c,d)$ also simply by juxtaposition.

We will be particularly interested in the case that  $a=b=c=d=1$. Here, the algebra $\mathcal{A}_\gamma=\mathcal{A}_\gamma(1,1,1,1)$ is defined for every $\gamma\in \R$ such that $1+\gamma^2\neq 0$.
At $\gamma=0$ the subalgebra of $\mathcal{A}_0$ generated by
$1,i,j,k$ is the quaternion algebra $(-1,-1)_\R$ with canonical basis $1,i,j,k$.

\begin{table}[H]
\centering
\renewcommand{\arraystretch}{1.35}
$
\begin{array}{c|ccccc}
\star_1 & 1 & i & j & k & \ell \\
\hline
1 & 1 & i & j & k & \ell \\
i & i & -a & \sqrt{\frac{ab}{c+d} } (k+\ell)& -\sqrt{\frac{ac}{b+d}}  (j+\ell) & -\sqrt{\frac{ad}{b+c}}  (j+k) \\
j & j & -\sqrt{\frac{ab}{c+d}}(k+\ell) & -b & \sqrt{\frac{bc}{a+d}}(i+\ell) & -\sqrt{\frac{bd}{a+c} }(k+i)\\
k & k & \sqrt{\frac{ac}{b+d}}(j+\ell) & -\sqrt{\frac{bc}{a+d}} (i+\ell) & -c & -\sqrt{\frac{cd}{a+b}} (i+j)\\
\ell & \ell & \sqrt{\frac{ad}{b+c}}(j+k) & \sqrt{\frac{bd}{a+c}}(k+i) & \sqrt{ \frac{cd}{a+b}}( i+j) & -d
\end{array}
$
\caption{Multiplication table of $\Pent(a,b,c,d)$.}
\label{tab:multiplicationPent}
\end{table}

Every $p\in \mathcal{A}_\gamma(a,b,c,d)$   can be written as
$$p=p_0+p_1i+p_2j+p_3k+p_4\ell,$$
for $p_0,p_1,p_2,p_3,p_4\in \R.$
We call  $p_0$  the \emph{scalar part}, and $ \Imag(p)=p_1i+p_2j+p_3k+p_4\ell
$ the \emph{pure part} of $p$. It is a tedious but straightforward computation to see that
\[
  p=p_0+p_1i+p_2j+p_3k+p_4\ell \mapsto  \overline{p}=p_0-p_1i-p_2j-p_3k-p_4\ell,
\]
is a scalar involution on $\mathcal{A}_\gamma(a,b,c,d)$ and that
$$p\overline{p}=p_0^2+ap_1^2+bp_2^2+cp_3^2+dp_4^2.$$
Since $\mathcal{A}_\gamma(a,b,c,d)$  has a scalar involution, there exists a quadratic form  $N_{\mathcal{A}_\gamma(a,b,c,d)}:\mathcal{A}_\gamma(a,b,c,d) \to \R$,
\[
  N_{\mathcal{A}_\gamma(a,b,c,d)}(p)=p\overline{p}=p_0^2+ap_1^2+bp_2^2+cp_3^2+dp_4^2,
\]
 the  \emph{norm} of $\mathcal{A}_\gamma(a,b,c,d)$, and a linear form  $\Tr_{\mathcal{A}_\gamma(a,b,c,d)}:\mathcal{A}_\gamma(a,b,c,d) \to F$,
\[
  \Tr_{\mathcal{A}_\gamma(a,b,c,d)}(p)=p+\overline{p}=2p_0.
\]
 the  \emph{trace}  of $\mathcal{A}_\gamma(a,b,c,d)$  \cite[p.~86]{McCrimmon1985}. We note that
 $$N_{\mathcal{A}_\gamma(a,b,c,d)}\cong\langle 1,a,b,c,d\rangle$$  and that
$$N_{\mathcal{A}_\gamma(a,b,c,d)}(p)=p\overline{p}=\overline{p}p.$$
We conclude that $\mathcal{A}_\gamma(a,b,c,d)$ is a unital quadratic algebra; every $p\in \mathcal{A}_\gamma(a,b,c,d)$ satisfies the quadratic equation
$$ p^2-\Tr(p)p+N(p)1=0.$$
 This also means that $\mathcal{A}_\gamma(a,b,c,d)$ is power associative.

  The form $N_{\mathcal{A}_\gamma(a,b,c,d)}$ is not multiplicative (if it were we would have found a 5-dimensional algebra with a nondegenerate multiplicative norm, a contradiction) but still satisfies
 $$N_{\mathcal{A}_\gamma(a,b,c,d)}(p^2)=N_{\mathcal{A}_\gamma(a,b,c,d)}(p)^2.$$

The algebra $\mathcal{A}_\gamma(a,b,c,d)$ is not flexible, since
\[
 [i+j,\ell,i+j]
 =2\left(a\sqrt{\frac{bd}{a+c}}+
          b\sqrt{\frac{ad}{b+c}}\right)1\neq0.
\]
Indeed, if $a=b$ the coefficient equals $4a\sqrt{ad/(a+c)}\neq 0$. If $a\neq b$ and the coefficient
vanished, squaring would yield $a^2\frac{bd}{a+c}=b^2\frac{ad}{b+c}$, i.e.\ $abcd(a-b)=0$,
a contradiction.

Since every alternative algebra is flexible,
$\mathcal{A}_\gamma(a,b,c,d)$ is not alternative either.

\begin{lemma}\label{le;easy}
 Let ${\rm Cay}(\R,a)$ denote the Cayley-Dickson doubling of $\R$ employing $a\in \R$.
 \\
(i) Every subalgebra of $\mathcal{A}_\gamma(a,b,c,d)$  generated by 1 and some $p\in \mathcal{A}_\gamma(a,b,c,d)\setminus \R1$ is either the split \'etale algebra $\R\times \R$, a quadratic field extension of $\R$, or isomorphic to the ``generalised dual numbers'' $\R[x]/(x^2)={\rm Cay}(\R,0)$.
\\ (ii) If the norm $N_{\mathcal{A}_\gamma(a,b,c,d)}$ of $\mathcal{A}_\gamma(a,b,c,d)$  is anisotropic, then every subalgebra generated by 1 and some $p\in \mathcal{A}_\gamma(a,b,c,d)\setminus \R1$ is a quadratic field extension of $\R$.
\\ (iii) We have $\mathcal{A}_\gamma(a,b,c,d)\cong \mathcal{A}_\gamma(s^2a,s^2b,s^2c,s^2d)$ for all $s\in \R^\times$.
\end{lemma}

\begin{proof}
$(i)$ Since $\mathcal{A}_\gamma(a,b,c,d)$ is a quadratic algebra, every subalgebra $\R[p]$ generated by 1 and some $p\in \mathcal{A}_\gamma(a,b,c,d)\setminus \R 1$ is a unital associative subalgebra of dimension 2 and these algebras are classified: $\R[p]$ is either the split algebra $\R\times \R={\rm Cay}(\R,1)$, a quadratic field extension of $\R$, or isomorphic to the ``generalised dual numbers'' $\R[x]/(x^2)={\rm Cay}(\R,0)$.
\\ $(ii)$ Let $N_{\mathcal{A}_\gamma(a,b,c,d)}$ be anisotropic. The subalgebras $A_0$ given by $\R\times \R$ and $\R[x]/(x^2)={\rm Cay}(\R,0)$ have an isotropic  norm $N_0$. Since norms are unique we have $N_0(s)=N_{\mathcal{A}_\gamma(a,b,c,d)}(s)$ for all $s\in A_0$, a contradiction, since $N_{\mathcal{A}_\gamma(a,b,c,d)}$ is anisotropic.
\\ $(iii)$ Put $r=|s|$ and let $1,i',j',k',\ell'$ be the canonical
basis of $\mathcal{A}_\gamma(s^2a,s^2b,s^2c,s^2d)$.
The $\R$-algebra isomorphism is given by
$ 1\mapsto1,$ $ i'\mapsto ri $$ j'\mapsto rj,$ $
 k'\mapsto rk, $ $\ell'\mapsto r\ell$.
\end{proof}

Since the norm $N_{\mathcal{A}_\gamma(a,b,c,d)}$  has dimension 5, $N_{\mathcal{A}_\gamma(a,b,c,d)}$ is either anisotropic and $N\cong  \langle 1,1,1,1,1\rangle$, $N\cong  \langle 1,-1, 1,1,1\rangle$ or $N\cong  \langle 1,-1, -1,-1,-1\rangle$, $N\cong  \langle 1,-1,1,-1,1\rangle$ or $N\cong  \langle 1,-1,1,-1,-1\rangle$.

\begin{proposition}
 Let $\gamma,\delta\in \R$ and suppose that the number of positive elements in the set $\{a,b,c,d\}$ is not the same as the number of positive elements in  $\{a',b',c',d'\}$. Then $\mathcal{A}_\gamma(a,b,c,d)$ is not isomorphic to $\mathcal{A}_\delta(a',b',c',d')$.
\end{proposition}

\begin{proof}
 Nondegenerate quadratic forms with different signatures cannot be isomorphic. This yields the assertion immediately.
\end{proof}

Since $\mathcal{A}_\gamma(a,b,c,d)$ is a quadratic algebra, we have
 $\mathcal{A}_\gamma(a,b,c,d)=\R 1\oplus V_\gamma(a,b,c,d)$ with the pure part
\[
V_\gamma(a,b,c,d)
 =\{u\in\mathcal{A}_\gamma(a,b,c,d)\mid \Tr(u)=0\}
 =\Span_\R\{i,j,k,\ell\}.
\]
Define
\[
\begin{aligned}
\times_\gamma:\;&V_\gamma(a,b,c,d)\times V_\gamma(a,b,c,d)
   \longrightarrow V_\gamma(a,b,c,d),\\
&u\times_\gamma v=\operatorname{pr}(uv),
\end{aligned}
\]
where $\operatorname{pr}$ is the canonical projection from
$\mathcal{A}_\gamma(a,b,c,d)$ to $V_\gamma(a,b,c,d)$.  Then
$(V_\gamma(a,b,c,d),\times_\gamma)$ is a four-dimensional anti-commutative algebra over $F$, and
\[
uv=-\frac{1}{2}N_{\mathcal{A}_\gamma(a,b,c,d)}(u,v)+u\times_\gamma v.
\]
for all $u,v\in V_\gamma(a,b,c,d)$.
Here,  $\frac{1}{2}N_{\mathcal{A}_\gamma(a,b,c,d)}(u,v)$ denotes the symmetric bilinear form  associated to $N_{\mathcal{A}_\gamma(a,b,c,d)}$, and
  $\times_\gamma $ is anti-symmetric. By Section \ref{subsec:quadr}, the multiplication of $\mathcal{A}_\gamma(a,b,c,d)=\R\oplus V_\gamma(a,b,c,d)$ is given by
 $$(\alpha,u)(\beta,v)=(\alpha\beta-\frac{1}{2}N_{\mathcal{A}_\gamma(a,b,c,d)}(u,v),\alpha v+\beta u+u\times_\gamma v) $$
 for all $\alpha,\beta\in \R$ and $u,v\in V_\gamma(a,b,c,d)$.

  For pure elements $u,v\in V_\gamma(a,b,c,d)$, define the anti-symmetric map
\[
\Pi_\gamma:\Lambda^2V_\gamma(a,b,c,d)\longrightarrow V_\gamma(a,b,c,d), \quad u\wedge v \mapsto u\times_\gamma v.
\]
Note that
\[
  uv+vu=-N_{\mathcal{A}_\gamma(a,b,c,d)}(u,v) \text{ and }   uv-vu=2\Pi_\gamma(u\wedge v).
\]
For ease of notation, we write $\Pi(u\wedge v)=u\times v$ and $\times$ from now on, although $\Pi$ and $\times$ clearly depend on the choice of $a,b,c,d$ and $\gamma$. If we need to make clear which $\gamma$ is involved (usually we will work with fixed $a,b,c,d$ and vary $\gamma$), we will write $\Pi_\gamma$ and $\times_\gamma$. We also write $N=N_{\mathcal{A}_\gamma(a,b,c,d)}$ when the context is clear.

\begin{remark}
For $\gamma=1$, the pure product is given by
\[
\begin{aligned}
i\times j
 &=
 \sqrt{\frac{ab}{c+d}}\,(k+\ell),
&
 j\times k
 &=
 \sqrt{\frac{bc}{a+d}}\,(i+\ell),
&
 k\times i
 &=
 \sqrt{\frac{ac}{b+d}}\,(j+\ell),
\\[0.4em]
 \ell\times i
 &=
 \sqrt{\frac{ad}{b+c}}\,(j+k),
&
 \ell\times j
 &=
 \sqrt{\frac{bd}{a+c}}\,(k+i),
&
 \ell\times k
 &=
 \sqrt{\frac{cd}{a+b}}\,(i+j).
\end{aligned}
\]
The canonical Pentarnion algebra $\mathbb P$ used in
\cite{PoshtkohiPumplunSubmitted} is obtained by setting
$a=b=c=d=1$.
\end{remark}

 \begin{remark}\label{re:3.4}
 Define the symmetrised product $p\circ r=\frac{pr+rp}{2}$
on $\mathcal{A}_\gamma(a,b,c,d)$, then
\[
(a,u)\circ(b,v)
=(ab-\frac{1}{2} N(u,v),a v+b u).
\]
This is the five-dimensional spin-factor-type Jordan product, hence $(\mathcal{A}_\gamma(a,b,c,d), \circ)$ is a Jordan algebra and $\mathcal{A}_\gamma(a,b,c,d)$ is \emph{Jordan-admissible} (i.e., its symmetrisation is a Jordan algebra). Over $\R$, noncommutative Jordan algebras are flexible Jordan-admissible algebras. It is easy to see that the algebra $\mathcal{A}_\gamma(a,b,c,d)$ is not flexible.
Therefore $\mathcal{A}_\gamma(a,b,c,d)$ is not a (noncommutative) Jordan algebra.
\end{remark}

 The nuclei and automorphism groups of the   family $\mathcal{A}_\gamma(a,b,c,d)$  depend on the
coefficients $a,b,c,d$.  They are not needed in the present paper, so we omit them here to keep the paper within reasonable length.  For the family $\mathcal{A}_\gamma=\mathcal{A}_\gamma(1,1,1,1)$ they are treated in \cref{sec:switching-algebra}.

 \begin{example} The algebras $\Pent(-1,-1,-1,2)$, $\Pent(-4,-4,-4,8)$ and $\Pent(-7,-7,8,-7)$ all have isotropic norms isometric to $\langle 1,-1,-1,-1,1 \rangle$, and are hence not isomorphic to $\Pent$ and $\Pent(1,1,1,8)$. However, we have $\Pent(-1,-1,-1,2)\cong  \Pent(-4,-4,-4,8)$ (Lemma \ref{le;easy}$(iii)$)
 \end{example}

 \begin{example}
The multiplication table for $\Pent(1,1,1,8)$ is given by
\begin{table}[H]
\centering
\renewcommand{\arraystretch}{1.35}
$
\begin{array}{c|ccccc}
\cdot & 1 & i & j & k & \ell \\
\hline
1 & 1 & i & j & k & \ell \\
i & i & -1 & \frac{1}{3}  (k+\ell)& - \frac{1}{3} (j+\ell) & -2  (j+k) \\
j & j & -\frac{1}{3} (k+\ell) & -1 & \frac{1}{3} (i+\ell) & -2(k+i)\\
k & k & \frac{1}{3} (j+\ell) & -\frac{1}{3}  (i+\ell) & -1 & -2(i+j)\\
\ell & \ell & 2(j+k) &2(k+i) & 2( i+j) & -8
\end{array}
$
\caption{Multiplication table $\Pent(1,1,1,8)$, $F=\R$.}
\label{tab:pentarnion-1118}
\end{table}
and the norm is isometric to $\langle 1,1,1,1,1 \rangle_\R$. By Lemma \ref{le;easy}$(iii)$, $\Pent(4,4,4,32)\cong \Pent(1,1,1,8)$,
and the multiplication table for $\Pent(4,4,4,32)$ is given by
\begin{table}[H]
\centering
\renewcommand{\arraystretch}{1.35}
$
\begin{array}{c|ccccc}
\cdot & 1 & i & j & k & \ell \\
\hline
1 & 1 & i & j & k & \ell \\
i & i & -4 & \frac{2}{3}  (k+\ell)& - \frac{2}{3} (j+\ell) & -4  (j+k) \\
j & j & -\frac{2}{3} (k+\ell) & -4 & \frac{2}{3} (i+\ell) & -4(k+i)\\
k & k & \frac{2}{3} (j+\ell) & -\frac{2}{3}  (i+\ell) & -4 & -4(i+j)\\
\ell & \ell & 4(j+k) &4(k+i) & 4(i+j) & -32
\end{array}
$
\caption{Multiplication table $\Pent(4,4,4,32)$, $F=\R$.}
\label{tab:multiplication}
\end{table}
 The fact that the norm of  $\Pent$ is the same  does not decide whether the algebra is isomorphic to $\Pent$; isometry of quadratic norms is necessary but not sufficient for algebra isomorphism.
\end{example}

\section{Important properties of $\mathcal A_\gamma$ over $\R$}
\label{sec:switching-algebra}

\subsection{The algebraic structure of $\mathcal A_\gamma$}\label{sec:4.1}

We will study $\mathcal{A}_\gamma=\mathcal{A}_\gamma(1,1,1,1)$, and denote $V=V_\gamma={\rm span}_\R\{i,j,k,\ell\}$, $\langle u,v\rangle=\frac{1}{2}N(u,v)$ and  write $u\times v=u\times_\gamma v={\rm pr}(uv)$ when the context is clear.

For $0\leq\gamma\leq1$, set
\[
 \kappa=\kg=\frac{1}{\sqrt{1+\gamma^2}},
 \quad
 \sigma=\sg=\frac{\gamma}{\sqrt{1+\gamma^2}}
\]
and note that $\kg^2+\sg^2=1$.
Since $(\kg,\sg)$ lies on the unit circle, the angle
\begin{equation}
 \theta_\gamma=\arctan\gamma,
 \qquad
 \kg=\cos\theta_\gamma,
 \qquad
 \sg=\sin\theta_\gamma
 \label{eq:switch-angle}
\end{equation}
parametrises the structure coefficients of the complementary-plane products.  As $\gamma$
varies from $0$ to $1$, the vector
\[
 i\star_\gamma j=\cos\theta_\gamma\,k+\sin\theta_\gamma\,\ell
\]
traces the unit-circle arc in $\Span\{k,\ell\}$ from $k$ to
$(k+\ell)/\sqrt2$.  The products $j\star_\gamma k$ and $k\star_\gamma i$
trace the corresponding arcs in $\Span\{i,\ell\}$ and $\Span\{j,\ell\}$.
 Note that at $\gamma=0$, $ \mathbb H=\R1\oplus\R i\oplus\R j\oplus\R k$
is a quaternion subalgebra of $\SwitchAlg_0$.

 Let $F^\gamma_u:V\to V,$ $F^\gamma_u(v)=u\times_\gamma v$. We will simply  write  $F_u(v)=F^\gamma_u(v)$ when the context is clear.
 For different $\gamma$, the algebras $\mathcal{A}_\gamma$ have the same trace $\rm Tr(\alpha+u)=2\alpha$ and the same norm $N(\alpha+u)=\alpha^2+\ip uu$, only $\times_\gamma$ varies.

Set $E_0={\rm span}\{i,j,k\}$ and $\omega=i+j+k$, then we have
$$P=\{xi+yj+zk\in E_0\,|\,x+y+z=0\}=E_0\cap\omega^\perp,$$
and it is easy to check that $P$ has the basis $\xi=i-j$, $\zeta=i+j-2k $.
Then $\omega,\xi,\zeta,\ell$ is an orthogonal basis and we have $V=\R\omega\perp P\perp \R\ell$, and $\langle \omega,\omega\rangle=3$, $\langle\xi,\xi\rangle=2$, $\langle\zeta,\zeta\rangle=6$, $\langle\ell,\ell\rangle=1$,
\begin{equation}\label{eq:adapted}
\begin{aligned}
&\omega\times\xi=\kappa\zeta,\quad \omega\times\zeta=-3\kappa\xi,\quad \xi\times\zeta=2\kappa\omega+6\sigma\ell,\\
&\ell\times\omega=\sqrt2\,\omega,\quad \ell\times\xi=-\tfrac1{\sqrt2}\,\xi,\quad \ell\times\zeta=-\tfrac1{\sqrt2}\,\zeta .
\end{aligned}
\end{equation}

\begin{lemma}\label{lem:red}
 (i)
 Every $\R$-algebra isomorphism $\Psi:\mathcal{A}_\gamma\to\mathcal{A}_\delta$ satisfies $\Psi|_V\in O(V,\langle {\,},{\,}\rangle)$ and
\[\Psi(u\times_\gamma v)=\Psi(u)\times_\delta\Psi(v)\]
for all $u,v\in V$.
Conversely, if $\Psi\in O(V,\langle {\,},{\,}\rangle)$ satisfies the above identity for all $u,v\in V$, then $\mathcal{A}_\gamma\to\mathcal{A}_\delta$, $\alpha+u\mapsto\alpha+\Psi(u)$, is an isomorphism.
\\ (ii) If $\Psi\in{\rm GL}_F(V)$ satisfies $\Psi(u\times_\gamma v)=\Psi(u)\times_\delta\Psi(v)$ for all $u,v\in V$, then $F^{\delta}_{\Psi(u)}=\Psi\circ F^{\gamma}_u\circ\Psi^{-1}$ for all $u\in V$.
\end{lemma}

\begin{proof}
 $(i)$ We have $\Psi(1)=1$ by definition. Both algebras are quadratic with the same trace and norm. For all $x\in\mathcal{A}_\gamma$ we have $x^2-\Tr(x)x+N(x)1=0$. Applying $\Psi$ gives $\Psi(x)^2-\Tr(x)\Psi(x)+N(x)1=0$ in $\mathcal{A}_\delta$ with $\Psi(x)\notin \R$ whenever  $x\in\mathcal{A}_\gamma\setminus \R$. Since the trace and norm  of a quadratic algebra are unique, this immediately implies that $\Tr(\Psi (x))=\Tr(x)$ and $N(\Psi (x))=N(x)$. Hence $\Psi(V)=V$ for $V={\rm ker}(\Tr)$, and $\Psi|_V$ preserves $\langle u, u\rangle=N(u)$, hence $\langle \,,\,\rangle$. For $u,v\in V$, the multiplication on $\mathcal{A}_\gamma$ and $\mathcal{A}_\delta$ gives $\Psi(uv)=-\langle u,v\rangle+\Psi(u\times_\gamma v)$ and $\Psi(u)\Psi(v)=-\langle\Psi (u),\Psi (v)\rangle+\Psi(u)\times_\delta\Psi(v)$. Comparing pure parts gives the desired equality.

Conversely, applying the bijective $F$-linear map $\alpha+u\mapsto\alpha+\Psi(u)$ to both sides of the multiplication
$$(\alpha+u)(\beta+v)=\alpha\beta-\langle u,v\rangle+\alpha v+\beta u+u\times v $$
 in $\mathcal A_\gamma$
shows that the map is multiplicative as soon as $\Psi$ preserves $\langle \,,\,\rangle$.
\\ $(ii)$  $F^\delta_{\Psi(u)}(v)=\Psi(u)\times_\delta\Psi(\Psi^{-1}(v))=\Psi(u\times_\gamma\Psi^{-1}(v))=(\Psi F^\gamma_u\Psi^{-1})(v)$.
\end{proof}

\begin{lemma}\label{lem:sa}
Let $u\in V$. Then $F_u$  is self-adjoint with respect to $\langle \,,\,\rangle$ if and only if $u\in \R\ell$.
\end{lemma}

\begin{proof}
Write $u=x\omega+y\xi+z\zeta+w\ell$. Then we have
\begin{align*}
u\times\omega&=\phantom{-}\sqrt2\,w\,\omega+3\kappa z\,\xi-\kappa y\,\zeta,\\
u\times\xi&=-2\kappa z\,\omega-\tfrac{w}{\sqrt2}\,\xi+\kappa x\,\zeta-6\sigma z\,\ell,\\
u\times\zeta&=\phantom{-}2\kappa y\,\omega-3\kappa x\,\xi-\tfrac{w}{\sqrt 2}\,\zeta+6\sigma y\,\ell .
\end{align*}
Using $\langle\omega,\omega\rangle=3$, $\langle\xi,\xi\rangle=2$, $\langle\zeta,\zeta\rangle=6$ and the orthogonality of $\omega,\xi,\zeta,\ell$,
we obtain
\[
\begin{aligned}
\langle F_u(\xi),\zeta\rangle-\langle\xi,F_u(\zeta)\rangle&=12\kappa x,\\
\langle F_u(\zeta),\omega\rangle-\langle\zeta,F_u(\omega)\rangle&=12\kappa y,\\
\langle F_u(\omega),\xi\rangle-\langle\omega,F_u(\xi)\rangle&=12\kappa z .
\end{aligned}
\]
If $F_u$ is self-adjoint, we get $x=y=z=0$, i.e. $u=w\ell$.

Conversely, it is a straightforward computation that $F_\ell$ is diagonal in the orthogonal basis $\omega,\xi,\zeta,\ell$, hence self-adjoint.
\end{proof}

\begin{lemma}\label{lem:beta}
Define the symmetric bilinear form ${\rm tr}:V\times V\to \R$, ${\rm tr}(u,v)=\Tr(F_uF_v)$. Then
${\rm tr}(u,v)=-2\kappa^2\langle u,v \rangle $ for all $u,v\in E_0$, ${\rm tr}(E_0,\ell)=0$, and ${\rm tr}(\ell,\ell)=3$.
This is equivalent to ${\rm tr}(u,v)=\langle Bu, v\rangle$ with $$B=B_\gamma=-2\kappa^2\mathrm{id}_{E_0}\oplus 3 \mathrm{id}_{\R\ell}.$$
In particular we have
\begin{equation}\label{eq:trB}
\Tr (B_\gamma )=3-6\kappa^2=3-\frac{6}{1+\gamma^2}.
\end{equation}
\end{lemma}

\begin{proof}
We compute
\[F_i(i)=0,\quad F_i(j)=\kappa k+\sigma\ell,\quad F_i(k)=-\kappa j-\sigma\ell,\quad F_i(\ell)=-\frac{1}{\sqrt2}(j+k),\]
\[F_j(i)=-\kappa k-\sigma\ell,\quad F_j(j)=0,\quad F_j(k)=\kappa i+\sigma\ell,\quad F_j(\ell)=-\frac {1}{\sqrt2}(k+i),\]
\[F_\ell(i)=\frac{1}{\sqrt2}(j+k),\quad F_\ell(j)=\frac{1}{\sqrt2}(k+i),\quad F_\ell(k)=\frac{1}{\sqrt2}(i+j),\quad F_\ell(\ell)=0 .\]
Hence
\begin{align*}
F_i^2(i)&=0, & F_i^2(j)&=-\kappa^2 j-\kappa\sigma\ell-\frac{\sigma}{\sqrt{2}}(j+k),\\
F_i^2(k)&=-\kappa^2k-\kappa\sigma\ell+\frac{\sigma}{\sqrt{2}}(j+k), & F_i^2(\ell)&=\frac{\kappa}{\sqrt{2}}(j-k);\\
F_iF_j(i)&=\kappa^2 j+\kappa\sigma\ell+\tfrac{\sigma}{\sqrt{2}}(j+k), & F_iF_j(j)&=0,\\
F_iF_j(k)&=-\tfrac{\sigma}{\sqrt{2}}(j+k), & F_iF_j(\ell)&=\tfrac1{\sqrt2}(\kappa j+\sigma\ell);\\
F_iF_\ell(i)&=\frac{\kappa}{\sqrt{2}}(k-j), & F_iF_\ell(j)&=-\frac{1}{\sqrt{2}}(\kappa j+\sigma\ell),\\
F_iF_\ell(k)&=\frac{1}{\sqrt{2}}(\kappa k+\sigma\ell), & F_iF_\ell(\ell)&=0;\\
F_\ell^2(i)&=i+\frac{1}{2}(j+k), & F_\ell^2(j)&=j+\frac{1}{2}(k+i),\\
F_\ell^2(k)&=k+\frac{1}{2}(i+j), & F_\ell^2(\ell)&=0 .
\end{align*}
This implies that
\[\Tr (F_i^2)=(-\kappa^2-\tfrac{\sigma}{\sqrt{2}})+(-\kappa^2+\frac{\sigma}{\sqrt{2}})=-2\kappa^2,\quad
\Tr( F_iF_j)=-\frac{\sigma}{\sqrt{2}}+\frac{\sigma}{\sqrt{2}}=0,\]
\[\Tr (F_iF_\ell)=-\tfrac{\kappa}{\sqrt2}+\frac{\kappa}{\sqrt{2}}=0,\quad \Tr( F_\ell^2)=3 .\]
The cyclic permutation $\tau\colon i\mapsto j\mapsto k\mapsto i$, $\ell\mapsto\ell$ satisfies $\langle\, , \,\rangle$, hence is an automorphism of $\mathcal{A}_\gamma$ by Lemma~\ref{lem:red}$(i)$.  It preserves ${\rm tr}$ (Lemma~\ref{lem:red}$(ii)$) and satisfies $\tau(u\times v)=\tau(u)\times \tau(v)$, so that $\tau\in \Aut (\mathcal{A}_\gamma)$ by Lemma \ref{lem:red}$(i)$. This gives ${\rm tr}(j,j)={\rm tr}(k,k)=-2\kappa^2$, ${\rm tr}(j,k)={\rm tr}(k,i)=0$ and ${\rm tr}(j,\ell)={\rm tr}(k,\ell)=0$. Finally, this implies $\Tr ( B_\gamma)=\sum_{e\in\{i,j,k,\ell\}} \Tr(F_e^2)=3\cdot(-2\kappa^2)+3=3-6\kappa^2=3-\frac{6}{1+\gamma^2}.$
\end{proof}

\begin{theorem}
\label{thm:pairwise-nonisomorphic}
For all $\gamma,\delta\geq 0$, we have $\SwitchAlg_{\gamma}\cong \SwitchAlg_{\delta}$ if and only if $\gamma=\delta$.
In particular, we have  $\Tr( B_\gamma)=\Tr( B_\delta)$ whenever $\SwitchAlg_{\gamma}\cong \SwitchAlg_{\delta}$.
\end{theorem}

\begin{proof}
Let $\Psi:\SwitchAlg_\gamma\to \SwitchAlg_\delta$ be an $\R$-algebra isomorphism.
By Lemma~\ref{lem:red}, $\Psi|_V$ is an isometry of $V$ with $\Psi(u\times_\gamma v)=\Psi( u)\times_\delta\Psi (v)$, and  $F^\delta_{\Psi (u)}=\Psi F^\gamma_u\Psi^{-1}$. Hence for the trace forms ${\rm tr}_\gamma(u,v)=\Tr(F^\gamma_uF^\gamma_v)$ and ${\rm tr}_\delta(u,v)=\Tr(F^\delta_uF^\delta_v)$, we have
\[{\rm tr}_\delta(\Psi ( u),\Psi (v))=Tr\big(\Psi F^\gamma_uF^\gamma_v\Psi^{-1}\big)={\rm tr}_\gamma(u,v)\]
for all $u,v\in V$.
With $B_\gamma,B_\delta$ defined as in Lemma~\ref{lem:beta} this is the same as $\langle B_\delta\Psi( u),\Psi (v)\rangle =\langle B_\gamma u, v\rangle=\langle \Psi B_\gamma u,\Psi (v)\rangle$ for all $u,v\in V$, so $B_\delta\Psi=\Psi B_\gamma$. Taking traces and using  \eqref{eq:trB}, $3-6\kappa_\gamma^2=3-6\kappa_\delta^2$; therefore we get $\kappa_\gamma^2=\kappa_\delta^2$, i.e. $1+\gamma^2=1+\delta^2$, i.e. $\gamma^2=\delta^2$ thus $\gamma=\delta$. The converse is trivial.
\end{proof}

\subsection{Scalar--pure decomposition and fixed Jordan symmetrisation.}

Specialising the construction of \cref{sec:switch}, let
$\Vpure=\Span_\R\{i,j,k,\ell\}$
For pure vectors $u,v\in\Vpure$, we have
\[
 u\star_\gamma v=-\ip{u}{v}+u\crossg v,
 \qquad
 u\crossg v=\Pi_\gamma(u\wedge v).
\]
For $p=(a,u)$ and $q=(b,v)$,
\begin{equation}
 (a,u)\star_\gamma(b,v)
 =\bigl(ab-\ip{u}{v},\;av+bu+u\crossg v\bigr).
 \label{eq:scalar-pure-product}
\end{equation}
Because $\crossg$ is alternating,
\begin{equation}
 \frac{p\star_\gamma q+q\star_\gamma p}{2}
 =\bigl(ab-\ip{u}{v},\;av+bu\bigr)
 \label{eq:jordan-independent}
\end{equation}
is independent of $\gamma$.  Thus changing the parameters $\gamma$ changes the ordered, noncommutative part of
multiplication but not the common spin-factor-type Jordan symmetrisation. We call  $\crossg $ the \emph{alternating switching product}.

\begin{theorem}
\label{thm:quadratic-structure}
For every $\gamma\in \R$,
$\SwitchAlg_\gamma$ is a unital quadratic, power-associative, locally complex real algebra with scalar involution
\[
 \overline{p_0+u}=p_0-u
\]
 where $ p\star_\gamma\bar p=\bar p\star_\gamma p
 =\N(p)1,$ $\N(p)=p_0^2+\norm u^2.$
Every element satisfies
\begin{equation}
 p^2-2p_0p+\N(p)1=0.
 \label{eq:quadratic-recurrence}
\end{equation}
If $u\neq0$, $I=u/\norm u$, $\rho=\norm u$ and $p=p_0+\rho I$, then
\[
 \exp(p)=e^{p_0}\left(\cos\rho+I\sin\rho\right),
\]
with the usual complex-slice formulas for sine, cosine, and hyperbolic functions.
\end{theorem}

The proof of the first assertions is a trivial consequence of the results in Sections \ref{sec:switch} and \ref{subsec:quadr}. The formulae  for sine, cosine, and hyperbolic functions were shown in \cite{PoshtkohiPumplunSubmitted} for any quadratic locally complex algebra.

We make the following observation that will be crucial in later sections.

\begin{remark}\label{prop:one-element-invisibility}
Let $p=p_0+p_1i+p_2j+p_3k+p_4\ell$ with coordinates independent of $\gamma$. Since $\mathcal{A}_\gamma $ is quadratic, we have
$p\star_\gamma p=2p_0p-N(p)1$ and  $N$ does not depend on $\gamma$. Hence all powers of $p$,
and thus every real polynomial or convergent real power series in $p$, are independent of $\gamma$.
Thus $\gamma$ is only ``visible'' in products of at least two linearly independent elements. Thus \emph{switching} (i.e., changing $\gamma$, or more precisely, the algebra $\mathcal A_\gamma$) is detected only by interactions among independent elements, nonreal coefficients, or different bracketings.
\end{remark}

\subsection{Associator, invariants and automorphisms}\label{sec:4.3}

We introduce the orthonormal \emph{adapted basis} of $V=V_\gamma$ as
\begin{equation}
 \textbf a=\frac{i+j+k}{\sqrt3},
 \qquad
 \textbf b=\frac{i-j}{\sqrt2},
 \qquad
 \textbf c=\frac{i+j-2k}{\sqrt6},
 \qquad
 \textbf d=\ell.
 \label{eq:adapted-basis}
\end{equation}
Then
\begin{equation}
 \textbf a \textbf b=\kg  \textbf c,
 \quad \textbf a \textbf c=-\kg \textbf b,
 \quad \textbf a \textbf d=-\sqrt2\mathbf a,
 \quad \textbf b\textbf c=\kg \mathbf a+\sqrt3\sg \textbf d,
 \quad \textbf b\textbf d=\frac{\textbf b}{\sqrt2},
 \quad \textbf c\textbf d=\frac{\textbf c}{\sqrt2},
 \label{eq:adapted-products-main}
\end{equation}
where all products are in $\SwitchAlg_\gamma$ and reversing the order changes the sign.

Let $[x,y,z]_\gamma$ denote the associator in $\mathcal A_\gamma$,
and let $\|\cdot\|$ be the Euclidean norm on $\mathcal A_\gamma$.
The \emph{associator tensor} is the trilinear map
$\operatorname{Assoc}_\gamma:V\times V\times V\to\mathcal A_\gamma$,
$(x,y,z)\mapsto[x,y,z]_\gamma$. Define its squared Frobenius norm and that of its restriction by
\[
\begin{aligned}
 \left\|\operatorname{Assoc}_\gamma|_{\Span\{i,j,k\}}\right\|_{\Frob}^2
 &=\sum_{x,y,z\in\{i,j,k\}}\|[x,y,z]_\gamma\|^2,\\
\norm{\operatorname{Assoc}_\gamma}_{\Frob}^2
 &=\sum_{x,y,z\in\{i,j,k,\ell\}}\|[x,y,z]_\gamma\|^2.
\end{aligned}
\]

\begin{lemma}
\label{prop:core-associator}
We have
\begin{equation}
 [i,j,k]_\gamma
 =\frac{\gamma}{\sqrt{2(1+\gamma^2)}}(i+2j+k),
 \label{eq:core-associator}
\end{equation}
\[
 \norm{[i,j,k]_\gamma}^2=\frac{3\gamma^2}{1+\gamma^2},
\]
\begin{equation}
 \left\|\operatorname{Assoc}_\gamma|_{\Span\{i,j,k\}}\right\|_{\Frob}^2
 =\frac{42\gamma^2}{1+\gamma^2}
 =42\sin^2\theta_\gamma.
 \label{eq:assoc-core-norm}
\end{equation}
and
\begin{equation}
 \norm{\operatorname{Assoc}_\gamma}_{\Frob}^2
 =\frac{3(41\gamma^2+31)}{1+\gamma^2}.
 \label{eq:assoc-full-norm}
\end{equation}
\end{lemma}

\begin{proof}
We have
\[
 (i\star_\gamma j)\star_\gamma k
 =-\kg+\frac{\sg}{\sqrt2}(i+j) \text{ and } i\star_\gamma(j\star_\gamma k)
 =-\kg-\frac{\sg}{\sqrt2}(j+k).
\]
Subtraction gives \cref{eq:core-associator}.
The remaining two equations follow from a direct tedious computation
of the different associators using the multiplication table
(i.e. of the 27, respectively 64, associators); the symbolic certificate
is supplied in the accompanying code
(Appendix~\ref{app:reproducibility}).
\end{proof}

\begin{remark}
Every isomorphism $\Psi:\mathcal A_\gamma\to\mathcal A_\delta$ fixes $1$, is orthogonal on $V$ (Lemma \ref{lem:red}), and satisfies $\Psi([x,y,z]_\gamma)=[\Psi x,\Psi y,\Psi z]_\delta$.Orthogonality on the input and output spaces therefore preserves the squared Frobenius norm of the associator tensor. Since $3(41\gamma^2+31)/(1+\gamma^2)=3\bigl(41-\tfrac{10}{1+\gamma^2}\bigr)$ is strictly increasing for $\gamma\ge 0$, different $\gamma$ give non-isomorphic algebras. This gives an alternative proof for Theorem \ref{thm:pairwise-nonisomorphic}.
\end{remark}

\begin{figure}[t]
\centering
\begin{subfigure}[t]{0.48\textwidth}
\centering
\includegraphics[width=\textwidth]{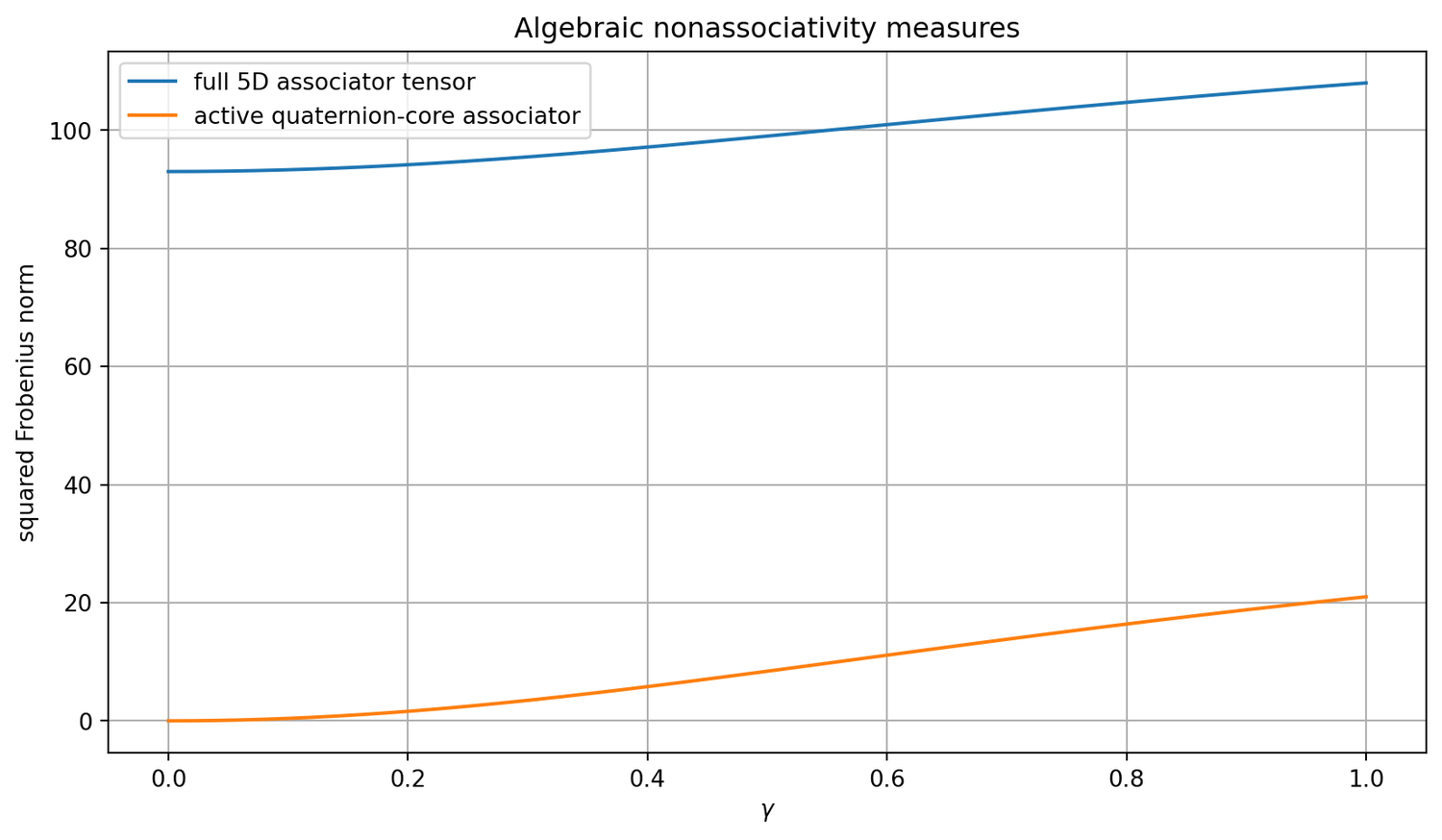}
\caption{Squared norms of the associator tensor.}
\end{subfigure}
\hfill
\begin{subfigure}[t]{0.48\textwidth}
\centering
\includegraphics[width=\textwidth]{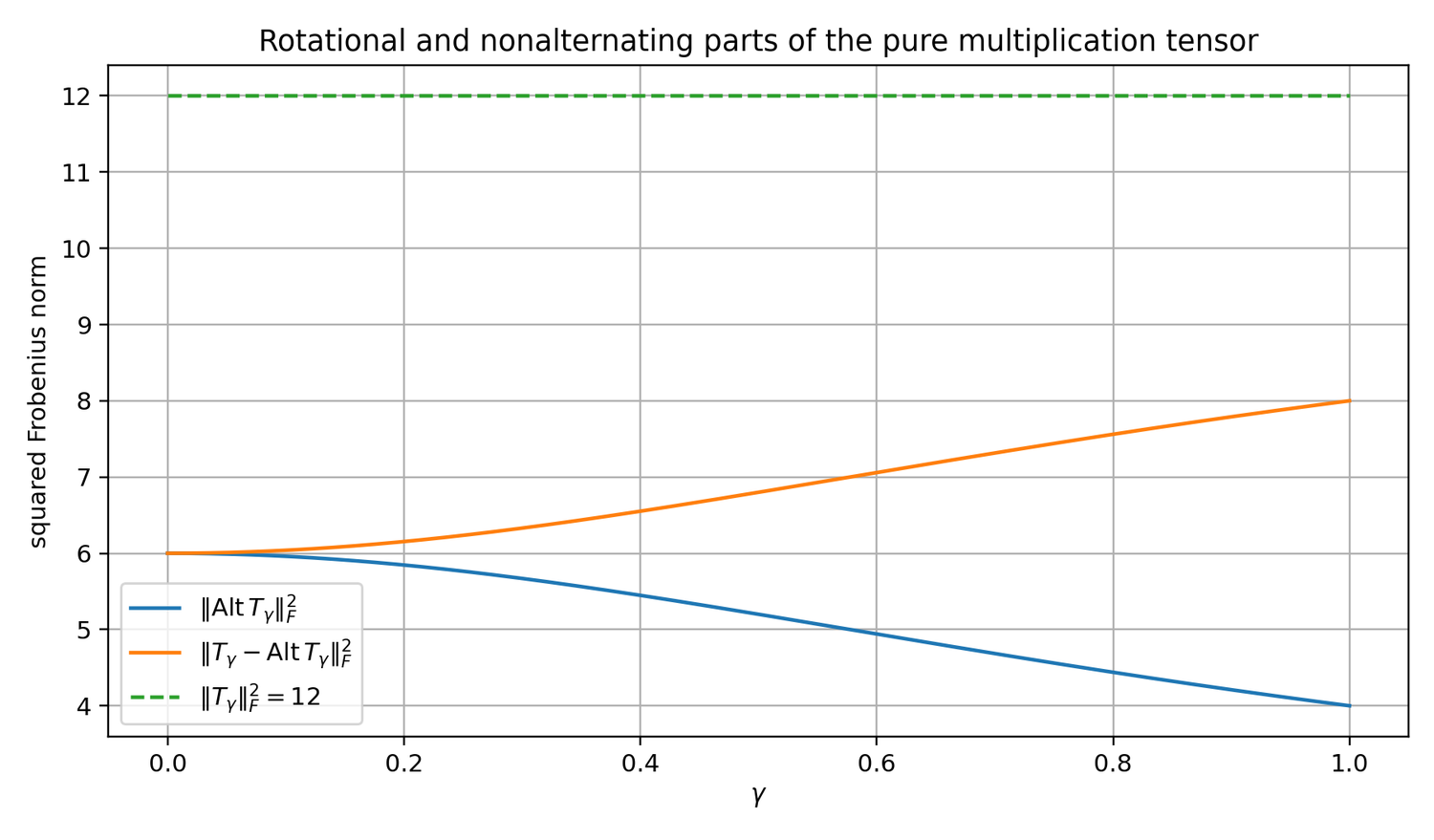}
\caption{Squared norms of trilinear forms.}
\end{subfigure}
\caption{Squared Frobenius norms as functions of $\gamma$.
Left: the associator tensor and its restriction to
$\Span\{i,j,k\}^3$ (Lemma~\ref{prop:core-associator}).
Right: the trilinear form $T_\gamma$, its alternating part
$\varphi_\gamma=\Alt T_\gamma$, and the difference
$T_\gamma-\varphi_\gamma$, as defined and calculated in
\cref{thm:rotational-three-form}.}
\label{fig:structural-invariants}
\end{figure}

\begin{theorem}
\label{thm:minimal-nuclei}
For every $0\leq\gamma\leq1$,
\[
 \operatorname{Nuc}_l(\SwitchAlg_\gamma)
 =\operatorname{Nuc}_m(\SwitchAlg_\gamma)
 =\operatorname{Nuc}_r(\SwitchAlg_\gamma)
 =\R1.
\]
\end{theorem}

\begin{proof}
We know that $\R\subset  \operatorname{Nuc}_l(\SwitchAlg_\gamma)$, $\R\subset  \operatorname{Nuc}_m(\SwitchAlg_\gamma)$ and $\R\subset  \operatorname{Nuc}_r(\SwitchAlg_\gamma)$.
Since the associator vanishes whenever an argument is scalar,
w.l.o.g. we now prove that no nonzero pure element belongs to any of the three nuclei.
Write $ x=\alpha i+\beta j+\zeta k+\delta\ell.$
Let $x\in\operatorname{Nuc}_l(\mathcal A_\gamma)$. A straightforward  calculation gives
\[
 [x,\ell,\ell]
 =
 \frac12(
 (4\alpha+\beta+\zeta)i
 +(\alpha+4\beta+\zeta)j
 +(\alpha+\beta+4\zeta)k
 ).
\]
The coefficient matrix we obtain from this set of equations is given by
\[
 \begin{pmatrix}
  4&1&1\\
  1&4&1\\
  1&1&4
 \end{pmatrix}
\]
and has determinant $54\not=0$, so that $[x,\ell,\ell]=0$ implies $ \alpha=\beta=\zeta=0.$
Thus $x=\delta\ell$ for all $x\in \operatorname{Nuc}_l(\mathcal{A}_\gamma)$.
We know that $[\ell,i,i] = \ell+\frac{\kg}{\sqrt2}(j-k)\not=0.$ Since $[x,i,i]=\delta[\ell,i,i],$
 $[x,i,i]=0$ gives $\delta=0$, therefore $ \Nuc_l(\mathcal A_\gamma)=\R1.$

Now let $x\in\Nuc_r(\mathcal{A}_\gamma)$ then
\[
 [\ell,\ell,x]
 =
 -\frac12(
 (4\alpha+\beta+\zeta)i
 +(\alpha+4\beta+\zeta)j
 +(\alpha+\beta+4\zeta)k).
\]
An analogous calculation  gives $ \alpha=\beta=\zeta=0,$
so $x=\delta\ell$.
Moreover,
\[
 [i,i,\ell]
 =
 -\ell+\frac{\kg}{\sqrt2}(k-j)\not=0,
\]
implies that $[i,i,x]=0 $ hence $ \delta=0,$
and therefore $ \Nuc_r(\mathcal A_\gamma)=\R\,1.$

Finally, let
$x\in\Nuc_m(\mathcal A_\gamma)$ and put
$ \omega=i+j+k,$$ \eta=\alpha+\beta+\zeta.$
Then we have
$ [\omega,x,\ell]
 = -\sqrt2 \eta\,1+3\delta\omega-\eta\ell.$
Since $1,\omega,\ell$ are linearly independent, the condition $[\omega,x,\ell]=0$
implies that $ \eta=0$ and $\delta=0$, thus
 $ \zeta=-\alpha-\beta.$
Put $q=i+j-2k.$
After substituting $\zeta=-\alpha-\beta$ and $\delta=0$, a straightforward
calculation gives
$$ [q,x,\ell] = 3(  \sg(\alpha-\beta)
   +\frac{\alpha+\beta}{\sqrt2} )1 +\frac{3\kg} {\sqrt2}
  (\alpha-\beta)(i+j+k)
 +3(
   \frac{\sg}{\sqrt2}(\alpha-\beta)
   -(\alpha+\beta))\ell.$$
Since
\[
 \kg=\frac1{\sqrt{1+\gamma^2}}>0
 \qquad
 (0\leq\gamma\leq1),
\]
the coefficient of $i+j+k$ forces $ \alpha=\beta.$
The scalar component then gives
$ \frac{3}{\sqrt2}(\alpha+\beta) = 3\sqrt2\,\alpha =
 0.$
Thus
$ \alpha=\beta=0,$ $\zeta=-\alpha-\beta=0,$
 $ \delta=0$
and so $ \Nuc_m(\mathcal A_\gamma)=\R1.$
\end{proof}

We now use the  basis ${\bf a}, {\bf b},{\bf c}, {\bf d}$. We know that for $u\in V$, $F_u$ is self-adjoint if and only if $u\in\R{\bf d}$ (Lemma~\ref{lem:sa}), and $F_{\bf d}=\diag(\sqrt{2},-\frac{1}{\sqrt{2}},-\frac{1}{\sqrt{2}},0)$ has eigenspaces $\R{\bf a}$, ${\rm span}\{{\bf b},{\bf c}\}$, $\R{\bf d}$.

\begin{theorem}\label{thm:automorphism-groups}
The automorphism groups are
\[\Aut(\mathcal{A}_0)\cong  O(2)\text{ and } \Aut(\mathcal{A}_\gamma)\cong\SO(2)\quad(0<\gamma\le1).\]
More precisely, $\Phi\in GL(V)$ extends to an automorphism of $\mathcal{A}_\gamma$ if and only if
$$
\Phi({\bf d})={\bf d} \text{ and }  R=\Phi|_{{\rm span}\{{\bf b},{\bf c}\}}\in O(2),\, \Phi({\bf a})=\det(R)\,{\bf a},$$
and $R\in\SO(2)$ when $\gamma>0$.

Thus for $\gamma>0$, the group fixes ${\bf a}$ and ${\bf d}$ and rotates the ${\bf b}$-${\bf c}$ plane; for $\gamma=0$ it fixes ${\bf d}$, acts on the ${\bf b}$-${\bf c}$ plane by an arbitrary isometry and on ${\bf a}$ by its determinant.
\end{theorem}

\begin{proof}
We write  $\Phi$ for $\Phi|_V\in O(V)$  where $\Phi$ is an automorphism of $\mathcal{A}_\gamma$ (Lemma~\ref{lem:red}$(i)$).

By Lemma~\ref{lem:red}$(ii)$, we have $F_{\Phi({\bf d})}=\Phi F_{\bf d}\Phi^{-1}=\Phi F_{\bf d}\Phi^T$, which is self-adjoint because $F_{\bf d}$ is. By Lemma~\ref{lem:sa}, $\Phi({\bf d})\in\R{\bf d}$, and since $\|\Phi({\bf d})\|=\|{\bf d}\|=1$ we get $\Phi({\bf d})=\varepsilon{\bf d}$ with $\varepsilon=\pm1$.

 Since $\varepsilon F_{\bf d}=F_{\Phi({\bf d})}=\Phi F_{\bf d}\Phi^{-1}$,  $\varepsilon F_{\bf d}$ and $F_{\bf d}$ have the same eigenvalues. But $\sqrt{2}$ is an eigenvalue of $F_{\bf d}$ and not of $-F_{\bf d}$, whose eigenvalues are $-\sqrt{2},\tfrac1{\sqrt{2}},0$. Hence $\varepsilon=1$.

 Since $\varepsilon=1$, we see that $\Phi F_{\bf d}=F_{\bf d}\Phi$, so $\Phi$ preserves each eigenspace of $F_{\bf d}$. Thus $\Phi({\bf a})=\eta{\bf a}$ with $\eta=\pm1$ (as $\|\Phi({\bf a})\|=1$), and $\Phi$ maps the plane ${\rm span}\{{\bf b},{\bf c}\}$ onto itself. Let $R\in O(2)$ be the matrix of $\Phi|_{{\rm span}\{{\bf b},{\bf c}\}}$ in the orthonormal basis ${\bf b},{\bf c}$.

For any linear map of the plane ${\rm span}\{{\bf b},{\bf c}\}$ with matrix $\left(\begin{smallmatrix}p&r\\ q&s\end{smallmatrix}\right)$ in the basis ${\bf b},{\bf c}$, we get
\[(p{\bf b}+q{\bf c})\times(r{\bf b}+s{\bf c})=(ps-qr)\,{\bf b}\times{\bf c} \]
 since $\times$ is a bilinear and alternating map.

Applying $\Phi$ to ${\bf b}\times{\bf c}=\kg{\bf a}+\sqrt{3}\sg{\bf d}$ we obtain
\[\kg\eta\,{\bf a}+\sqrt3\sg\,{\bf d}=\Phi({\bf b}\times{\bf c})=\Phi({\bf b})\times\Phi({\bf c})=\det(R)\,\big(\kg{\bf a}+\sqrt3\sg{\bf d}\big).\]
Comparing the ${\bf a}$-coefficients gives $\eta=\det R$; comparing the ${\bf d}$-coefficients gives $\sg\,(1-\det R)=0$. If $\gamma>0$ then $\sg>0$, so $\det R=1=\eta$. This also proves  the description of the group action stated in the theorem.

For the converse,
let $\Phi\colon V\to V$ satisfy the conditions of the theorem, and put $\eta=\det R$; so $\eta=\pm1$, and $\eta=1$ if $\gamma>0$. Every element of $O(2)$ with determinant $\eta$ has the form $R=\left(\begin{smallmatrix}p&-\eta q\\ q&\eta p\end{smallmatrix}\right)$ with $p^2+q^2=1$, i.e.
\[\Phi({\bf a})=\eta{\bf a},\quad \Phi({\bf b})=p{\bf b}+q{\bf c},\quad \Phi({\bf c})=\eta(-q{\bf b}+p{\bf c}),\quad \Phi({\bf d})={\bf d} .\]
$\Phi$ is orthogonal, since it is an isometry on each of the mutually orthogonal subspaces $\R{\bf a}$, ${\rm span}\{{\bf b},{\bf c}\}$, $\R{\bf d}$. As $\times$ is bilinear and alternating, $\Phi(u\times v)=\Phi(u)\times\Phi(v)$ holds for all $u,v\in V$ as soon as it holds for the following six pairs; using $\eta^2=1$ and $\sg(1-\eta)=0$ we indeed obtain
\begin{align*}
\Phi({\bf a})\times\Phi({\bf b})&=\eta\,{\bf a}\times(p{\bf b}+q{\bf c})=\kg\eta\,(p{\bf c}-q{\bf b})=\kg\,\Phi({\bf c})=\Phi({\bf a}\times{\bf b}),\\
\Phi({\bf a})\times\Phi({\bf c})&=\eta^2\,{\bf a}\times(-q{\bf b}+p{\bf c})=-\kg\,(p{\bf b}+q{\bf c})=-\kg\,\Phi({\bf b})=\Phi({\bf a}\times{\bf c}),\\
\Phi({\bf b})\times\Phi({\bf c})&=\det(R)\,{\bf b}\times{\bf c}=\eta\kg\,{\bf a}+\eta\sqrt3\sg\,{\bf d}=\kg\eta\,{\bf a}+\sqrt3\sg\,{\bf d}=\Phi({\bf b}\times{\bf c}),\\
\Phi({\bf a})\times\Phi({\bf d})&=\eta\,{\bf a}\times{\bf d}=-\sqrt{2}\,\eta{\bf a}=\Phi({\bf a}\times{\bf d}),\\
\Phi({\bf b})\times\Phi({\bf d})&=(p{\bf b}+q{\bf c})\times{\bf d}=\tfrac1{\sqrt{2}}(p{\bf b}+q{\bf c})=\Phi({\bf b}\times{\bf d}),\\
\Phi({\bf c})\times\Phi({\bf d})&=\eta(-q{\bf b}+p{\bf c})\times{\bf d}=\tfrac1{\sqrt{2}}\,\eta(-q{\bf b}+p{\bf c})=\Phi({\bf c}\times{\bf d}).
\end{align*}
By Lemma~\ref{lem:red}$(i)$, $\Phi$ extends to an automorphism of $\mathcal{A}_\gamma$.

By the above, an automorphism $\Phi$ is determined by $R=\Phi|_{{\rm span}\{{\bf b},{\bf c}\}}$, since $\Phi({\bf a})=\det(R)\,{\bf a}$ and $\Phi({\bf d})={\bf d}$, and $\Phi\mapsto R$ is a group homomorphism. It is injective, and its image is $\SO(2)$ for $\gamma>0$ and $O(2)$ for $\gamma=0$, by the above. Hence $\Aut(\mathcal{A}_\gamma)\cong\SO(2)$ for $\gamma>0$ and $\Aut(\mathcal{A}_0)\cong O(2)$.
\end{proof}

\begin{remark}
The algebra $\mathcal{A}_0$ contains the quaternion algebra $\mathbb{H}={\rm span}\{1,i,j,k\}$, whose automorphism group is $\SO(3)$ acting on ${\rm span}\{i,j,k\}$. By Theorem \ref{thm:automorphism-groups}, $\Aut(\mathcal{A}_0)$ is the stabiliser in $\SO(3)$ of $\R{\bf a}$, which is $O(2)$ (rotations about ${\bf a}$, and rotations by $\pi$ about axes in the ${\bf b}$-${\bf c}$ plane, which reverse ${\bf a}$); for $\gamma>0$ the ${\bf d}$-component of ${\bf b}\times{\bf c}$ orients the axis and only the rotations about ${\bf a}$ survive.
\end{remark}

\begin{remark}
Let $q=a_0+u$ and $r=b_0+v$ with $a_0,b_0\in\R$ and $u,v\in\Span\{i,j,k\}$, and let $u\cdot v$ and
$u\times v$ be the classical dot and cross product on $\Span\{i,j,k\}\cong\R^3$. Then
\begin{equation}
 q\star_\gamma r
 =\underbrace{a_0b_0-u\cdot v+a_0v+b_0u+u\times v}_{=qr \text{ in } \mathbb H}
 +(\kappa_\gamma-1)u\times v+\sigma_\gamma\det(i+j+k,u,v)\ell .
 \label{eq:leakage}
\end{equation}
For $\gamma=0$ this is the quaternion product. For small nonnegative $\gamma$, the difference for $\star_\gamma$ in
$\Span\{i,j,k\}$ is given by $\|(\kappa_\gamma-1)\,u\times v\|\le\frac{\gamma^2}{2}\|u\|\|v\|$ , while the $\ell$-component has size
$|\sigma_\gamma\det(i+j+k,u,v)|\le\sqrt3\gamma\|u\|\|v\|$ using $1-\kappa_\gamma\le\gamma^2/2$, $\sigma_\gamma\leq\gamma$, $\|u\times v\|\leq\|u\|\|v\|$ and
$\|i+j+k\|=\sqrt3$.

These observations will reappear in the rotational geometry below and in
the exact elliptic solutions studied in Part~II~\cite{PoshtkohiPumpluenII}.
\end{remark}

\FloatBarrier
\section{Rotational geometry of $\mathcal A_\gamma$}
\label{sec:rotational-geometry}

\subsection{Sandwich products and the alternating trilinear form}\label{sec:5.1}

Let $\theta\in \R$ and $u\in\Vpure=\Span_\R\{i,j,k,\ell\}$ with $\|u\|=1$. Then $\R(u)={\rm span}_{\mathbb R}\{1,u\}$ is a subalgebra of $\mathcal A_\gamma$ isomorphic to $\C$ (also called a \emph{slice}).
Set
\[
 r=\cos (\frac{\theta}{2})+u\sin (\frac{\theta}{2}),
 \quad
 \bar r=\cos(\frac{\theta}{2})-u\sin
 (\frac{\theta}{2}).
\]
For $\gamma=0$ and $u,v\in\Span_\R\{i,j,k\}$, the map $v\mapsto r v\bar r$ is the quaternion rotation by the angle $\theta$ about the axis $\R u$. In $\mathcal A_\gamma$ we get
 $r\star_\gamma\bar r=1$.
In $\mathcal A_\gamma$, however,
the expression $r\star_\gamma v\star_\gamma\bar r$ is in general not well-defined and we have two possible bracketings:
\[
 T^L_{\gamma,u,\theta}(v)
 =(r\star_\gamma v)\star_\gamma\bar r,
 \qquad
 T^R_{\gamma,u,\theta}(v)
 =r\star_\gamma(v\star_\gamma\bar r).
\]
Define
\[
 C_u(v)=u\crossg v,
 \qquad
 P_u(v)=u\ip{u}{v},
 \qquad
 h_u(v)=\ip{u\crossg v}{u}.
\]

\begin{theorem}\label{thm:raw-sandwich}
Let $u\in V$ with $\|u\|=1$, $\theta\in\R$, $r=\cos\frac\theta2+u\sin (\frac{\theta}{2})$ and
$\bar r=\cos(\frac{\theta}{2})-u\sin(\frac{\theta}{2})$. For every $v\in V$, we have
\begin{align}
 T^L_{\gamma,u,\theta}(v)=(r\star_\gamma v)\star_\gamma\bar r
 &={\sin^2(\frac{\theta}{2})}h_u(v)\,1+M_{\gamma,u,\theta}v,\label{eq:left-sandwich}\\
 T^R_{\gamma,u,\theta}(v)=r\star_\gamma(v\star_\gamma\bar r)
 &=-\sin^2(\frac{\theta}{2})h_u(v)\,1+M_{\gamma,u,\theta}v,\label{eq:right-sandwich}
\end{align}
where
\begin{equation}
 M_{\gamma,u,\theta}
 =\cos^2(\frac{\theta}{2})I+\sin (\theta)C_u+\sin^2 (\frac{\theta}{2})(C_u^2+P_u).
 \label{eq:raw-pure-operator-main}
\end{equation}
In particular, this means
\[
 T^L_{\gamma,u,\theta}(v)-T^R_{\gamma,u,\theta}(v)=2\sin^2 (\frac{\theta}{2}) h_u(v)\,1
 =-\sin^2 (\frac{\theta}{2})[u,v,u]_\gamma .
\]
\end{theorem}

\begin{proof}
For all $x,y\in V$ we have
\[
x\star_\gamma y=-\langle x,y\rangle\,1+x\times_\gamma y. \tag{$\ast$}
\]
This implies that
\begin{align*}
(u\star_\gamma v)\star_\gamma u&=-\langle u,v\rangle u-\langle u\times v,u\rangle 1+(u\times v)\times u,\\
u\star_\gamma(v\star_\gamma u)&=-\langle u,v\rangle u-\langle u,v\times u\rangle 1+u\times(v\times u).
\end{align*}
Since $\times$ is anticommutative, $(u\times v)\times u=-u\times(u\times v)=u\times(v\times u)$
and $\langle u,v\times u\rangle=-\langle u\times v,u\rangle=-h_u(v)$. This implies
 $[u,v,u]_\gamma=-2h_u(v)\,1$.

Put $\lambda=\cos\frac\theta2$, $\mu=\sin\frac\theta2$ and $w=\lambda v+\mu C_uv\in V$.
By $(\ast)$,
we have
$r\star_\gamma v=\lambda v+\mu u\star_\gamma v=-\mu\langle u,v\rangle 1+w,$
and $v\star_\gamma\bar r=\lambda v-\mu v\star_\gamma u=\mu\langle u,v\rangle 1+w .$
We need $\langle w,u\rangle=\lambda\langle u,v\rangle+\mu h_u(v)$ and
$w\times u=-u\times w=-\lambda C_uv-\mu C_u^2v$.

Using $(\ast)$ for $w\star_\gamma u$,
\begin{align*}
(r\star_\gamma v)\star_\gamma\bar r
&=(-\mu\langle u,v\rangle 1+w)\star_\gamma(\lambda-\mu u)\\
&=-\lambda\mu\langle u,v\rangle 1+\mu^2\langle u,v\rangle u+\lambda w
 +\mu\langle w,u\rangle 1-\mu w\times u\\
&=\mu^2h_u(v)\,1+\lambda^2v+2\lambda\mu C_uv+\mu^2(C_u^2v+P_uv).
\end{align*}

 Using $(\ast)$ for $u\star_\gamma w$, we have
\begin{align*}
r\star_\gamma(v\star_\gamma\bar r)
&=(\lambda+\mu u)\star_\gamma(\mu\langle u,v\rangle 1+w)\\
&=\lambda\mu\langle u,v\rangle 1+\lambda w+\mu^2\langle u,v\rangle u
 -\mu\langle u,w\rangle 1+\mu u\times w\\
&=-\mu^2h_u(v)\,1+\lambda^2v+2\lambda\mu C_uv+\mu^2(C_u^2v+P_uv).
\end{align*}
Since $\lambda^2=\cos^2\frac\theta2$, $2\lambda\mu=\sin\theta$ and $\mu^2=\sin^2\frac\theta2$,
this proves \eqref{eq:left-sandwich} and \eqref{eq:right-sandwich}. Subtracting these identities and using $[u,v,u]_\gamma=-2h_u(v)\,1$ gives the stated difference between the two bracketings.
\end{proof}

\begin{theorem}
\label{thm:no-go}

Let $0\le\gamma\le1$ and let $u\in V$ with $\|u\|=1$. Then there exists $\varepsilon>0$ such that
the linear map $M_{\gamma,u,\theta}: V\to V$ is not orthogonal for all
$0<|\theta|<\varepsilon$.

\end{theorem}

\begin{proof}
Write $M(\theta)=M_{\gamma,u,\theta}$. Let $O(\theta^2)$ denote a matrix with entries that are bounded by a constant times $\theta^2$ for small $\theta$.

 Since $\cos^2 (\frac{\theta}{2})=1+O(\theta^2)$,
$\sin(\theta)=\theta+O(\theta^3)$ and $\sin^2 (\frac{\theta}{2})=O(\theta^2)$ as $\theta\to 0$,
Equation~\eqref{eq:raw-pure-operator-main} yields
$M(\theta)=I+\theta C_u+O(\theta^2),$ so that we have
$M(\theta)^TM(\theta)=I+\theta\,(C_u+C_u^T)+O(\theta^2)$ as $\theta\to 0$.

Suppose $C_u+C_u^T\neq0$, that means there are integers $r$ and $t$ such that the $(r,t)$-entry $s=(C_u+C_u^T)_{r,t}\neq 0$ is not zero. By the equation above there
are $K,\theta_0>0$ such that the $(r,t)$-entry of $E(\theta)=M(\theta)^TM(\theta)-I-\theta (C_u+C_u^T)$
satisfies $|E(\theta)_{r,t}|\leq K\theta^2$ for all $|\theta|\leq\theta_0$. Hence
\[
\bigl|\bigl(M(\theta)^TM(\theta)-I\bigr)_{rt}\bigr|
\ \ge\ |\theta|\,|s|-K\theta^2
\ =\ |\theta|\,\bigl(|s|-K|\theta|\bigr)\ >\ 0
\]
for $0<|\theta|<\varepsilon=\min\{\theta_0,|s|/K\}$, thus $M(\theta)^T M(\theta)\neq I$.

It therefore remains to show that $C_u$ is not skew-symmetric.
Write $u=A\mathbf a+B\mathbf b+C\mathbf c+D\mathbf d$. By  \eqref{eq:adapted-products-main}, the matrix of $C_u$
in the orthonormal basis $\mathbf a,\mathbf b,\mathbf c,\mathbf d$ is
\[
\begin{pmatrix}
\sqrt{2} D & -\kappa_\gamma C & \kappa_\gamma B & -\sqrt{2} A\\
\kappa_\gamma C & -D/\sqrt2 & -\kappa_\gamma A & B/\sqrt{2}\\
-\kappa_\gamma B & \kappa_\gamma A & -D/\sqrt{2} & C/\sqrt{2}\\
0 & -\sqrt3\sigma_\gamma C & \sqrt{3}\sigma_\gamma B & 0
\end{pmatrix}.
\]
Suppose $C_u$ is skew-symmetric. This implies that the diagonal of $C_u$  gives $D=0$, and the entries $(1,4)$ and $(4,1)$
give $A=0$. The entries $(2,4),(4,2)$ and $(3,4),(4,3)$ of $C_u$ yield
\[
\frac{1}{\sqrt{2}}B-\sqrt{3}\sigma_\gamma C=0,\quad \sqrt{3}\sigma_\gamma B+\frac {1}{\sqrt{2}}C=0.
\]
 The  determinant of the corresponding $2\times 2$-matrix equals $\frac{1}{2}+3\sigma_\gamma^2>0$. Hence $B=C=0$, so $u=0$,
contradicting $\|u\|=1$.

\end{proof}

\begin{example}\label{ex:norm-defect}
Let $\gamma=1$,  $r=(1+i)/\sqrt2, j\in V$ so that $\|r\|=\|j\|=1$.  Then
\[
 (r\star_1 j)\star_1\bar r=\left(\frac1{\sqrt2}-\frac14\right)(k+\ell)
\]
and
\[
 \norm{(r\star_1 j)\star_1\bar r}^2=\frac{9-4\sqrt2}{8}\approx0.418.
\]
Thus we get $\|(r\star_1 j)\star_1\bar r\|<\|j\|$. Furthermore, since $h_i(j)=\langle i\times_1 j,i\rangle=0$, we get $T^L_{1,i,\pi/2}(j)=T^R_{1,i,\pi/2}(j)=M_{1,i,\pi/2}j$  (Theorem \ref{thm:no-go}). So both bracketings agree here, and $\|M_{1,i,\pi/2}\,j\|<\|j\|$ implies that $M_{1,i,\pi/2}$ is not
orthogonal. This shows the equation observed in
Theorem~\ref{thm:no-go} at the large angle $\theta=\pi/2$; see
Figure~\ref{fig:rotation-obstruction-and-shadow}(a).
Note that in $\mathcal A_0$, the same computation
yields that $(r\star_0 j)\star_0\bar r=k$. At $\gamma=0$, $(\Span\{\mathbf a,\mathbf b,\mathbf c\},\times_0)$ is isomorphic to $(\R^3,\times)$ with $\times$ the classical vector crossed product.
\end{example}

Define the \emph{scalar pure-product tensor}
\[
 T_\gamma(u,v,w)=\ip{u\crossg v}{w}.
\]
Since $T_\gamma$ is antisymmetric in its first two arguments, its full alternation is the cyclic average
\[
 \varphi_\gamma=\Alt T_\gamma.
\]

\begin{theorem}
\label{thm:rotational-three-form}
We have
\begin{equation}
\varphi_\gamma
 =\kg\,\textbf a^*\wedge \textbf b^*\wedge \textbf c^*
 +\frac{\sg}{\sqrt3}\,\textbf b^*\wedge \textbf c^*\wedge \textbf d^*.
 \label{eq:rotational-three-form}
\end{equation}
Moreover, for the Frobenius norm we get
\[
 \norm{T_\gamma}_{\Frob}^2=12,
 \qquad
 \norm{\varphi_\gamma}_{\Frob}^2
 =\frac{2(\gamma^2+3)}{1+\gamma^2},
\]
and
\[
 \norm{T_\gamma-\varphi_\gamma}_{\Frob}^2
 =\frac{2(5\gamma^2+3)}{1+\gamma^2}.
\]
\end{theorem}

\begin{proof}
For $x\in V$ we write $x^*=\langle x,\cdot\rangle\in V^*$.
Let $\mathbf a^*,\mathbf b^*,\mathbf c^*,\mathbf d^*$ denote the dual basis of $\mathbf a,\mathbf b,\mathbf c,\mathbf d$, and for linear forms
$\alpha,\beta,\delta$ on $V$ we use
$(\alpha\wedge\beta\wedge\delta)(x,y,z)=\det
\begin{pmatrix}\alpha(x)&\alpha(y)&\alpha(z)\\
\beta(x)&\beta(y)&\beta(z)\\
\delta(x)&\delta(y)&\delta(z)
\end{pmatrix}$.
Then every alternating trilinear form $S:V\times V\times V\to\R$ satisfies
$S=\sum_{1\le r<s<t\le4}S(e_r,e_s,e_t)e_r^*\wedge e_s^*\wedge e_t^*$,
where $(e_1,e_2,e_3,e_4)=(\mathbf a,\mathbf b,\mathbf c,\mathbf d)$.

Since $T_\gamma$ is antisymmetric in its first two arguments, we have
\[
\varphi_\gamma(x,y,z)=\frac{1}{3}(T_\gamma(x,y,z)+T_\gamma(y,z,x)+T_\gamma(z,x,y)).
\]
From \eqref{eq:adapted-products-main} we use
\[
\begin{aligned}
\mathbf a\times\mathbf b&=\kappa_\gamma\mathbf c,&
\mathbf b\times\mathbf c&=\kappa_\gamma\mathbf a+\sqrt3\,\sigma_\gamma\mathbf d,&
\mathbf c\times\mathbf a&=\kappa_\gamma\mathbf b,\\
\mathbf a\times\mathbf d&=-\sqrt2\,\mathbf a,&
\mathbf b\times\mathbf d&=\tfrac1{\sqrt2}\mathbf b,&
\mathbf c\times\mathbf d&=\tfrac1{\sqrt2}\mathbf c
\end{aligned}
\]
to compute
\begin{align*}
\varphi_\gamma(\mathbf a,\mathbf b,\mathbf c)
&=\tfrac13\bigl(\langle\mathbf a\times\mathbf b,\mathbf c\rangle
 +\langle\mathbf b\times\mathbf c,\mathbf a\rangle
 +\langle\mathbf c\times\mathbf a,\mathbf b\rangle\bigr)=\kappa_\gamma,\\
\varphi_\gamma(\mathbf b,\mathbf c,\mathbf d)
&=\tfrac13\bigl(\langle\mathbf b\times\mathbf c,\mathbf d\rangle
 +\langle\mathbf c\times\mathbf d,\mathbf b\rangle
 +\langle\mathbf d\times\mathbf b,\mathbf c\rangle\bigr)\\
&=\tfrac13\sqrt3\,\sigma_\gamma=\tfrac{\sigma_\gamma}{\sqrt3},
\end{align*}
and analogously we obtain $\varphi_\gamma(\mathbf a,\mathbf b,\mathbf d)=\varphi_\gamma(\mathbf a,\mathbf c,\mathbf d)=0$,
since $\mathbf a\times\mathbf d\in\R\mathbf a$, $\mathbf b\times\mathbf d\in\R\mathbf b$,
$\mathbf c\times\mathbf d\in\R\mathbf c$ and $\mathbf a\times\mathbf b, \mathbf a\times\mathbf c\perp\mathbf d$.
This proves \eqref{eq:rotational-three-form}.

Each nonzero coefficient of the alternating form $\varphi_\gamma$ occurs, up to sign, in the
six orderings of its triple, so
\[
\|\varphi_\gamma\|_\Frob^2=6\Bigl(\kappa_\gamma^2+\frac{\sigma_\gamma^2}{3}\Bigr)=\frac{2(\gamma^2+3)}{1+\gamma^2}.
\]
In the orthonormal basis $i,j,k,\ell$, we have
$\|T_\gamma\|_\Frob^2=\sum_{x,y}\|x\times y\|^2$, and each of the $12$ ordered pairs of distinct
basis vectors satisfies $\|x\times y\|=1$
(e.g.\ $\|i\times j\|^2=\kappa_\gamma^2+\sigma_\gamma^2=1$, $\|i\times\ell\|^2=\frac12+\frac12=1$).
Hence $\|T_\gamma\|_\Frob^2=12$.

Let $S,T: V\times V\times V\to\R$ be  trilinear forms on $V$ and let  $e_1,\dots,e_4$ be an orthonormal basis of $V$, then write $\langle S,T\rangle_\Frob=\sum_{p,q,r}S(e_p,e_q,e_r)T(e_p,e_q,e_r)$, so that in particular $\|T\|_\Frob^2=\langle T,T\rangle_\Frob$.

 For a permutation $\pi\in S_3$ define the trilinear form
$(\pi T)(x_1,x_2,x_3)=T(x_{\pi(1)},x_{\pi(2)},x_{\pi(3)})$, so that
$\operatorname{Alt}T=\frac16\sum_{\pi\in S_3}\operatorname{sgn}(\pi)\,\pi T$.
Since $(p_1,p_2,p_3)\mapsto(p_{\pi(1)},p_{\pi(2)},p_{\pi(3)})$ is a bijection, we have $\langle\pi S,\pi T\rangle_\Frob=\langle S,T\rangle_\Frob$, hence
$\langle\pi S,T\rangle_F=\langle S,\pi^{-1}T\rangle_F$. Since $\varphi_\gamma$ is alternating,
$\pi^{-1}\varphi_\gamma=\operatorname{sgn}(\pi)\varphi_\gamma$. Therefore
\[
\begin{aligned}
\langle\varphi_\gamma,\varphi_\gamma\rangle_F
&=\frac16\sum_{\pi\in S_3}\operatorname{sgn}(\pi)\,\langle\pi T_\gamma,\varphi_\gamma\rangle_F\\
&=\frac16\sum_{\pi\in S_3}\operatorname{sgn}(\pi)^2\,\langle T_\gamma,\varphi_\gamma\rangle_F
=\langle T_\gamma,\varphi_\gamma\rangle_\Frob,
\end{aligned}
\]
i.e.\ $\langle T_\gamma-\varphi_\gamma,\varphi_\gamma\rangle_\Frob=0$. Consequently
\[
\|T_\gamma-\varphi_\gamma\|_F^2=\|T_\gamma\|_\Frob^2-\|\varphi_\gamma\|_\Frob^2
=12-\frac{2(\gamma^2+3)}{1+\gamma^2}=\frac{2(5\gamma^2+3)}{1+\gamma^2}.
\]
\end{proof}

 From the Sandwich Theorem \ref{thm:raw-sandwich}, we have $h_u(v)=T_\gamma(u,v,u)$ which vanishes for all $u,v\in V$ if and only if $T_\gamma$ is alternating.

\subsection{A distinguished three-dimensional subspace}

Define
\begin{equation}
 h_\gamma=\frac{\sqrt3\,\textbf a+\gamma \textbf d}{\sqrt{3+\gamma^2}},
 \qquad
 n_\gamma=\frac{-\gamma \textbf a+\sqrt3 \textbf d}{\sqrt{3+\gamma^2}}.
 \label{eq:h-n-main}
\end{equation}
These are orthonormal.

\begin{theorem}[The
distinguished subspace $H_\gamma$]
\label{thm:Hgamma}
Let $\gamma\ge0$. Let $\varphi_\gamma=\operatorname{Alt}T_\gamma$ be the alternating trilinear form from Theorem \ref{thm:rotational-three-form},
i.e.,
$$\varphi_\gamma(x,y,z)=\tfrac13(\langle x\times_\gamma y,z\rangle+\langle y\times_\gamma z,x\rangle+\langle z\times_\gamma x,y\rangle)$$
for all $x,y,z\in V$.
Then
\[
 \varphi_\gamma
 =\nu_\gamma h_\gamma^*\wedge \textbf b^*\wedge \textbf c^*,
 \qquad
 \nu_\gamma=\sqrt{\frac{3+\gamma^2}{3(1+\gamma^2)}}
\]
and the
\emph{kernel} is
$\ker\varphi_\gamma=\{x\in V\,|\,\varphi_\gamma(x,y,z)=0\text{ for all }y,z\in V\}=\R n_\gamma$. Moreover, we have
$$H_\gamma:=n_\gamma^\perp=\Span\{h_\gamma,\textbf b,\textbf c\}=\{xi+yj+zk+t\ell\,|\,3t=\gamma(x+y+z)\}. $$
The angle $\beta_\gamma\in[0,\pi/2)$ between $H_\gamma$ and $H_0=\Span\{i,j,k\}$, i.e., between
 $n_\gamma$ and $\ell$, is
\begin{equation}
 \beta_\gamma=\arctan\frac{\gamma}{\sqrt3}.
 \label{eq:tilt-angle}
\end{equation}
In particular $\beta_0=0$ and $\beta_1=\pi/6$.
\end{theorem}

\begin{proof}
Since $\mathbf b^*\wedge\mathbf c^*\wedge\mathbf d^*=\mathbf d^*\wedge\mathbf b^*\wedge\mathbf c^*$,
 we have
\[
 \varphi_\gamma=(\kappa_\gamma\mathbf a^*+\frac{\sigma_\gamma}{\sqrt3}\mathbf d^*)\wedge\mathbf b^*\wedge\mathbf c^*
\]
by \eqref{eq:rotational-three-form}.
For $\kappa_\gamma\mathbf a+\frac{\sigma_\gamma}{\sqrt3}\mathbf d$ we have
$\kappa_\gamma^2+\frac{\sigma_\gamma^2}{3}=\frac{3+\gamma^2}{3(1+\gamma^2)}=\nu_\gamma^2$, and
\[
 \frac{\kappa_\gamma}{\nu_\gamma}=\frac{\sqrt3}{\sqrt{3+\gamma^2}},\qquad
 \frac{\sigma_\gamma}{\sqrt3\,\nu_\gamma}=\frac{\gamma}{\sqrt{3+\gamma^2}} .
\]
Therefore $\kappa_\gamma\mathbf a+\frac{\sigma_\gamma}{\sqrt3}\mathbf d=\nu_\gamma h_\gamma$.

Since $h_\gamma,\mathbf b,\mathbf c,n_\gamma$ form an orthonormal basis of $V$
for $h_\gamma^*\wedge\mathbf b^*\wedge\mathbf c^*$ we get
$$ \varphi_\gamma(x,\mathbf b,\mathbf c)=\nu_\gamma\langle h_\gamma,x\rangle,
 \varphi_\gamma(x,\mathbf c,h_\gamma)=\nu_\gamma\langle\mathbf b,x\rangle,
 \varphi_\gamma(x,h_\gamma,\mathbf b)=\nu_\gamma\langle\mathbf c,x\rangle.$$
Hence $x\in\ker\varphi_\gamma$ implies $x\perp h_\gamma,\mathbf b,\mathbf c$, i.e. $x\in\R n_\gamma$.

Conversely, $\varphi_\gamma(n_\gamma,y,z)=0$ for all $y,z$, since every entry in the first column of
the determinant is $h_\gamma^*(n_\gamma)$, $\mathbf b^*(n_\gamma)$ or $\mathbf c^*(n_\gamma)$, and all are zero.
Thus $\ker\varphi_\gamma=\R n_\gamma$ and $n_\gamma^\perp=\Span\{h_\gamma,\mathbf b,\mathbf c\}$.

For $v=xi+yj+zk+t\ell$, using $\mathbf a=(i+j+k)/\sqrt3$ and $\mathbf d=\ell$, we obtain
\[
 \sqrt{3+\gamma^2}\,\langle n_\gamma,v\rangle=-\gamma\langle\mathbf a,v\rangle+\sqrt3\langle\mathbf d,v\rangle
 =\frac{-\gamma(x+y+z)+3t}{\sqrt3},
\]
which gives the coordinate description of $H_\gamma$. Finally, $n_0=\ell$ and
$\cos\beta_\gamma=\langle n_\gamma,\ell\rangle=\sqrt3/\sqrt{3+\gamma^2}$, so $\tan\beta_\gamma=\gamma/\sqrt3$.
\end{proof}

We note that $\cos\beta_\gamma=\langle n_\gamma,\ell\rangle$, where $H_\gamma=n_\gamma^\perp$ and $H_0=\ell^\perp$.

\subsection{\texorpdfstring{An explicit embedding of $G_\gamma\cong\mathrm{SO}(3)$ into $SO(4)$}{An explicit embedding of G-gamma into SO(4)}}\label{sec:5.3}
We know  that $h_\gamma,\mathbf b,\mathbf c$ is an
orthonormal basis of $H_\gamma$ (Theorem~\ref{thm:Hgamma})
 and  $n_\gamma=(-\gamma\mathbf a+\sqrt3\,\mathbf d)/\sqrt{3+\gamma^2}$.
 Thus we know that $V=H_\gamma\perp\R n_\gamma$.
For all $u,v\in H_\gamma$, let $u\boxtimes_\gamma v\in H_\gamma$ be the unique vector such that
\begin{equation}
 \langle u\boxtimes_\gamma v,w\rangle=\frac{\varphi_\gamma(u,v,w)}{\nu_\gamma}
 \label{eq:box-cross}
\end{equation}
for all  $ w\in H_\gamma$.
By Theorem~\ref{thm:Hgamma}, we have
\[
 h_\gamma\boxtimes_\gamma\mathbf b=\mathbf c,\quad
 \mathbf b\boxtimes_\gamma\mathbf c=h_\gamma,\quad
 \mathbf c\boxtimes_\gamma h_\gamma=\mathbf b,
\]
i.e.
$\boxtimes_\gamma$ is the
classical cross product on $\R^3$.
For $u\in H_\gamma$ with $\|u\|=1$ and $\theta\in\R$, let $R_{\gamma,u}(\theta): V\to V$ be the
linear map with $R_{\gamma,u}(\theta)n_\gamma=n_\gamma$ and
\begin{equation}
 R_{\gamma,u}(\theta)v
 =v\cos(\theta)+(u\boxtimes_\gamma v)\sin(\theta)+\langle u,v\rangle u(1-\cos(\theta))
 \label{eq:rodrigues-shadow}
\end{equation}
for all $v\in H_\gamma $.
Thus $R_{\gamma,u}(\theta)$ is the rotation of $H_\gamma$ by the angle $\theta$ about the axis $\R u$,
extended to $V$ by fixing $n_\gamma$.

Recall that  ${\rm SO}(V)=\{Q\in \rm GL(V) \,|\, \langle Qx,Qy\rangle=\langle x,y\rangle
\text{ for all } x,y\in V,\ \det Q=1\}$
is the special orthogonal group of $(V,\langle\,,\,\rangle)$.

\begin{theorem}
\label{thm:exact-shadow}
Let
\[
 G_\gamma=\{Q\in\SO(V)\,|\, Q(n_\gamma)=n_\gamma\}.
\]
 (i) For every $u\in H_\gamma$ with $\|u\|=1$ and every $\theta\in\R$, we have
 $R_{\gamma,u}(\theta)\in G_\gamma$, and
 $G_\gamma=\{R_{\gamma,u}(\theta)\,|\, u\in H_\gamma, \|u\|=1,\ \theta\in\mathbb R\}$.
\\ (ii) Let $Q\in G_\gamma$ then
 $Q|_{H_\gamma}\in \SO(H_\gamma)$
 and \[
 \rho\colon G_\gamma\to\SO(H_\gamma),\quad \rho(Q)=Q|_{H_\gamma},
 \]
 is a group isomorphism. In particular, $\SO(V)\cong\SO(4)$ and $G_\gamma\cong\SO(3)$
and $G_\gamma$ is the subgroup of all maps in $\SO(V)\cong\SO(4)$ that fix $n_\gamma$.

\end{theorem}

\begin{proof}
$(i)$ Let $u\in H_\gamma$ with $\|u\|=1$, $w\in H_\gamma$ with $\|w\|=1$ and $w\perp u$, and put
$w'=u\boxtimes_\gamma w$. Since $\boxtimes_\gamma$ is the usual cross product in coordinates with respect
to the orthonormal basis $h_\gamma,\mathbf b,\mathbf c$ of $H_\gamma$, the vectors $u,w,w'$ form an
orthonormal basis of $H_\gamma$, and $u\boxtimes_\gamma u=0$, $u\boxtimes_\gamma w'=-w$.
By \eqref{eq:rodrigues-shadow}, we have
\[
\begin{aligned}
 R_{\gamma,u}(\theta)u&=u,\\
 R_{\gamma,u}(\theta)w&=\cos(\theta)w+\sin(\theta)w',\\
 R_{\gamma,u}(\theta)w'&=-\sin(\theta)w+\cos(\theta)w',
\end{aligned}
\]
and $R_{\gamma,u}(\theta)n_\gamma=n_\gamma$ by definition. Hence the matrix of $R_{\gamma,u}(\theta)$ with
respect to the orthonormal basis $u,w,w',n_\gamma$ of $V$ is given by
\[
 \begin{pmatrix}1&0&0&0\\0&\cos(\theta)&-\sin(\theta)&0\\0&\sin(\theta)&\cos(\theta)&0\\0&0&0&1\end{pmatrix},
\]
and orthogonal with determinant $1$. Thus $R_{\gamma,u}(\theta)\in\SO(V)$, and since $R_{\gamma,u}(\theta)$ fixes
$n_\gamma$, it lies in $G_\gamma$. By Euler's Rotation Theorem we also know that
$$G_\gamma=\{R_{\gamma,u}(\theta)\,|\, u\in H_\gamma, \|u\|=1,\ \theta\in\mathbb R\}$$
which proves the remaining assertion.
\\
$(ii)$ Let $Q\in G_\gamma$ and $v\in H_\gamma$. Then $\langle Q(v),n_\gamma\rangle=\langle Q(v),Q(n_\gamma)\rangle
=\langle v,n_\gamma\rangle=0$, so $Q(H_\gamma)\subset H_\gamma=n_\gamma^\perp$, with equality since $Q$ is
injective and $\dim H_\gamma=3$. The map $Q|_{H_\gamma}$ preserves the inner product, and with respect to
the orthonormal basis $h_\gamma,\mathbf b,\mathbf c,n_\gamma$, the matrix of $Q$ is block diagonal with
blocks the matrix of $Q|_{H_\gamma}$ and $1$. Hence $\det Q|_{H_\gamma}=\det Q=1$, so
$\rho(Q)\in\SO(H_\gamma)$. Clearly we also have $\rho(QQ')=\rho(Q)\rho(Q')$.

The map $\rho$ is injective, since $Q$ is determined by $Q|_{H_\gamma}$ and $Qn_\gamma=n_\gamma$, because
$V=H_\gamma\oplus\R n_\gamma$. It is also surjective: let $P\in\SO(H_\gamma)$, then the linear map $Q : V\to V$
with $Q|_{H_\gamma}=P$ and $Q(n_\gamma)=n_\gamma$ lies in $G_\gamma$ and satisfies $\rho(Q)=P$.
Finally, taking matrices with respect to the orthonormal basis $h_\gamma,\mathbf b,\mathbf c$ give a canonical
isomorphism $\SO(H_\gamma)\cong\SO(3)$.
\end{proof}

For $\gamma>0$, identifying each automorphism of $\mathcal A_\gamma$ with its restriction to $V$, Theorem~\ref{thm:automorphism-groups} gives
\[
 \Aut(\mathcal A_\gamma)=\{R_{\gamma,h_\gamma}(\phi):\phi\in\R\}.
\]
Since automorphisms preserve $\times_\gamma$ and $\langle\,,\,\rangle$, they preserve
$\varphi_\gamma$, and we obtain the following chain for $\gamma>0$:
\[
 \begin{array}{ccccc}
 \Aut(\mathcal A_\gamma) & \subsetneq & G_\gamma & \subsetneq & \SO(V)\\[2pt]
 \cong\SO(2) & & \cong\SO(3) & & \cong\SO(4)
 \end{array}
\]

For example, $R_{\gamma,\mathbf b}(t)\mathbf c
=\sin t\,h_\gamma+\cos t\,\mathbf c$ traces a unit circle in
$H_\gamma\cap\mathbf b^\perp=\Span\{h_\gamma,\mathbf c\}$.
Figure~\ref{fig:rotation-obstruction-and-shadow} compares these circles
and the corresponding disks with the norm variation in
Example~\ref{ex:norm-defect}.

\begin{figure}[H]
\centering
\begin{subfigure}[t]{0.32\textwidth}
\centering
\includegraphics[width=\linewidth]{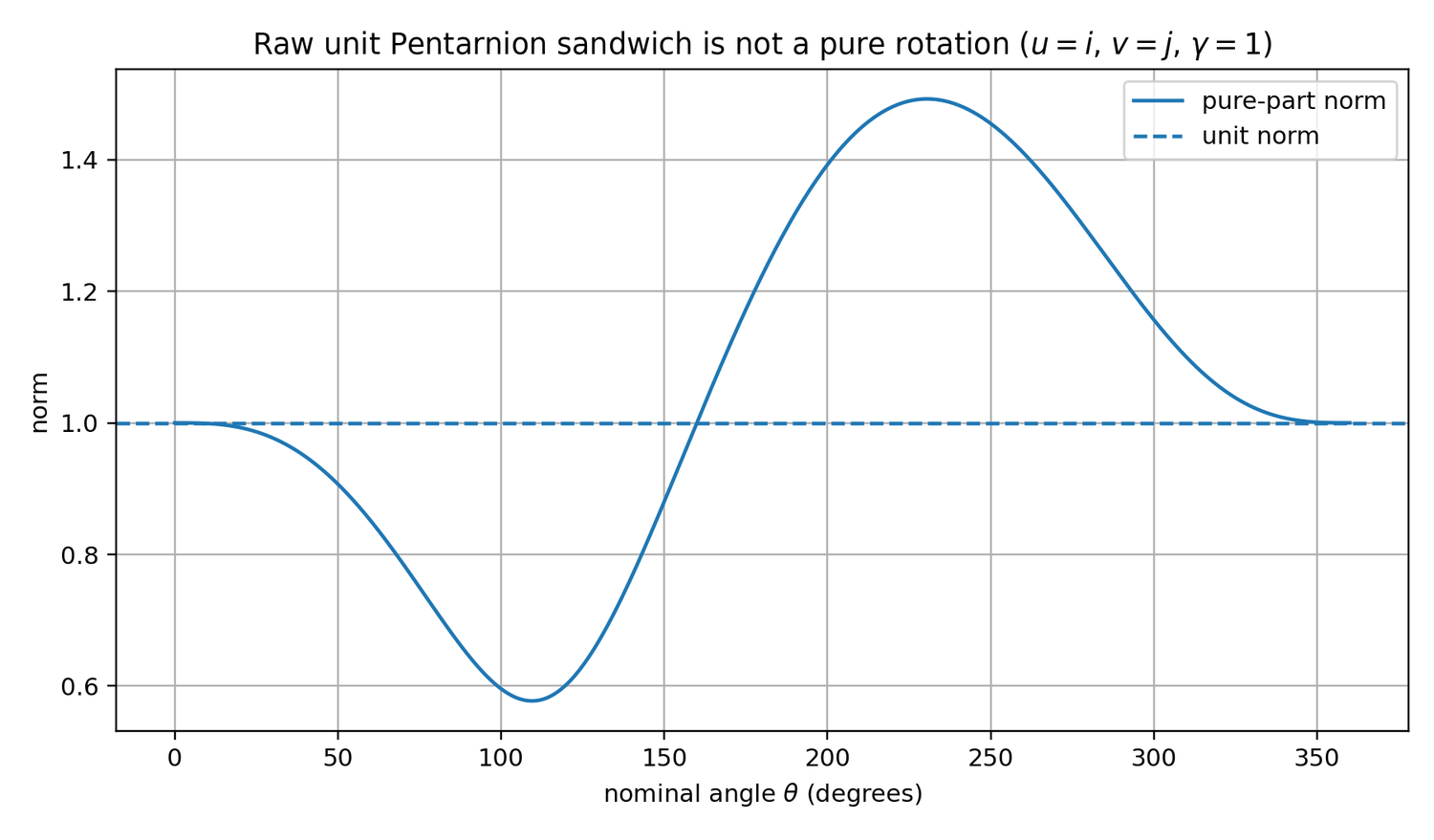}
\caption{Norm of bracketed conjugation.}
\end{subfigure}\hfill
\begin{subfigure}[t]{0.32\textwidth}
\centering
\includegraphics[width=\linewidth]{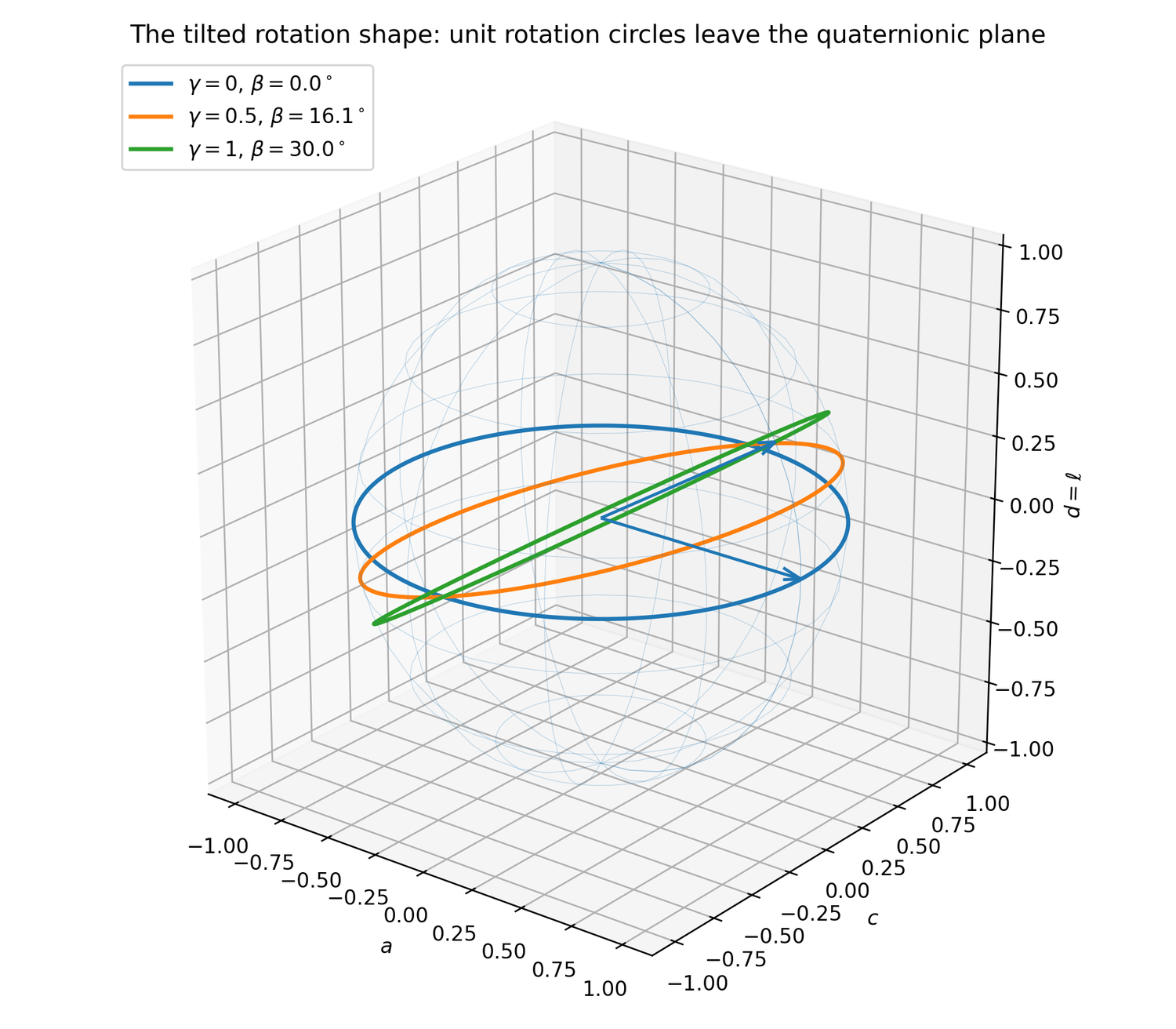}
\caption{Unit circles in $H_\gamma\cap\mathbf b^\perp$.}
\end{subfigure}\hfill
\begin{subfigure}[t]{0.32\textwidth}
\centering
\includegraphics[width=\linewidth]{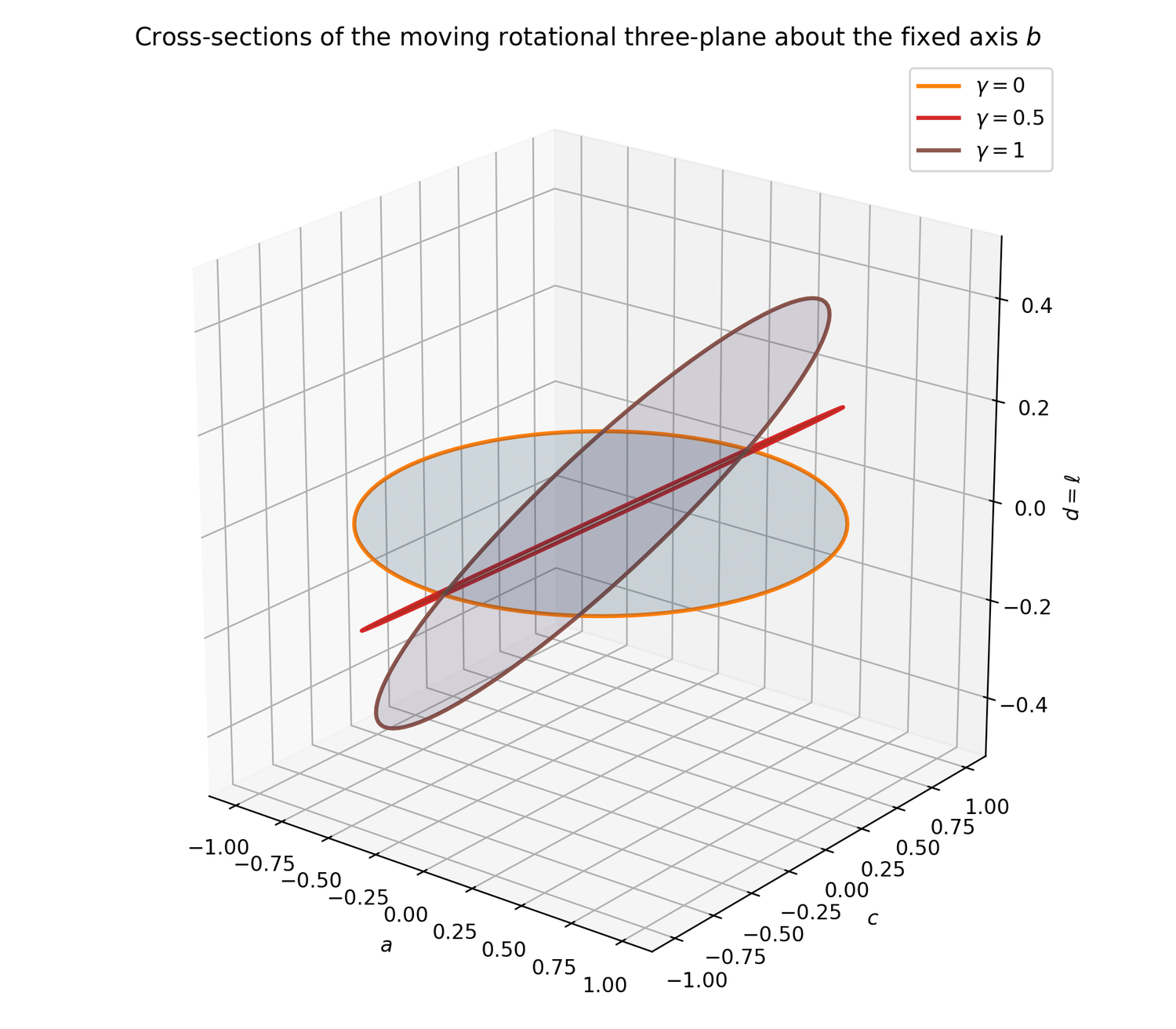}
\caption{Unit disks in $H_\gamma\cap\mathbf b^\perp$.}
\end{subfigure}
\caption{
(A) $\|(r\star_1j)\star_1\bar r\|$ for
$r=\cos(\theta/2)+i\sin(\theta/2)$ and
$0^\circ\le\theta\le360^\circ$; the dashed line is $\|j\|=1$.
(B) The circles $\sin t\,h_\gamma+\cos t\,\mathbf c$ for
$\gamma=0,0.5,1$, with tilt angles $0^\circ,16.1^\circ,30^\circ$.
(C) The disks $\{r h_\gamma+s\mathbf c:r^2+s^2\le1\}$ for the same
parameters. Panels (B) and (C) use $(\mathbf a,\mathbf c,\mathbf d)$
coordinates and show two-dimensional sections of $H_\gamma$.}
\label{fig:rotation-obstruction-and-shadow}
\end{figure}

For orthonormal $x,y\in V$ let $J_{x,y}\in\mathfrak{so}(V)$ be the infinitesimal rotation in the
plane $\Span\{x,y\}$, i.e.\ $J_{x,y}(v)=\langle x,v\rangle y-\langle y,v\rangle x$. The six maps
$J_{x,y}$ with $x,y$ running over pairs from $\mathbf a,\mathbf b,\mathbf c,\mathbf d$ form a basis of
$\mathfrak{so}(V)\cong\mathfrak{so}(4)$. By Theorem~\ref{thm:exact-shadow}, then
\[
 \mathfrak g_\gamma=\Span\{J_{h_\gamma,\mathbf b},\,J_{h_\gamma,\mathbf c},\,J_{\mathbf b,\mathbf c}\}
\]
is the Lie algebra of
$G_\gamma$.

\begin{theorem}[$G_{\gamma_1}$ and $G_{\gamma_2}$  generate $\mathfrak{so}(4)$]
\label{thm:so4-generation}
Let $\mathfrak g_\gamma$ be the Lie algebra of $G_\gamma$.
If $\gamma_1\neq\gamma_2$ are nonnegative, then
 the Lie algebra generated by $\mathfrak g_{\gamma_1}$ and
$\mathfrak g_{\gamma_2}$ is given by
\[
 \operatorname{Lie}\langle\mathfrak g_{\gamma_1},\mathfrak g_{\gamma_2}\rangle
 =\mathfrak{so}(V).
\]
In other words,  $G_{\gamma_1}$ and
$G_{\gamma_2}$ generate $\mathrm{SO}(4)$ when  $\gamma_1\ne\gamma_2$.
\end{theorem}

\begin{proof}
Let $\mathfrak h$ be the Lie algebra generated by $\mathfrak g_{\gamma_1}$ and $\mathfrak g_{\gamma_2}$,
and write $\beta_r=\beta_{\gamma_r}$ for short. Since $h_\gamma=\cos\beta_\gamma\,\mathbf a+\sin\beta_\gamma\,\mathbf d$
and $J_{x,y}$ is linear in $x$,
 we have
\[
 \begin{pmatrix}J_{h_{\gamma_1},\mathbf b}\\ J_{h_{\gamma_2},\mathbf b}\end{pmatrix}
 =\begin{pmatrix}\cos\beta_1&\sin\beta_1\\ \cos\beta_2&\sin\beta_2\end{pmatrix}
 \begin{pmatrix}J_{\mathbf a,\mathbf b}\\ J_{\mathbf d,\mathbf b}\end{pmatrix},
\]
and the left-hand side lies in $\mathfrak h$.
This matrix has determinant  $\cos(\beta_1)\sin(\beta_2)-\sin(\beta_1)\cos(\beta_2)=\sin(\beta_2-\beta_1)\neq0$,
because $\beta_\gamma=\arctan(\gamma/\sqrt3)$ is injective with values in $[0,\pi/2)$. Hence
$J_{\mathbf a,\mathbf b},J_{\mathbf d,\mathbf b}\in\mathfrak h$. The same argument with $\mathbf c$ in place of
$\mathbf b$ gives $J_{\mathbf a,\mathbf c},J_{\mathbf d,\mathbf c}\in\mathfrak h$, and $J_{\mathbf b,\mathbf c}\in\mathfrak g_{\gamma_1}$.
Finally, $[J_{\mathbf a,\mathbf b},J_{\mathbf d,\mathbf b}]=J_{\mathbf a,\mathbf d}$, so $\mathfrak h$ contains all
six basis elements and $\mathfrak h=\mathfrak{so}(V)$. The subgroup generated by the connected groups $G_{\gamma_1}$ and $G_{\gamma_2}$ is a connected immersed Lie subgroup with Lie algebra $\mathfrak h$. Having the full Lie algebra, it is open in $\SO(V)$; connectedness of $\SO(V)$ therefore makes it all of $\SO(V)$.
\end{proof}

\subsection{Time-dependent switching and compositions of rotations}

Let $\gamma=\gamma(t)$ be differentiable and nonnegative.
Let $\beta(t)=\beta_{\gamma(t)}$ be the angle between
$H_{\gamma(t)}$ and $H_0$ in \eqref{eq:tilt-angle}.
Differentiating gives
\begin{equation}
 \dot\beta=\frac{\sqrt3\,\dot\gamma}{3+\gamma^2}.
 \label{eq:beta-dot}
\end{equation}
Let $P_\gamma$ be the orthogonal matrix whose columns are the coordinates of
$h_\gamma,\mathbf b,\mathbf c,n_\gamma$ in the fixed basis
$\mathbf a,\mathbf b,\mathbf c,\mathbf d$. Differentiating these columns gives
\begin{equation}
 \dot P_\gamma P_\gamma^T=\dot\beta J_{\textbf a,\textbf d}.
 \label{eq:rotational-connection}
\end{equation}

If $y(t)\in\R^3$ satisfies $\dot y=F(y,t)$, then its representation
$x=P_{\gamma(t)}\binom{y}{0}\in H_{\gamma(t)}$ in the fixed basis satisfies
\[
 \dot x=\dot\beta J_{\textbf a,\textbf d}x
 +P_\gamma\binom{F(y,t)}0.
\]
Since $P_\gamma$ is orthogonal, $\|x\|=\|y\|$ and
$\frac{\dd}{\dd t}\|x\|^2=2\langle y,F(y,t)\rangle$.
Thus the norm is preserved whenever $\langle y,F(y,t)\rangle=0$.
In particular, this holds for $F(y,t)=\Omega(t)y$ with
$\Omega(t)^T=-\Omega(t)$.
These equations describe motion constrained to $H_{\gamma(t)}$.
We next compose the rotations $R_{\gamma,u}$ directly on all of $V$,
without requiring intermediate vectors to lie in the corresponding
subspaces $H_\gamma$.

\begin{proposition}\label{prop:holonomy}
Let $\gamma_1,\gamma_2>0$ be distinct and put $\Delta=\beta_{\gamma_1}-\beta_{\gamma_2}$, with
$\beta_\gamma$ from \eqref{eq:tilt-angle}. Let $\mathbf e=(\mathbf b+\mathbf c)/\sqrt2$ and
$\mathbf f=(\mathbf b-\mathbf c)/\sqrt2$, and set
\[
 Q_1=R_{\gamma_1,\mathbf b}(\pi)\in G_{\gamma_1},\quad Q_2=R_{\gamma_2,\mathbf f}(\pi)\in G_{\gamma_2}.
\]
Then $K=Q_1Q_2Q_1^{-1}Q_2^{-1}$ preserves the planes $P_1=\Span\{\mathbf a,\mathbf d\}$ and
$P_2=\Span\{\mathbf b,\mathbf c\}$. Oriented by the basis $\mathbf a,\mathbf d$, it acts on $P_1$,  as the
rotation by $4\Delta$, and on $P_2$ as $-\mathrm{id}$. In particular, $Kv=v$ implies $v=0$, so
$K\notin G_\gamma$ for every $\gamma\geq 0$.

\end{proposition}

\begin{proof}
By Equation \eqref{eq:rodrigues-shadow}, we know that
 $R_{\gamma,u}(\pi)v=-v+2\langle u,v\rangle u$ for  $\theta=\pi$ and
$v\in H_\gamma$. So $R_{\gamma,u}(\pi)$ fixes $u$ and $n_\gamma$ and acts as $-\mathrm{id}$ on
$u^\perp\cap H_\gamma$. Since $H_\gamma=\Span\{h_\gamma,\mathbf b,\mathbf c\}=\Span\{h_\gamma,\mathbf e,\mathbf f\}$,
the map $Q_1$ fixes $\mathbf b$ and $n_{\gamma_1}$ and negates $\mathbf c$ and $h_{\gamma_1}$, and
$Q_2$ fixes $\mathbf f$ and $n_{\gamma_2}$ and negates $\mathbf e$ and $h_{\gamma_2}$.

By \eqref{eq:h-n-main} and \eqref{eq:tilt-angle},
\[
 n_\gamma=\cos(\beta_\gamma)\mathbf d-\sin(\beta_\gamma)\mathbf a,\quad
 h_\gamma=\cos(\beta_\gamma)\mathbf a+\sin(\beta_\gamma)\mathbf d .
\]
Hence both $Q_1$ and $Q_2$ preserve $P_1$ and $P_2$. On $P_1$, the map $Q_r$ is the reflection in the
line $\R n_{\gamma_r}$, which makes the angle $\beta_{\gamma_r}$ with $\R\mathbf d$. On $P_2$, $Q_1$ is the
reflection in $\R\mathbf b$ and $Q_2$ the reflection in $\R\mathbf f$, and these lines make the angle
$\pi/4$. Since the composition of the reflections in two lines meeting at angle $\alpha$ is the
rotation by $2\alpha$, the map $Q_1Q_2$ acts on $P_1$ as the rotation by $2\Delta$ and on $P_2$ as the
rotation by $\pi/2$.

As $Q_1^2=Q_2^2=I$, we have $K=(Q_1Q_2)^2$, which acts on $P_1$ as the rotation by $4\Delta$ and on
$P_2$ as the rotation by $\pi$, i.e.\ as $-\mathrm{id}$. Since $\beta_\gamma=\arctan(\gamma/\sqrt3)$ is
strictly increasing with values in $[0,\pi/2)$, we have $0<|4\Delta|<2\pi$, so the rotation by $4\Delta$
fixes no nonzero vector of $P_1$. Writing $v=v_1+v_2$ with $v_r\in P_r$, the equation $Kv=v$ gives
$Kv_1=v_1$ and $Kv_2=v_2$, hence $v=0$. Since every element of $G_\gamma$ fixes $n_\gamma\neq0$,
$K\notin G_\gamma$.
\end{proof}

Whether a composition of such rotations is a simple or a double
rotation depends on the axes and rotation angles.
To distinguish the compositions in Figure~\ref{fig:holonomy-panel}, set
$U_\varepsilon=R_{0,\mathbf c}(\varepsilon)$,
$W_\varepsilon=R_{1,\mathbf c}(\varepsilon)$ and
$K_{\mathbf c}(\varepsilon)=W_\varepsilon^{-1}U_\varepsilon^{-1}
W_\varepsilon U_\varepsilon$.

\begin{proposition}
\label{prop:finite-rotation-compositions}
\begin{enumerate}[label=\textnormal{(\roman*)}]
\item For $0\leq\varepsilon\leq\pi/2$, the restriction of
$K_{\mathbf c}(\varepsilon)$ to $\mathbf c^\perp$ has principal
rotation angle
\[
\Phi(\varepsilon)
=4\arcsin\!\left(\tfrac12\sin^2(\varepsilon/2)\right).
\]
Thus $K_{\mathbf c}(0)=I$, and
$K_{\mathbf c}(\varepsilon)$ is a simple rotation for
$0<\varepsilon\leq\pi/2$. Moreover,
$\Phi(\varepsilon)=\varepsilon^2/2+O(\varepsilon^4)$
as $\varepsilon\to0$.

\item The map
$Q=R_{1,\mathbf c}(\pi/2)R_{0,\mathbf b}(\pi/2)$ satisfies
$\det(Q-I)=1$. Hence it has no nonzero fixed vector and acts
as a rotation through a nonzero angle on each of two
mutually orthogonal invariant planes.
\end{enumerate}
\end{proposition}

\begin{proof}
\textnormal{(i)}
Both $U_\varepsilon$ and $W_\varepsilon$ fix $\mathbf c$ and act on
$\mathbf c^\perp=\Span\{\mathbf a,\mathbf b,\mathbf d\}$ as rotations by $\varepsilon$ about the axes $n_0$ and $n_1$, respectively. From \eqref{eq:h-n-main}, we have
\[
n_0=\mathbf d,\qquad
n_1=-\frac12\mathbf a+\frac{\sqrt3}{2}\mathbf d,\qquad
\langle n_0,n_1\rangle=\frac{\sqrt3}{2}.
\]
The angle between these axes is therefore $\pi/6$.
Identify the oriented Euclidean space $\mathbf c^\perp$ with
orthonormal basis $(\mathbf a,\mathbf b,\mathbf d)$ with the pure
Hamilton quaternions. (Note that this is an auxiliary representation of
ordinary three-dimensional rotations; all quaternion products in this part of the proof are Hamilton products, not
$\star_\gamma$.)

Let $u,v$ be the unit pure quaternions corresponding to $n_0,n_1$.
Put $C=\cos(\varepsilon/2)$ and $S=\sin(\varepsilon/2)$.
The rotations $U_\varepsilon$ and $W_\varepsilon$ are represented
by $r_0=C+Su$ and $r_1=C+Sv$, respectively, so
$K_{\mathbf c}(\varepsilon)|_{\mathbf c^\perp}$ is represented by
\[
q_K=r_1^{-1}r_0^{-1}r_1r_0.
\]
Set $v'=r_0^{-1}vr_0$. Then
\[
q_K=(C-Sv)(C+Sv'),\qquad
\operatorname{Sc}(q_K)=C^2+S^2\langle v,v'\rangle,
\]
since the scalar part of the product of two pure quaternions
is the negative of their Euclidean inner product.
The vector $v'$ is obtained by rotating $v$ through
$-\varepsilon$ about $u$. Writing
$v=\frac{\sqrt3}{2}u+v_\perp$, where $v_\perp\perp u$ and
$\|v_\perp\|^2=\frac14$, gives
\[
\langle v,v'\rangle=\frac34+\frac14\cos\varepsilon.
\]
Consequently, we obtain that
\[
\operatorname{Sc}(q_K)
=C^2+S^2\left(\frac34+\frac14\cos\varepsilon\right)
=C^2+S^2\left(1-\frac12S^2\right)
=1-\frac12S^4.
\]
For $0\leq\varepsilon\leq\pi/2$, this scalar part is at least
$7/8$. The principal rotation angle $\Phi\in[0,\pi]$ thus satisfies
\[
\cos\frac{\Phi}{2}
=1-\frac12S^4
=\cos\!\left(2\arcsin\frac{S^2}{2}\right).
\]
Both half-angles lie in $[0,\pi/2]$, where cosine is injective,
so $\Phi=4\arcsin(S^2/2)$.

At $\varepsilon=0$, both factors are the identity.
For $0<\varepsilon\leq\pi/2$, the angle $\Phi$ is nonzero,
so the restriction to $\mathbf c^\perp$ fixes exactly its
rotation axis. Together with the fixed direction $\mathbf c$,
this gives a two-dimensional fixed subspace in $V$; hence
$K_{\mathbf c}(\varepsilon)$ is a simple rotation.
Finally,
$\sin^2(\varepsilon/2)=\varepsilon^2/4+O(\varepsilon^4)$ and
$\arcsin z=z+O(z^3)$ give the stated small-angle expansion.

\textnormal{(ii)}
By \eqref{eq:rodrigues-shadow}, we have
\[
\begin{aligned}
R_{0,\mathbf b}(\pi/2)\mathbf a&=-\mathbf c,
& R_{0,\mathbf b}(\pi/2)\mathbf c&=\mathbf a,\\
R_{1,\mathbf c}(\pi/2)h_1&=\mathbf b,
& R_{1,\mathbf c}(\pi/2)\mathbf b&=-h_1.
\end{aligned}
\]
The first rotation fixes $\mathbf b,\mathbf d$, and the second
fixes $\mathbf c,n_1$.
Using $h_1=(\sqrt3\,\mathbf a+\mathbf d)/2$ and the expression
for $n_1$ above, the matrix of $Q$ in the basis
$(\mathbf a,\mathbf b,\mathbf c,\mathbf d)$ is
\[
[Q]=
\begin{pmatrix}
0 & -\frac{\sqrt3}{2} & \frac14 & -\frac{\sqrt3}{4}\\
0 & 0 & \frac{\sqrt3}{2} & \frac12\\
-1 & 0 & 0 & 0\\
0 & -\frac12 & -\frac{\sqrt3}{4} & \frac34
\end{pmatrix}.
\]
Subtracting the third row of $[Q]-I_4$ from its first row and
expanding along the first column yields
\[
\det(Q-I)
=-\det
\begin{pmatrix}
-\frac{\sqrt3}{2} & \frac54 & -\frac{\sqrt3}{4}\\
-1 & \frac{\sqrt3}{2} & \frac12\\
-\frac12 & -\frac{\sqrt3}{4} & -\frac14
\end{pmatrix}
=-\left(0-\frac58-\frac38\right)=1.
\]
Thus $Q$ has no eigenvalue $1$. Since $Q\in\mathrm{SO}(V)$,
it admits a decomposition into rotations on two mutually
orthogonal invariant planes. Neither rotation angle is zero,
as otherwise every vector in the corresponding plane would
be fixed.
\end{proof}

Panels (A) and (B) illustrate $K_{\mathbf c}(\varepsilon)$,
whereas panel (C) shows iterates of $Q$; neither is the
commutator of Proposition~\ref{prop:holonomy}.
The plotting code is available with
the supporting data (Appendix~\ref{app:reproducibility}).

\begin{figure}[H]
\centering
\begin{subfigure}[t]{0.32\textwidth}
\centering
\includegraphics[width=\linewidth]{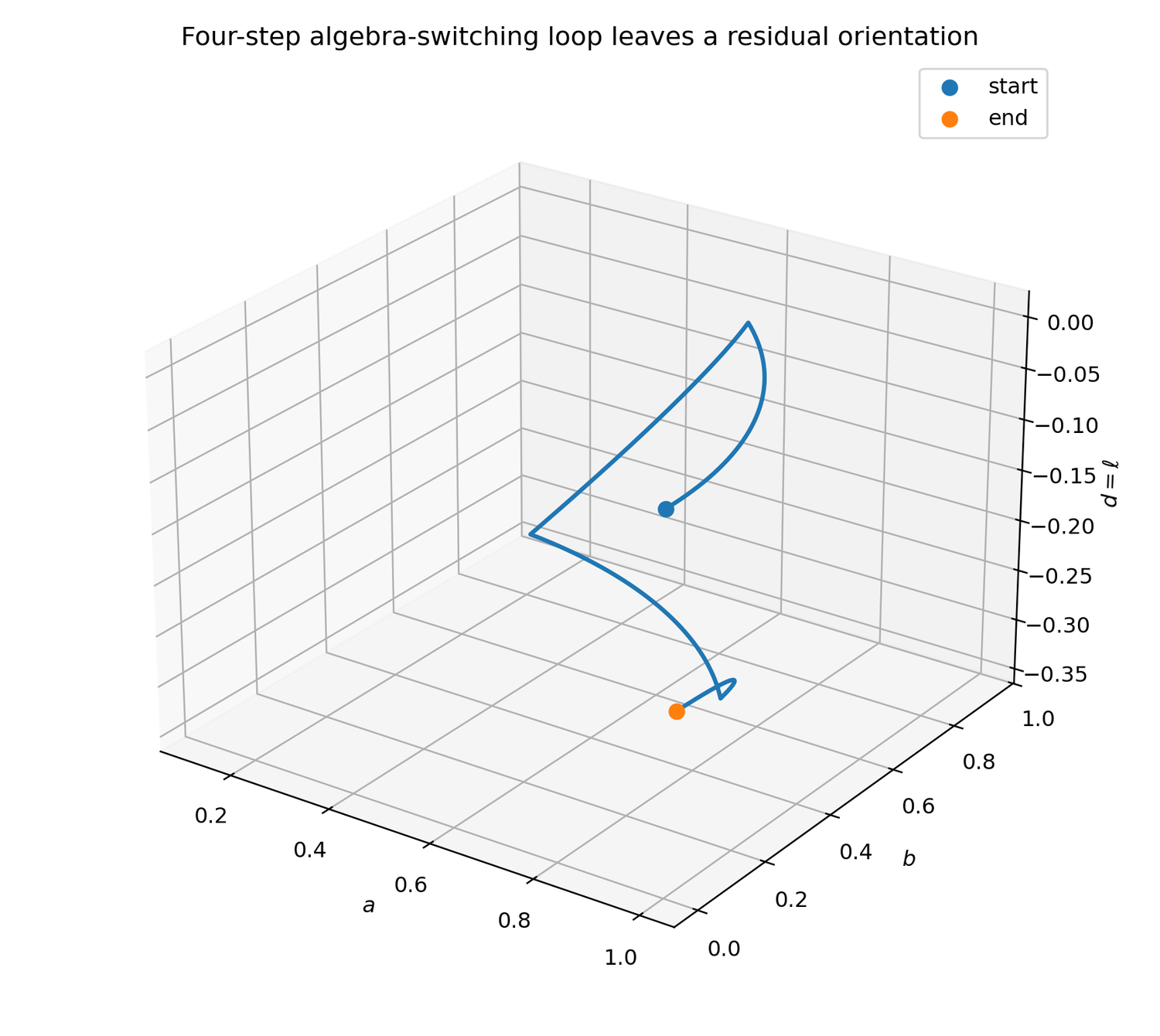}
\caption{Four rotations about $\mathbf c$.}
\end{subfigure}\hfill
\begin{subfigure}[t]{0.32\textwidth}
\centering
\includegraphics[width=\linewidth]{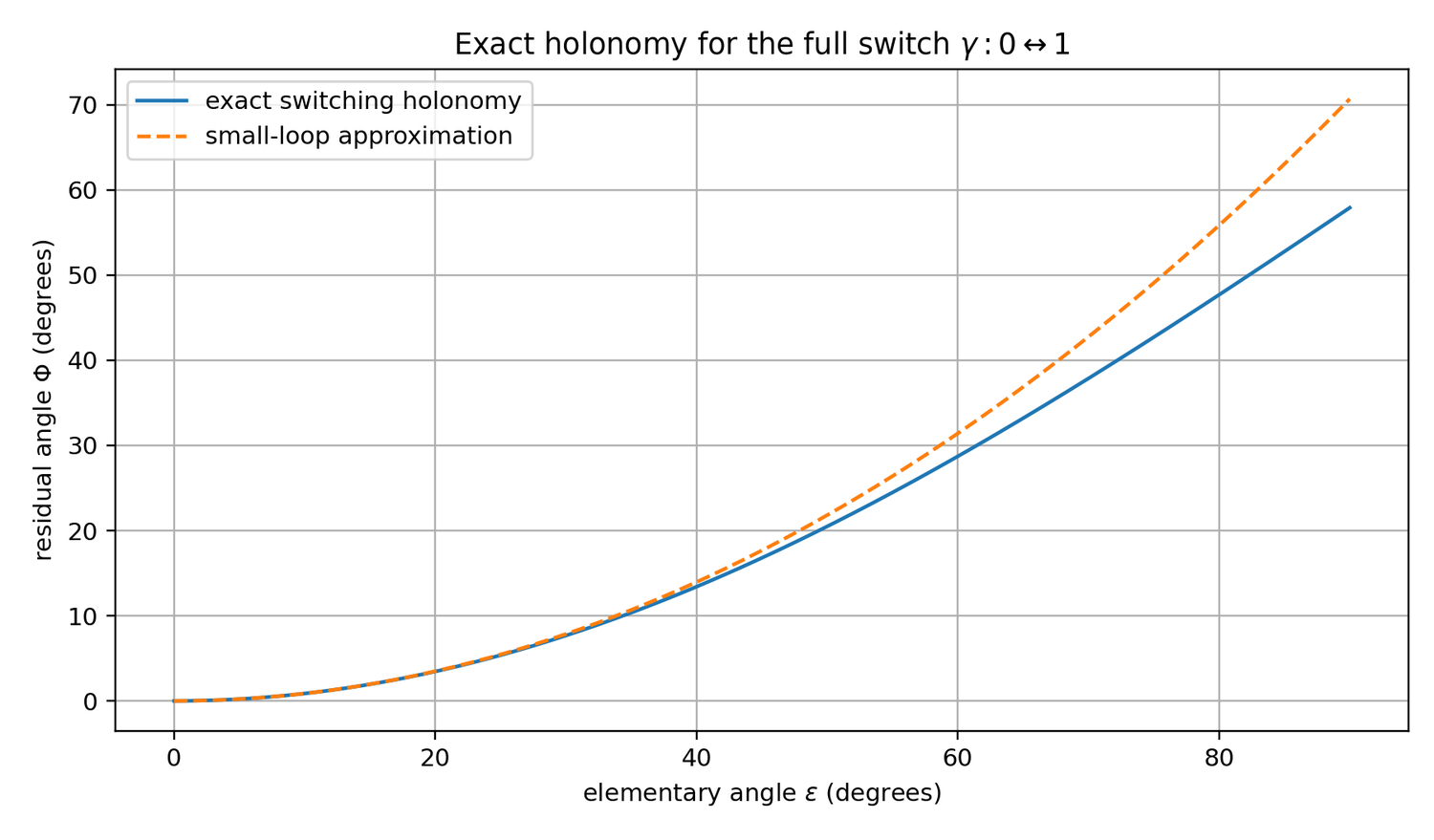}
\caption{Residual angle $\Phi(\varepsilon)$}.
\end{subfigure}\hfill
\begin{subfigure}[t]{0.32\textwidth}
\centering
\includegraphics[width=\linewidth]{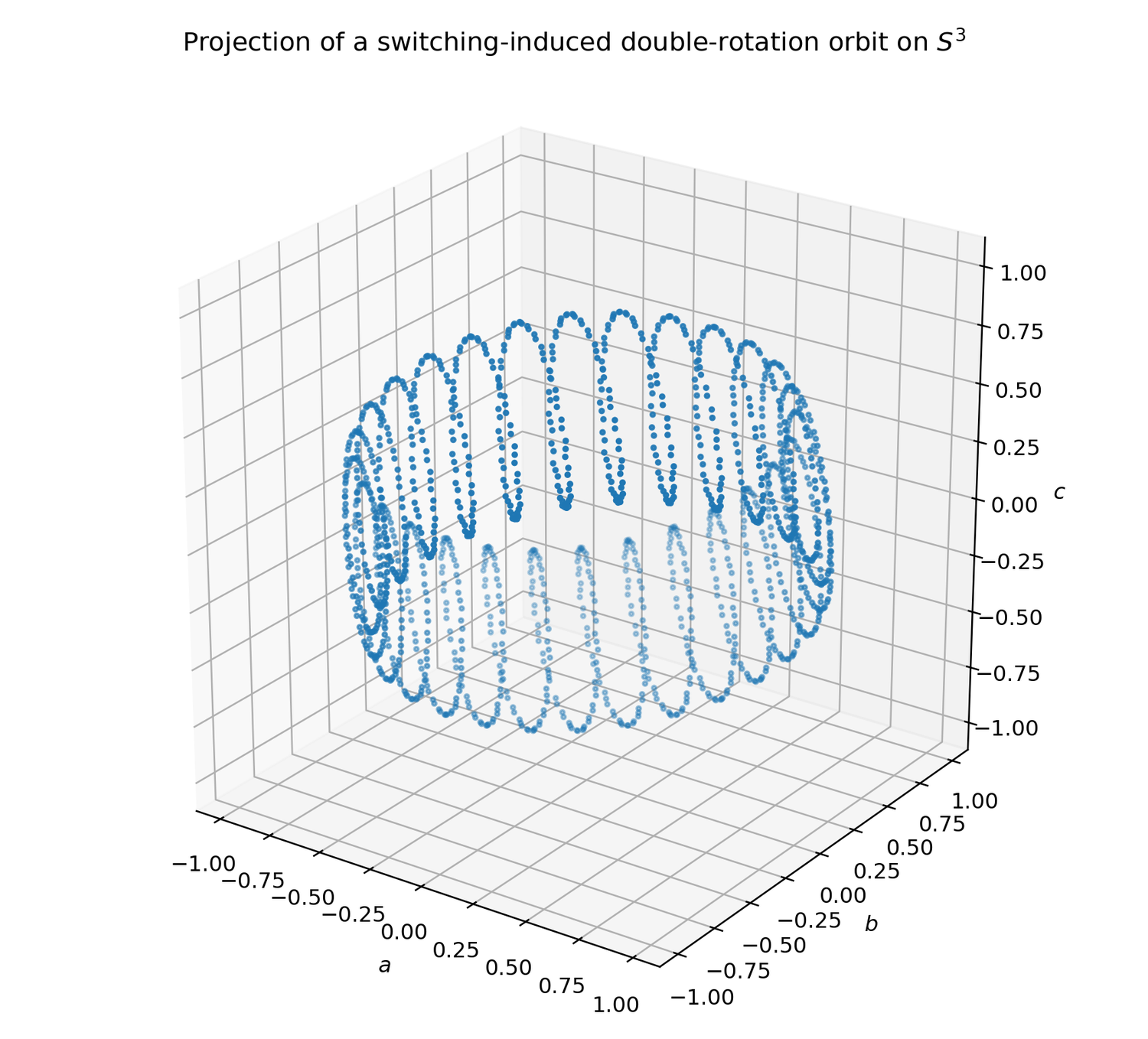}
\caption{Iterates of the two-step map $Q$.}
\end{subfigure}
\caption{
Two finite rotation compositions.
(A) Starting from $\mathbf a$, apply $U_\varepsilon,W_\varepsilon,
U_\varepsilon^{-1},W_\varepsilon^{-1}$ successively, with
$\varepsilon=\pi/4$; the path is shown in
$(\mathbf a,\mathbf b,\mathbf d)$ coordinates.
(B) The angle of $K_{\mathbf c}(\varepsilon)$ for
$0\le\varepsilon\le\pi/2$; the dashed curve is $\varepsilon^2/2$.
Both axes are displayed in degrees.
(C) The first $1600$ iterates $Q^n x_0$, $0\le n<1600$,
for $x_0=(\mathbf a+\mathbf b+\mathbf c+\mathbf d)/2$, projected onto
$(\mathbf a,\mathbf b,\mathbf c)$.}
\label{fig:holonomy-panel}
\end{figure}
\FloatBarrier

\section{Geometry depending on ordering and parentheses}
\label{sec:bracket-geometry}

We now demonstrate how different bracketing changes the geometry. In \cite{PoshtkohiPumplunSubmitted}, we studied the six
homogeneous quadratic maps of the Pentarnion algebra $\Pent$ obtained by multiplying a fixed element
with two copies of the variable $p$, in all possible orders and bracketings. Here we vary $\gamma$ and take
the fixed element to be $i$ and define
\begin{align}
M_{0,\gamma}(p)&=i\star_\gamma(p\star_\gamma p),&
M_{1,\gamma}(p)&=(p\star_\gamma p)\star_\gamma i,\label{eq:six-maps-01}\\
M_{2,\gamma}(p)&=(i\star_\gamma p)\star_\gamma p,&
M_{3,\gamma}(p)&=p\star_\gamma(i\star_\gamma p),\label{eq:six-maps-23}\\
M_{4,\gamma}(p)&=(p\star_\gamma i)\star_\gamma p,&
M_{5,\gamma}(p)&=p\star_\gamma(p\star_\gamma i).\label{eq:six-maps-45}
\end{align}
These maps define pairwise different nonlinear transformations of $\SwitchAlg_\gamma$.

\subsection{Jacobians and conjugacy classes}\label{sec:7.1}

A map $M:\R^{n}\to\R^{n}$ is called \emph{homogeneous quadratic} if there exists a  bilinear map $\beta:\R^{n}\times \R^{n}\to\R^{n}$ such that $M(x)=\beta(x,x)$.
Then $DM(p)(h)=\beta(p,h)+\beta(h,p)$, so the entries
of $DM(p)$ are linear in $p$. Consequently $\det DM(p)$ is either identically zero or a homogeneous
polynomial of degree $n$, and every $p\neq0$ with $M(p)=0$ is a \emph{critical point}, since
$DM(p)(p)=\beta(p,p)+\beta(p,p)=2M(p)=0$. Moreover, we see that $DM(-p)=-DM(p)$. Thus when  $n>1$ is odd, $\det DM(-p)=-\det DM(p)$.

Let
$p,q\in\R^n$ be orthonormal, $c(t)=\cos(\pi t) p+\sin(\pi t)q$ for $t\in[0,1]$ then this immediately implies that
the continuous function $\det DM(c(t))$ must vanish at some $t_0\in[0,1]$, and $c(t_0)\neq0$ is a critical point of $M$. We summarise our observations as follows.

\begin{lemma}\label{thm:degree-n}
Let $M:\R^n\to \R^n$  be a homogeneous quadratic map
with $M(p)=\beta(p,p)$ for a bilinear map $\beta:\R^n\times \R^n\to \R^n$. Then
$DM(p)h=\beta(p,h)+\beta(h,p)$, and every entry of the matrix of $DM(p)$ is a linear form in $p$.
Hence $\det DM(p)$ is either identically zero or a homogeneous polynomial of degree $n$ in $p$.
Moreover, if $M(p)=0$ for some $p\neq0$, then $\det DM(p)=0$.
\end{lemma}

For $a\in\SwitchAlg_\gamma$, write
\[
 L_a^{\gamma}(q):=a\star_\gamma q,
 \qquad
 R_a^{\gamma}(q):=q\star_\gamma a.
\]
These linear  maps can be written as $5\times5$ real matrices with respect to  the canonical basis.  We obtain the six Jacobians
\begingroup
\allowdisplaybreaks[1]
\begin{align}
D M_{0,\gamma}(p)&=L_i^{\gamma}
 \left(L_p^{\gamma}+R_p^{\gamma}\right),\label{eq:J0}\\
D M_{1,\gamma}(p)&=R_i^{\gamma}
 \left(L_p^{\gamma}+R_p^{\gamma}\right),\label{eq:J1}\\
D M_{2,\gamma}(p)&=R_p^{\gamma}L_i^{\gamma}
 +L_{i\star_\gamma p}^{\gamma},\label{eq:J2}\\
D M_{3,\gamma}(p)&=R_{i\star_\gamma p}^{\gamma}
 +L_p^{\gamma}L_i^{\gamma},\label{eq:J3}\\
D M_{4,\gamma}(p)&=R_p^{\gamma}R_i^{\gamma}
 +L_{p\star_\gamma i}^{\gamma},\label{eq:J4}\\
D M_{5,\gamma}(p)&=R_{p\star_\gamma i}^{\gamma}
 +L_p^{\gamma}R_i^{\gamma}.\label{eq:J5}
\end{align}
\endgroup
Up to numbering, these are the maps $F_1,\dots,F_6$ of Appendix~\ref{app:coordinate-matrices} (with $A=i$).
Their Jacobian determinant polynomials are
\begin{equation}
 Q_r(p;\gamma)=\det D M_{r,\gamma}(p),
 \qquad r=0,\ldots,5.
 \label{eq:Qr-definition}
\end{equation}
Then we know by the above that each of the maps
\eqref{eq:six-maps-01}--\eqref{eq:six-maps-45} on $\mathcal A_\gamma\cong\R^5$ has nonzero critical points. Each $M_{r,\gamma}$ is a homogeneous quadratic map, with $\beta(p,q)$ obtained from the product of $i,p,p$ by
replacing the second $p$ by $q$. Hence each $Q_r(\cdot;\gamma)$ is either identically zero or a
homogeneous quintic.

For $\gamma>0$, $Q_r(\cdot;\gamma)$ is nonzero for $r=0,\dots,5$.
For $i\star_\gamma(p\star_\gamma p)$,
$(i\star_\gamma p)\star_\gamma p$,
$p\star_\gamma(p\star_\gamma i)$ and
$(p\star_\gamma p)\star_\gamma i$, the Jacobian at
$p=1$ is $2L_i$ or $2R_i$, and
$\det L_i=-\det R_i=\sqrt2\,\kappa_\gamma\sigma_\gamma\neq0$.
For the two maps with $i$ in the middle,
an exact symbolic calculation using \eqref{eq:J3} and
\eqref{eq:J4} gives
\[
Q_3(1+\ell;\gamma)=Q_4(1+\ell;\gamma)
=4\kappa_\gamma\sigma_\gamma
=\frac{4\gamma}{1+\gamma^2}>0.
\]
The evaluated Jacobian matrices and their exact determinant
verification are supplied with the accompanying code
(Appendix~\ref{app:reproducibility}).

For $\gamma=0$, however, $\det L_i=\det R_i=0$, since
$i\star_0(j-k+\sqrt2\,\ell)=0$, so
$Q_r(\cdot;0)\equiv0$ for $i\star_0(p\star_0p)$ and
$(p\star_0p)\star_0i$.

For the canonical conjugation we have
$ \overline{p\star_\gamma q}=\bar q\star_\gamma\bar p$ and $\bar i=-i$. We obtain
$$ \overline{M_{0,\gamma}(p)}=-M_{1,\gamma}(\bar p),\quad \overline{M_{2,\gamma}(p)}=-M_{5,\gamma}(\bar p), \quad \overline{M_{3,\gamma}(p)}=-M_{4,\gamma}(\bar p).$$
Differentiation and $n=5$ imply that
$$ Q_1(p;\gamma)=-Q_0(\bar p;\gamma),
 \quad
 Q_5(p;\gamma)=-Q_2(\bar p;\gamma),
 \quad
 Q_4(p;\gamma)=-Q_3(\bar p;\gamma).$$
Hence the six maps define at most three critical geometries up to conjugation.

Let $S(p)=p\star_\gamma p$. Then $M_{0,\gamma}=L_i^\gamma\circ S$ and $M_{1,\gamma}=R_i^\gamma\circ S$, so
\[
 D M_{0,\gamma}(p)=L_i^{\gamma}DS(p),
 \quad
 D M_{1,\gamma}(p)=R_i^{\gamma}DS(p) \quad DS(p)=L_p^\gamma+R_p^\gamma.
\]
If left and right multiplication by $i$ are nonsingular, both maps have the same critical set as $S$.

For $p=p_0+u$ and $x=x_0+v$ with $u,v\in V$, we have
$DS(p)x=2p_0x_0+2x_0u+2p_0v-2\langle u,v\rangle$, so that
\[
 DS(p)=2\begin{pmatrix}p_0&-u^T\\ u&p_0I_4\end{pmatrix},\quad
 \det DS(p)=32p_0^3(p_0^2+\|u\|^2)=32\,p_0^3N(p)
\]
in the canonical basis.
Moreover, $\det L_i^\gamma=\sqrt2\kappa_\gamma\sigma_\gamma=\frac{\sqrt2\gamma}{1+\gamma^2}$ and
$R_i^\gamma=L_i^\gamma$ on $\Span\{1,i\}$, $R_i^\gamma=-L_i^\gamma$ on $\Span\{j,k,\ell\}$, so
$\det R_i^\gamma=-\det L_i^\gamma$. Hence we obtain
$$Q_0(p;\gamma)
 =\frac{32\sqrt2\,\gamma}{1+\gamma^2}
 p_0^3\N(p),
 \qquad
 Q_1=-Q_0.$$
 For all $\gamma>0$, thus the nonzero critical points of $M_{0,\gamma}$ and $M_{1,\gamma}$ are exactly the
nonzero $p\in V$, and on $V$ both maps are explicit: since $u\star_\gamma u=-\|u\|^2$ for $u\in V$,
\[
 M_{0,\gamma}(u)=M_{1,\gamma}(u)=-\|u\|^2\,i .
\]
In particular, the unit sphere of $V$ is mapped to the single point $-i$.

\subsection{The image of the unit circle under the maps $ M_{r,\gamma}$}

We now describe the unit circle as
\begin{equation}
 p(\theta)=\cos\theta\,j+\sin\theta\,k.
 \label{eq:jk-circle}
\end{equation}
Because $j^2=k^2=-1$ and $j\star_\gamma k=-k\star_\gamma j$, we have $ p(\theta)^2=-1$ and therefore  $ M_{0,\gamma}(p(\theta))=-i$ for all $\theta$.
Let $c=\kg$ and $s=\sg$. A direct computation yields
\begin{equation}
 M_{2,\gamma}(p(\theta))
=z_\gamma+a_\gamma\cos(2\theta)+b_\gamma\sin(2\theta),
 \label{eq:ellipse-param}
\end{equation}
where
\[
 z_\gamma=
 \left(-c^2,-\frac{s}{2\sqrt2},\frac{s}{2\sqrt2},-sc\right),
\]
\[
 a_\gamma=
 \left(\frac{s}{\sqrt2},\frac{s}{2\sqrt2},\frac{s}{2\sqrt2},0\right),
 \quad
 b_\gamma=
 \left(0,\frac{s}{2\sqrt2},-\frac{s}{2\sqrt2},0\right),
\]
working with the basis $(i,j,k,\ell)$.
Furthermore, $a_\gamma$ and  $b_\gamma$ are orthogonal and
\[
 \norm{a_\gamma}=\frac{\sqrt3}{2}s,
 \qquad
 \norm{b_\gamma}=\frac12s.
\]

We now see how much the solution sets depend on our choice of $\gamma$.

\begin{theorem}[Associator ellipse]
\label{thm:associator-ellipse}
For $\gamma>0$, $M_{2,\gamma}$ maps the  circle \ref{eq:jk-circle}  to an  ellipse with semi-axes
\[
 \frac{\sqrt3\,\gamma}{2\sqrt{1+\gamma^2}}
 \text { and }
 \frac{\gamma}{2\sqrt{1+\gamma^2}},
\]
and constant $\ell$-coordinate $ -\frac{\gamma}{1+\gamma^2}.$
At $\gamma=0$ the image is the point $-i$.  Since
\[
 M_{2,\gamma}(p)-M_{0,\gamma}(p)=[i,p,p]_\gamma,
\]
the ellipse is the translate by $-i$ of the associator image along the circle.
\end{theorem}

On this symmetric circle the six maps reduce to
\[
 M_0=M_1=-i,
 \qquad
 M_2=M_5,
 \qquad
 M_3=M_4=-M_2,
\]
since we have $M_2+M_3=wp+pw=-2\langle w,p\rangle=0$
for $p$ on the circle ($w=i\star_\gamma p$ is pure with $\langle w,p\rangle=0$), and $M_5(p)=-\overline{M_2(p)}=M_2(p)$, $M_4(p)=-\overline{M_3(p)}=M_3(p)$, since $M_2(p)$ and $M_3(p)$ are pure and  $\bar p=-p$.

\begin{figure}[t]
\centering
\begin{subfigure}[t]{0.32\textwidth}
\centering
\includegraphics[width=\textwidth]{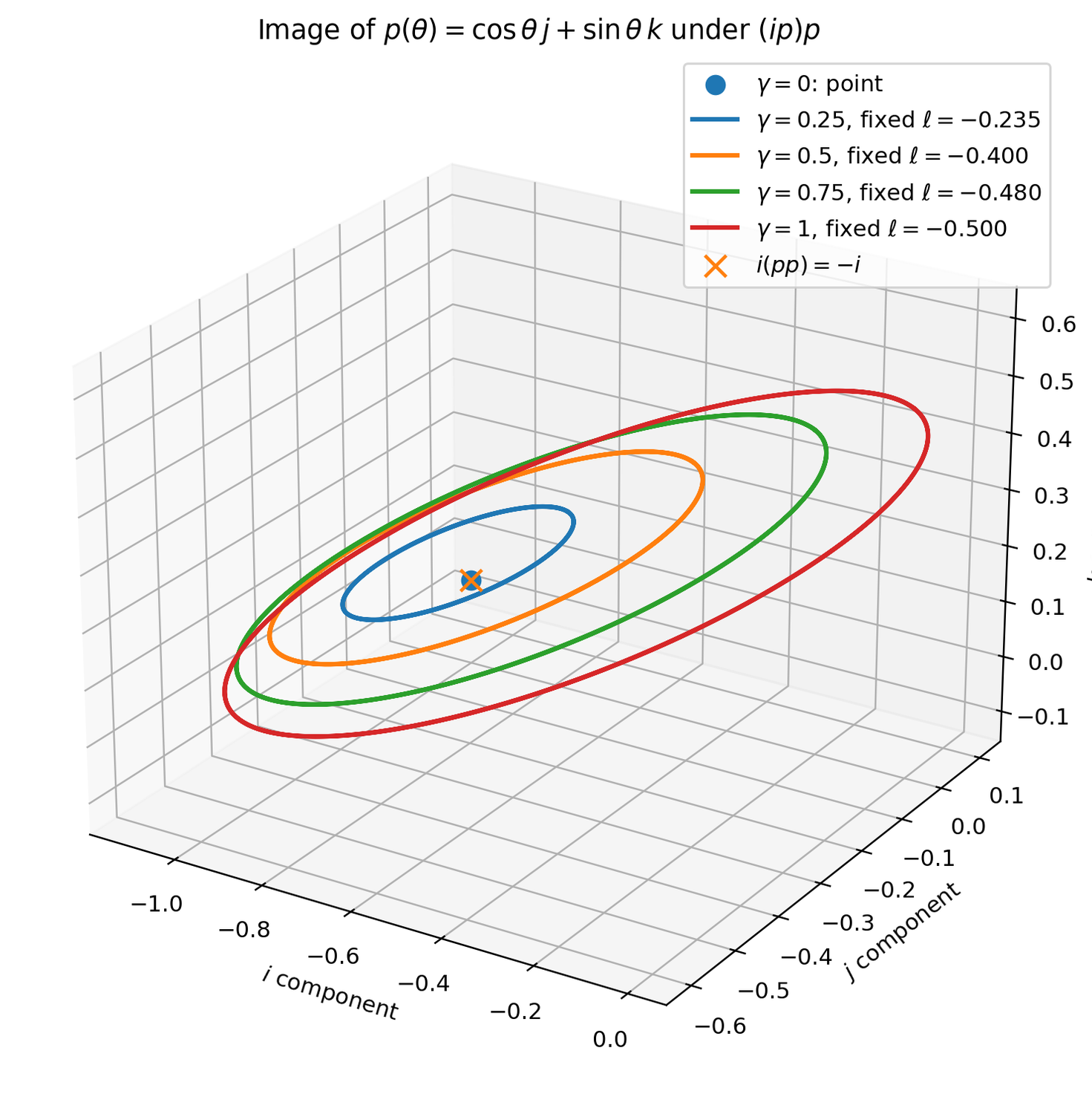}
\caption{}
\end{subfigure}\hfill
\begin{subfigure}[t]{0.32\textwidth}
\centering
\includegraphics[width=\textwidth]{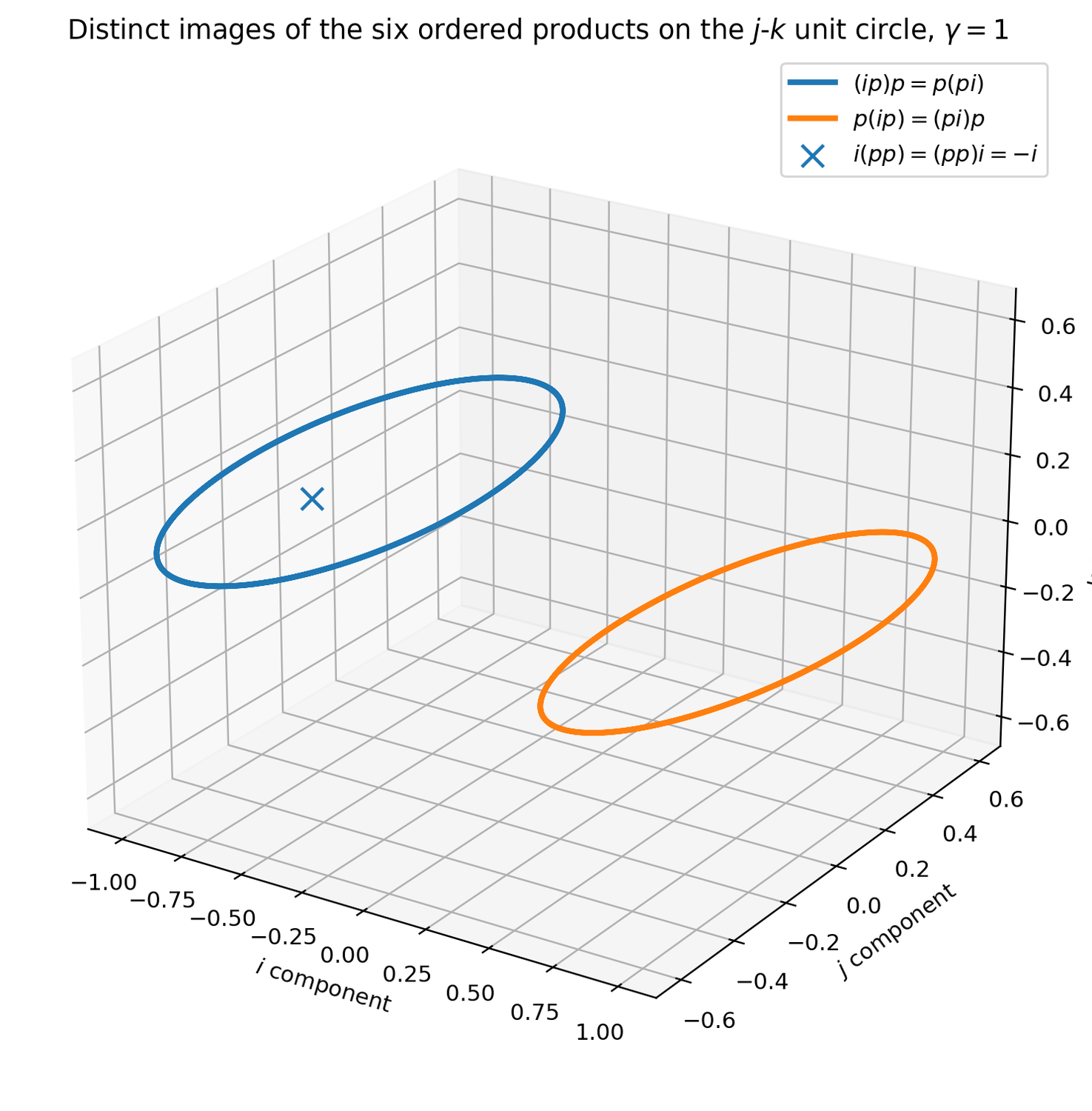}
\caption{}
\end{subfigure}\hfill
\begin{subfigure}[t]{0.32\textwidth}
\centering
\includegraphics[width=\textwidth]{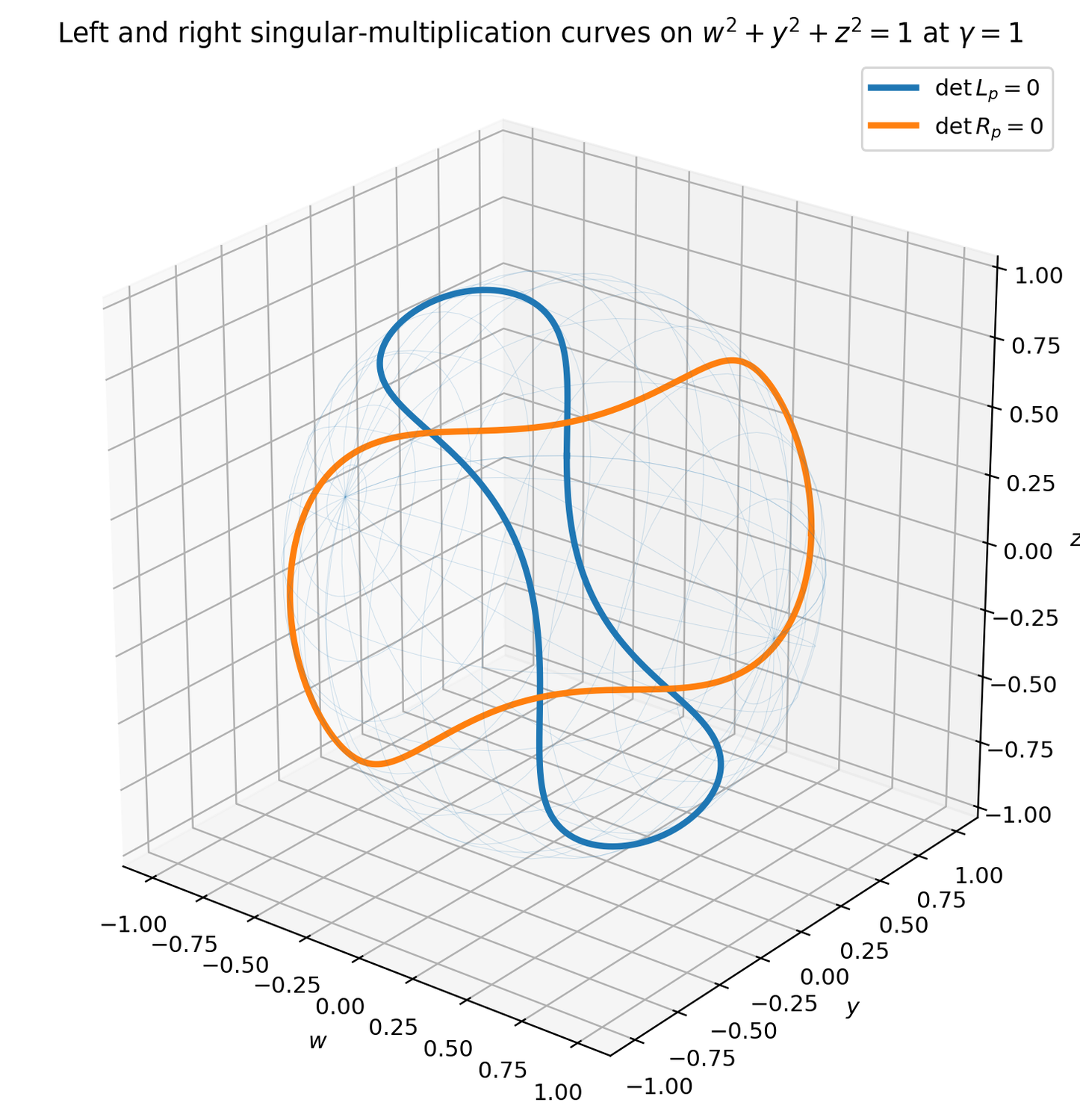}
\caption{}
\end{subfigure}
\caption{
(a) Images of the circle \eqref{eq:jk-circle} under $M_{2,\gamma}$
for $\gamma=0,0.25,0.5,0.75,1$: a point for $\gamma=0$ and
ellipses for $\gamma>0$
(Theorem~\ref{thm:associator-ellipse}).
The cross marks $-i$, the image under $M_{0,\gamma}$ for every
$\gamma$.
(b) Images of the same circle under the six maps at $\gamma=1$,
giving two distinct ellipses and the point $-i$.
Panels (a) and (b) show the $i,j,k$ components of the images.
(c) Zero sets of $\det L_p^1$ and $\det R_p^1$ restricted to
$p=w+yj+zk$ on the unit sphere $w^2+y^2+z^2=1$,
shown in $(w,y,z)$ coordinates.}
\label{fig:bracket-low-dimensional}
\end{figure}

\subsection{Left and right zero divisors
}

For all $p\in\mathcal A_\gamma$ let
\[
 \Delta_L(p;\gamma):=\det L_p^{\gamma},
 \quad
 \Delta_R(p;\gamma):=\det R_p^{\gamma}.
\]
The operators $L_p^\gamma$ and $R_p^\gamma$ depend linearly on $p$;
their determinants $\Delta_L(p;\gamma)$ and $\Delta_R(p;\gamma)$ are
homogeneous polynomials of degree $5$. An element $p\neq0$ is a left zero divisor  in $\mathcal{A}_\gamma$, i.e. $p\star_\gamma x=0$ for some $x\neq0$, if
and only if $L_p^\gamma$ is not injective, i.e. $\Delta_L(p;\gamma)=0$, and the same holds for right zero divisors and $\Delta_R$. Since $x\star_\gamma p=\overline{\bar p\star_\gamma\bar x}$, we have
$R_p^\gamma=C\,L_{\bar p}^\gamma\,C$, where $C(x)=\bar x$ satisfies $\det C=1$. Thus we obtain
\[
 \Delta_R(p;\gamma)=\Delta_L(\bar p;\gamma)
\]
for all $p\in\mathcal A_\gamma$.

\begin{remark}
Critical points of the quadratic maps need not be zero divisors of the
algebra. For $\gamma>0$, $p=i$ is critical for $M_{0,\gamma}$, whereas
$\det L_i^\gamma=-\det R_i^\gamma=\sqrt2\,\kappa_\gamma\sigma_\gamma\neq0$,
so $i$ is neither a left nor a right zero divisor.
\end{remark}

For $p=p_0+p_2j+p_3k$, substituting $p_1=p_4=0$
into \eqref{eq:left-multiplication-matrix} and taking
the determinant gives $\Delta_L(p;\gamma)$ below;
the formula for $\Delta_R(p;\gamma)$ follows from
$\Delta_R(p;\gamma)=\Delta_L(\bar p;\gamma)$.
$$\Delta_L(p;\gamma)=\frac{p_0^2+p_2^2+p_3^2}{1+\gamma^2}
((1+\gamma^2)p_0^3+p_0(p_2^2+p_3^2)
+\sqrt2\gamma(p_2^3+p_3^3)),$$
$$\Delta_R(p;\gamma)=\frac{p_0^2+p_2^2+p_3^2}{1+\gamma^2}
((1+\gamma^2)p_0^3+p_0(p_2^2+p_3^2)
-\sqrt2\gamma(p_2^3+p_3^3)).$$
 Thus, for $\gamma=0$,
the nonzero zero divisors in $\Span\{1,j,k\}$ are exactly the elements with $p_0=0$, and they are
both left and right zero divisors. For $\gamma>0$ the left and right zero divisors in
$\Span\{1,j,k\}$ are different: they are the zeros of the two cubic factors (which incidentally  are mirror
images of each other under $(p_2,p_3)\mapsto(-p_2,-p_3)$).

At the quaternion endpoint $\gamma=0$ the two singular curves coincide with an equator of the unit sphere.  Positive switching bends them in opposite directions, giving a direct geometric signature of left--right asymmetry.

\section{A classical spin model}
\label{sec:spin-dynamics}

We now use the multiplication of $\mathcal A_\gamma$ to define an energy for a unit vector
$x\in H_\gamma$, which we interpret as the direction of a classical spin, and study the resulting
dynamics.

\subsection{An energy on $H_\gamma$}

Let $ Q=\Span\{i,j,k\}=\Span\{\textbf a,\textbf b,\textbf c\}$
and let $\{q_1,q_2,q_3\}$ be any orthonormal basis of $Q$.  For $x\in H_\gamma$, define
\begin{equation}
 \mathcal T_\gamma(x)
 =\frac14\sum_{r=1}^3
 \left(
 \norm{q_r\star_\gamma x}^2+
 \norm{x\star_\gamma q_r}^2
 \right).
 \label{eq:multiplication-energy-main}
\end{equation}

Since $q\star_\gamma x=-\langle q,x\rangle+q\times_\gamma x$ and $x\star_\gamma q=-\langle q,x\rangle-q\times_\gamma x$
for $q,x\in V$, we have
$$\mathcal T_\gamma(x)=\frac12\sum_{r=1}^3\norm{q_r\star_\gamma x}^2.$$

\begin{proposition}[Basis independence]
\label{prop:energy-basis-independence}
The quantity $\mathcal T_\gamma(x)$ is independent of the chosen orthonormal basis of $Q$.
\end{proposition}

\begin{proof}
For fixed $x$, the two sums in \cref{eq:multiplication-energy-main} are squared Hilbert--Schmidt norms of the linear maps $q\mapsto q\star_\gamma x$ and $q\mapsto x\star_\gamma q$ from $Q$ to $\SwitchAlg_\gamma$.  Hilbert--Schmidt norms are invariant under orthogonal changes of domain basis.
\end{proof}

Write $x=y_h\,h_\gamma+y_{\mathbf b}\,\mathbf {\mathbf b}+y_{\mathbf c}\,\mathbf c$, $y=(y_h,y_{\mathbf b},y_{\mathbf c})^T$, and put
$C=\cos\beta_\gamma=\sqrt3/\sqrt{3+\gamma^2}$, $S=\sin\beta_\gamma=\gamma/\sqrt{3+\gamma^2}$,
$\kappa=\kappa_\gamma$, $\sigma=\sigma_\gamma$, so that $h_\gamma=C\,\mathbf a+S\,\mathbf d$. Using \eqref{eq:adapted-products-main}  we see that
\begin{align}
 {\mathbf a}\star_\gamma x={}&-Cy_h-\sqrt2Sy_h {\mathbf a}-\kappa y_{\mathbf c} {\mathbf b}+\kappa y_{\mathbf b}{\mathbf c},\label{eq:ax-spin}\\
 {\mathbf b}\star_\gamma x={}&-y_{\mathbf b}+\kappa y_{\mathbf c} {\mathbf a}+\frac{S}{\sqrt2}y_h{\mathbf b}
 -\kappa Cy_h\, {\mathbf c}+\sqrt3\sigma y_{\mathbf c}\,{\mathbf d},\label{eq:bx-spin}\\
 {\mathbf c}\star_\gamma x={}&-y_{\mathbf c}-\kappa y_{\mathbf b}{\mathbf a}+\kappa Cy_h{\mathbf b}
 +\frac{S}{\sqrt2}y_h{\mathbf c}-\sqrt3\sigma y_{\mathbf b}{\mathbf d}.\label{eq:cx-spin}
\end{align}

\begin{theorem}
\label{thm:multiplication-anisotropy}
For all $x\in H_\gamma$, the Energy \cref{eq:multiplication-energy-main} has the exact form
\begin{equation}
 \mathcal T_\gamma(x)=\frac12y^TK_\gamma y,
 \qquad
 K_\gamma=\diag(A_\gamma,B_\gamma,B_\gamma),
 \label{eq:Kgamma-main}
\end{equation}
where
\begin{equation}
 A_\gamma=
 \frac{3(\gamma^4+2\gamma^2+3)}{(1+\gamma^2)(3+\gamma^2)},
 \qquad
 B_\gamma=\frac{4\gamma^2+3}{1+\gamma^2}.
 \label{eq:ABgamma-main}
\end{equation}
In particular, we have $ K_0=3I_3,$ and $K_1=\diag\left(\frac94,\frac72,\frac72\right).$
\end{theorem}

\begin{proof}
Taking norms in \eqref{eq:ax-spin}--\eqref{eq:cx-spin} and cancelling we get
\begin{align*}
\norm{\mathbf a\star_\gamma x}^2&=(C^2+2S^2)y_h^2+\kappa^2(y_b^2+y_{\mathbf c}^2),\\
\norm{\mathbf b\star_\gamma x}^2&=(\tfrac{S^2}{2}+\kappa^2C^2)y_h^2+y_{\mathbf b}^2+(\kappa^2+3\sigma^2)y_{\mathbf c}^2,\\
\norm{\mathbf c\star_\gamma x}^2&=(\tfrac{S^2}{2}+\kappa^2C^2)y_h^2+(\kappa^2+3\sigma^2)y_{\mathbf b}^2+y_{\mathbf c}^2.
\end{align*}
Since $\mathcal T_\gamma(x)$ is half the sum of these three terms, $A_\gamma=C^2+3S^2+2\kappa^2C^2$ and
$B_\gamma=1+2\kappa^2+3\sigma^2$. Substituting
\[
 C^2=\frac3{3+\gamma^2},\quad S^2=\frac{\gamma^2}{3+\gamma^2},\quad
 \kappa^2=\frac1{1+\gamma^2},\quad \sigma^2=\frac{\gamma^2}{1+\gamma^2}
\]
 yields \cref{eq:Kgamma-main,eq:ABgamma-main}.
\end{proof}

At $\gamma>0$, we have $A_\gamma<B_\gamma$, since
$A_\gamma-B_\gamma=-\frac{\gamma^2(\gamma^2+9)}{(1+\gamma^2)(3+\gamma^2)}$. Thus the energy is isotropic for $\gamma=0$,
and for $\gamma>0$ it distinguishes the axis $\R h_\gamma$ (a uniaxial anisotropy).

\begin{figure}[t]
\centering
\includegraphics[width=0.84\textwidth]{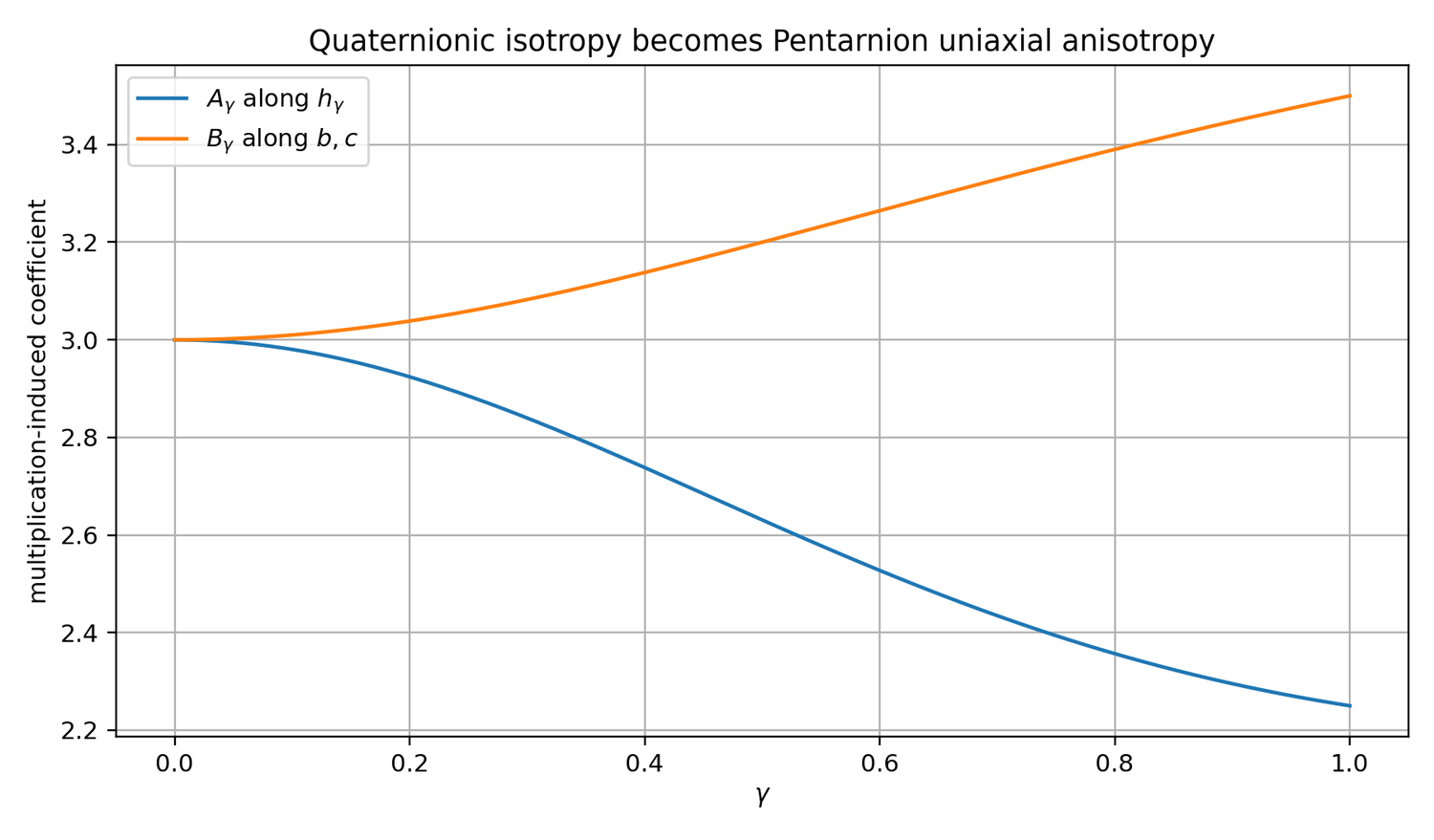}
\caption{The coefficients $A_\gamma$ and $B_\gamma$ of Theorem~\ref{thm:multiplication-anisotropy}
for $0\le\gamma\le1$. They agree for $\gamma=0$ and separate for $\gamma>0$.
}
\label{fig:spin-anisotropy}
\end{figure}

\subsection{Equations of motion and conserved quantities}

For $\gamma>0$, the energy $\mathcal T_\gamma$ is invariant under rotations of $H_\gamma$ about $h_\gamma$.

To remove this symmetry we add a constant field in the direction $i$. Its orthogonal projection onto
$H_\gamma$, written in the basis $h_\gamma,\mathbf b,\mathbf c$, has coordinates
\begin{equation}
 g_\gamma=(\langle i,h_\gamma\rangle,\langle i,\mathbf b\rangle,\langle i,\mathbf c\rangle)^T
 =((3+\gamma^2)^{-1/2},\,2^{-1/2},\,6^{-1/2})^T.
 \label{eq:g-gamma-main}
\end{equation}
 This field does not depend on $\gamma$, only these
coordinates do.

For all $y\in\R^3$ with $\norm y=1$, we define
\begin{equation}
 \mathcal H_\gamma(y)=\tfrac12\,y^TK_\gamma y-g_\gamma^Ty,
 \label{eq:spin-H-main}
\end{equation}
and consider the  Lie-Poisson equation
\begin{equation}
 \dot y=y\times_3\nabla\mathcal H_{\gamma(t)}(y),
 \label{eq:spin-ode-main}
\end{equation}
where $\times_3$ is the cross product on $\R^3$.

Put $K_\gamma=\diag(A,B,B)$ and $g_\gamma=(g_1,g_2,g_3)^T$, then we thus obtain a system of equations with respect to the basis $h_\gamma$, ${\mathbf b}$, ${\mathbf c}$ given by
\begin{align}
 \dot y_h&=-g_3y_{\mathbf b}+g_2y_{\mathbf c},\label{eq:spin-component-1}\\
 \dot y_{\mathbf b}&=(A-B)y_hy_{\mathbf c}-g_1y_{\mathbf c}+g_3y_h,\label{eq:spin-component-2}\\
 \dot y_{\mathbf c}&=(B-A)y_hy_{\mathbf b}-g_2y_h+g_1y_{\mathbf b}.\label{eq:spin-component-3}
\end{align}
Its Jacobian matrix is
\begin{equation}
D_yF=
\begin{pmatrix}
0&-g_3&g_2\\
(A-B)y_{\mathbf c}+g_3&0&(A-B)y_h-g_1\\
(B-A)y_{\mathbf b}-g_2&(B-A)y_h+g_1&0
\end{pmatrix}.
\label{eq:spin-jacobian}
\end{equation}
The time dependence enters through $\gamma(t)$ in $A,B,g_1$.

\begin{theorem}
\label{thm:static-spin-integrability}
For fixed $\gamma$, every solution of of \eqref{eq:spin-ode-main} satisfies
\[
 \norm{y(t)}=\norm{y(0)},
 \qquad
\mathcal H_\gamma(y(t))= \mathcal H_\gamma(y(0)).
\]
In particular,  the solutions on the unit sphere move along level curves of $\mathcal H_\gamma$, and there is no chaotic motion.
\end{theorem}

\begin{proof}
Since $y\times w$ is orthogonal to $y$ and to $w$, we have
\[
 \tfrac{\dd}{\dd t}\norm y^2=2y\cdot(y\times_3\nabla\mathcal H_\gamma)=0,
 \quad \tfrac{\dd}{\dd t}\mathcal H_\gamma(y)=\nabla\mathcal H_\gamma\cdot(y\times_3\nabla\mathcal H_\gamma)=0 .
\]
Hence each solution lies in the intersection of the unit sphere with a level set of $\mathcal H_\gamma$.
\end{proof}

\begin{proposition}
\label{thm:work-identity}
Let $\gamma(t)$ be differentiable and let $y(t)$ be a solution of \eqref{eq:spin-ode-main}.
Then
\begin{equation}
\frac{\dd}{\dd t} \mathcal H_{\gamma(t)}(y(t))
 =\frac{\partial \mathcal H_\gamma}{\partial\gamma}(y(t))\dot\gamma(t).
 \label{eq:work-law-main}
\end{equation}
Thus the energy changes only through the  dependence of $\mathcal H_\gamma$ on $\gamma$.
\end{proposition}

\begin{proof}
Apply the chain rule. Then
the left-hand side is the right-hand side plus
$\nabla\mathcal H_\gamma\cdot\dot y=\nabla\mathcal H_\gamma\cdot(y\times_3\nabla\mathcal H_\gamma)=0$.
\end{proof}

\begin{remark}
\label{thm:energy-enclosure}
For $0\leq\gamma\leq1$ and $\norm y=1$,  we have $\frac18\le\mathcal H_\gamma(y)\le\frac{11}{4}$.
 Indeed, for $u=\gamma^2$,
$A_\gamma=3(1-\frac{2u}{u^2+4u+3})$ decreases from $3$ to $\frac94$ on $[0,1]$, since
$\frac{2u}{u^2+4u+3}$ is increasing for $0\le u<\sqrt3$, and $B_\gamma=4-\frac1{1+u}$ increases from $3$ to
$\frac72$. Hence $\frac94\leq y^TK_\gamma y\leq\frac72$. Moreover, we have
$\norm{g_\gamma}^2=\frac1{3+\gamma^2}+\frac23\leq 1$, so $|g_\gamma^Ty|\leq 1$ by the Cauchy-Schwarz Inequality. This implies that
$\frac98-1\leq\mathcal H_\gamma(y)\leq\frac74+1$.
\end{remark}

In the fixed basis $i,j,k,\ell$ of $V$, the spin is the vector
\[
 x(t)=y_h(t) h_{\gamma(t)}+y_{\mathbf b}(t)\,\mathbf b+y_{\mathbf c}(t)\,\mathbf c\in H_{\gamma(t)}
\]
with the $\ell$-coordinate
\begin{equation}
 \langle x(t),\ell\rangle=\frac{\gamma(t)}{\sqrt{3+\gamma(t)^2}}y_h(t).
 \label{eq:spin-ell-coordinate}
\end{equation}
Thus for $\gamma(t)>0$ and $y_h(t)\neq0$ the vector $x(t)$ leaves $\Span\{i,j,k\}$, while
$\norm{x(t)}=\norm{y(t)}=1$, since $h_\gamma,\mathbf b,\mathbf c$ are orthonormal and $\norm{y}$ is conserved (note that the proof of Theorem~\ref{thm:static-spin-integrability} applies verbatim for time-dependent $\gamma$).

\subsection{Periodic switching and numerical regimes}

The simulations use
\begin{equation}
\gamma(t)=\frac12\bigl(1+\sin(\Omega t)\bigr),
 \label{eq:spin-protocol}
\end{equation}
with
\[
y(0)=\frac{(0.45,0.65,-0.61)^T}{\norm{(0.45,0.65,-0.61)}}.
\]

The system \eqref{eq:spin-component-1}--\eqref{eq:spin-component-3} and its linearisation
\eqref{eq:spin-jacobian} are integrated by the fourth-order Runge--Kutta method. The step size divides the
forcing period exactly, $y$ is projected back to $S^2$ after each step, and two tangent vectors are
periodically orthogonalised in the tangent plane to compute the Lyapunov exponents. The energy change
is also computed by an independent quadrature of the right-hand side of \eqref{eq:work-law-main}, to check
\eqref{eq:work-law-main} numerically. One point is recorded after every forcing period at a fixed phase.

The exact period map preserves area on $S^2$, so its two Lyapunov
exponents sum to zero. Agreement with this identity provides a numerical
check; exact area preservation is not asserted for the projected
Runge--Kutta discretisation.

The frequency scan in \cref{fig:spin-frequency-scan} indicates regular motion for small $\Omega$, a broad
range of intermediate $\Omega$ with a positive Lyapunov exponent, and regular motion again for large $\Omega$.

For $\Omega=0.30$ both Lyapunov exponents are close to zero, and the points recorded once per forcing
period, over $900$ periods after a transient, lie on a closed curve, indicating quasi-periodic motion.

 The evidence record follows $900$ forcing periods after the transient, retains $24{,}001$ long-run trajectory samples and gives an invariant-curve stroboscopic section.  The trajectory panel draws the complete retained post-transient horizon as one continuous, time-ordered line; the component panel shows a separate representative late-time interval at full stored resolution.

\begin{figure}[t]
\centering
\begin{subfigure}[t]{0.48\textwidth}
\includegraphics[width=\textwidth]{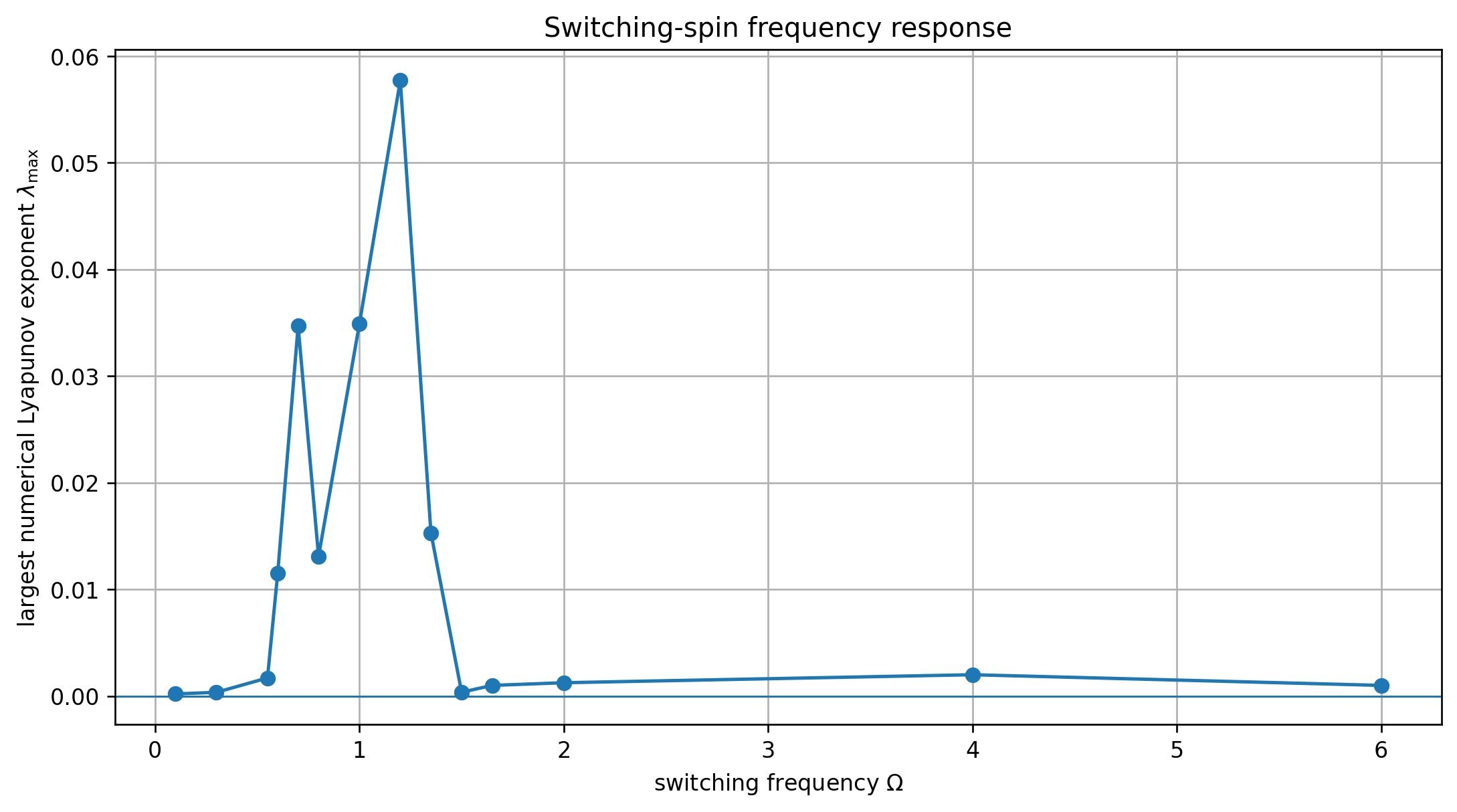}
\caption{Largest numerical Lyapunov exponent versus $\Omega$.}
\label{fig:spin-frequency-scan}
\end{subfigure}\hfill
\begin{subfigure}[t]{0.48\textwidth}
\includegraphics[width=\textwidth]{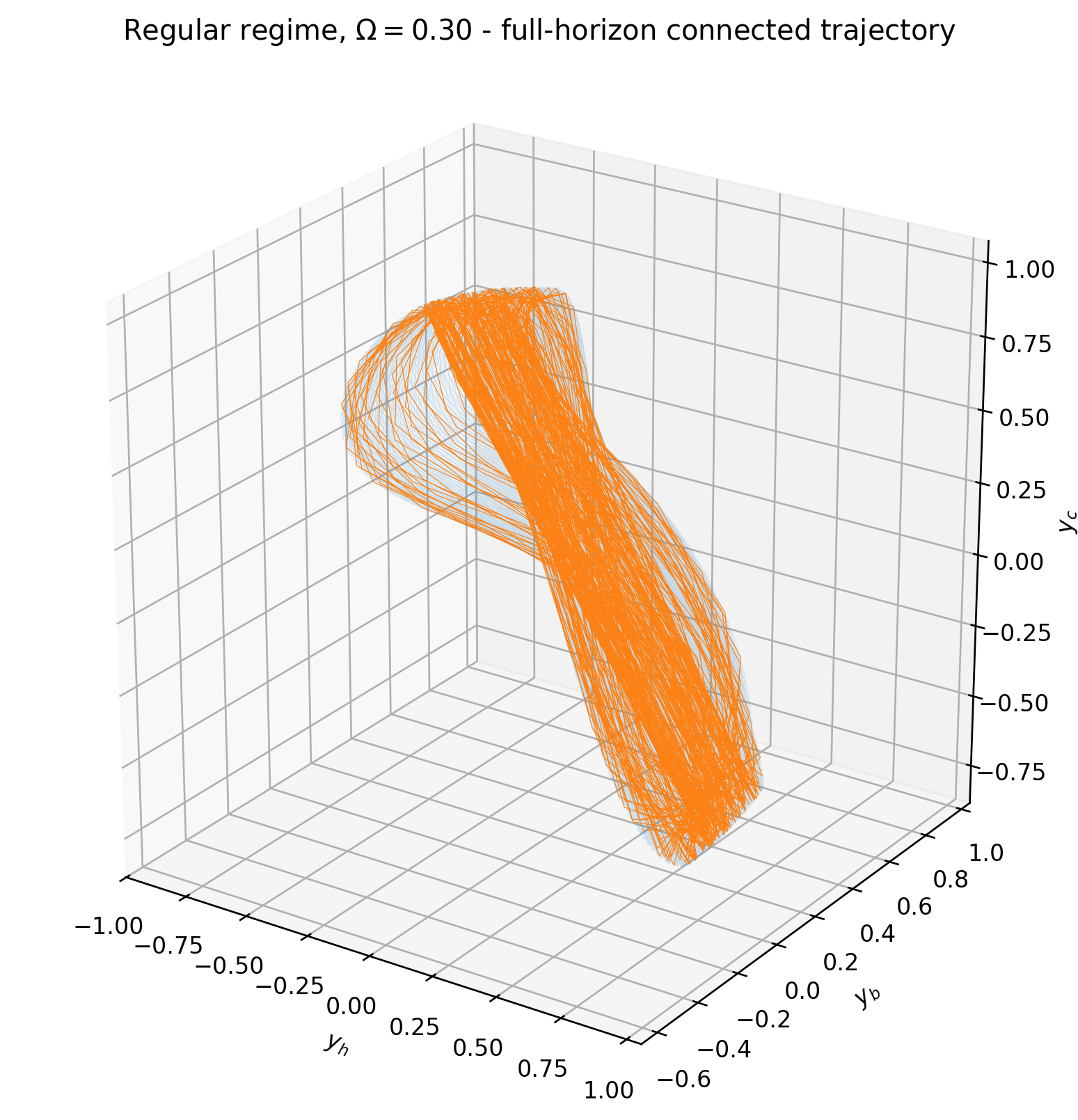}
\caption{Full-horizon connected regular trajectory.}
\label{fig:spin-regular-trajectory}
\end{subfigure}
\caption{Frequency organisation and a representative regular moving-frame orbit.  The scan shows regular--irregular--regular organisation; the trajectory is drawn across the complete retained horizon as a densely sampled connected curve.}
\end{figure}

\begin{figure}[t]
\centering
\begin{subfigure}[t]{0.48\textwidth}
\includegraphics[width=\textwidth]{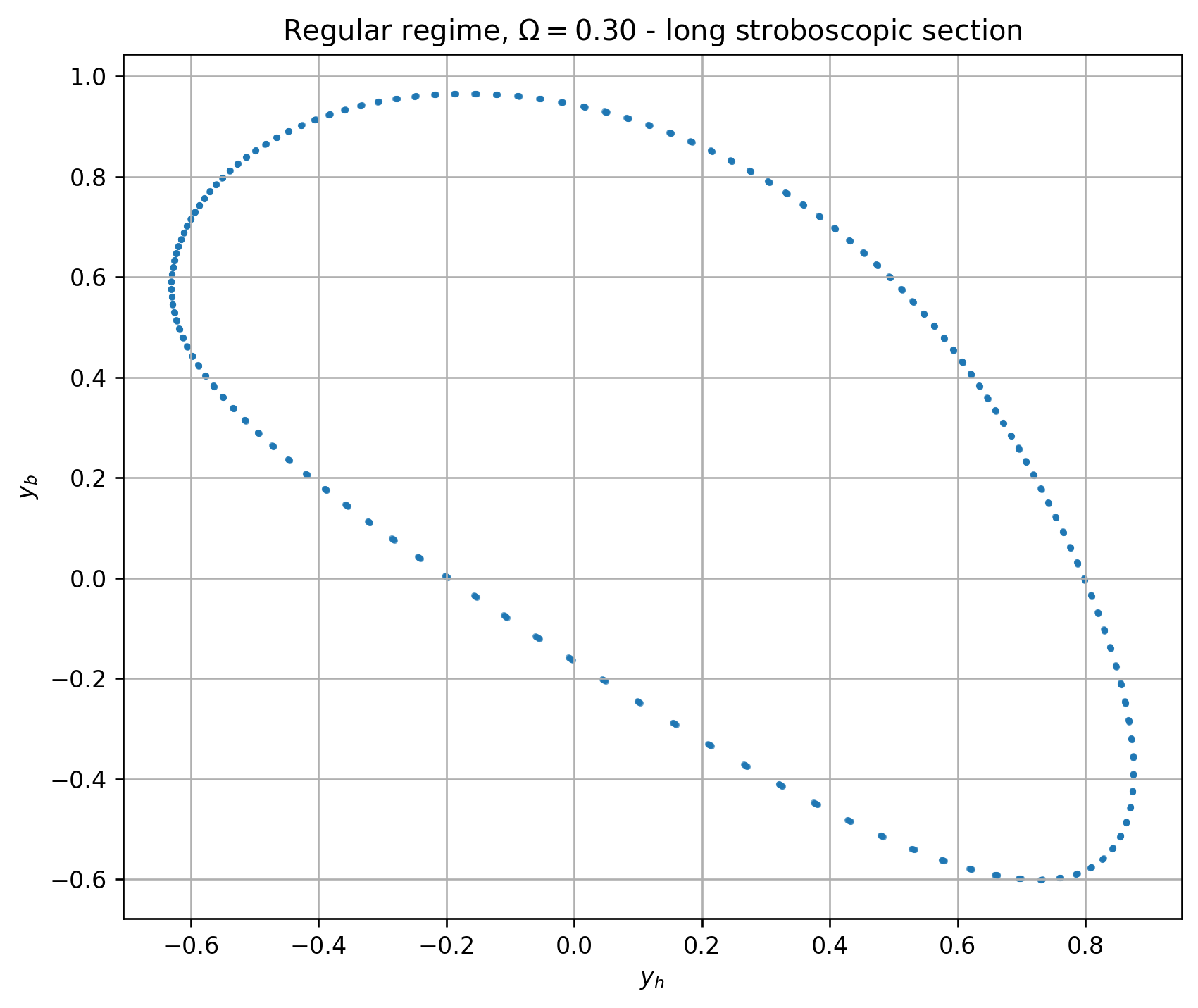}
\caption{Stroboscopic section from $900$ forcing periods.}
\label{fig:spin-regular-poincare}
\end{subfigure}\hfill
\begin{subfigure}[t]{0.48\textwidth}
\includegraphics[width=\textwidth]{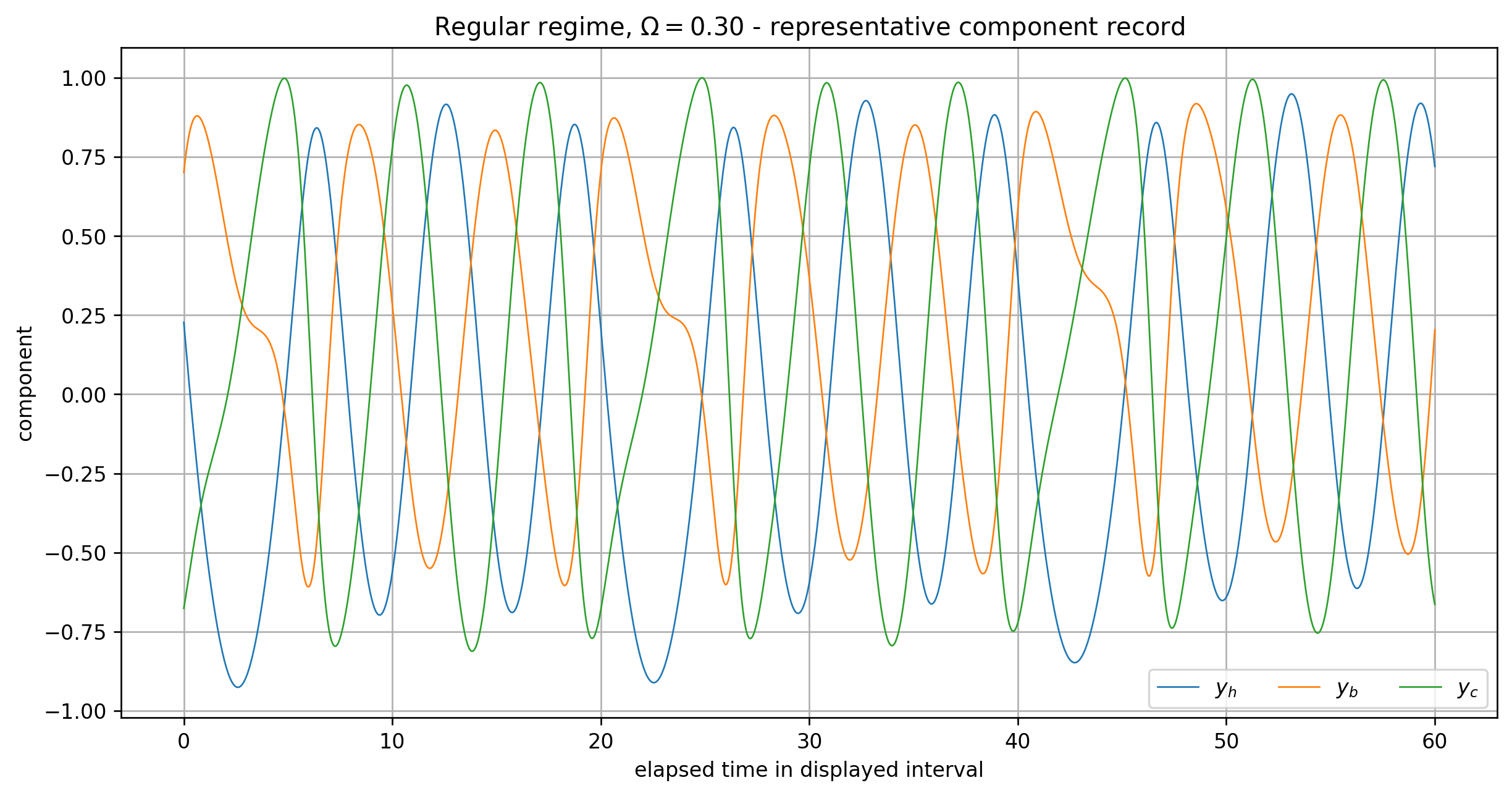}
\caption{Representative dense component record.}
\label{fig:spin-regular-timeseries}
\end{subfigure}
\caption{Complementary diagnostics for the regular switching-spin regime.}
\end{figure}

At $\Omega=1.20$ the largest numerical exponent is approximately $0.0577$.  The long evidence record retains $91{,}001$ post-transient samples and $3{,}000$ forcing phases.  The trajectory panel draws the complete retained post-transient horizon as one continuous line, while the signal panel shows a representative late-time interval at the full stored resolution.  The long stroboscopic section fills a two-dimensional region, consistently with the positive exponent.

\begin{figure}[t]
\centering
\begin{subfigure}[t]{0.48\textwidth}
\includegraphics[width=\textwidth]{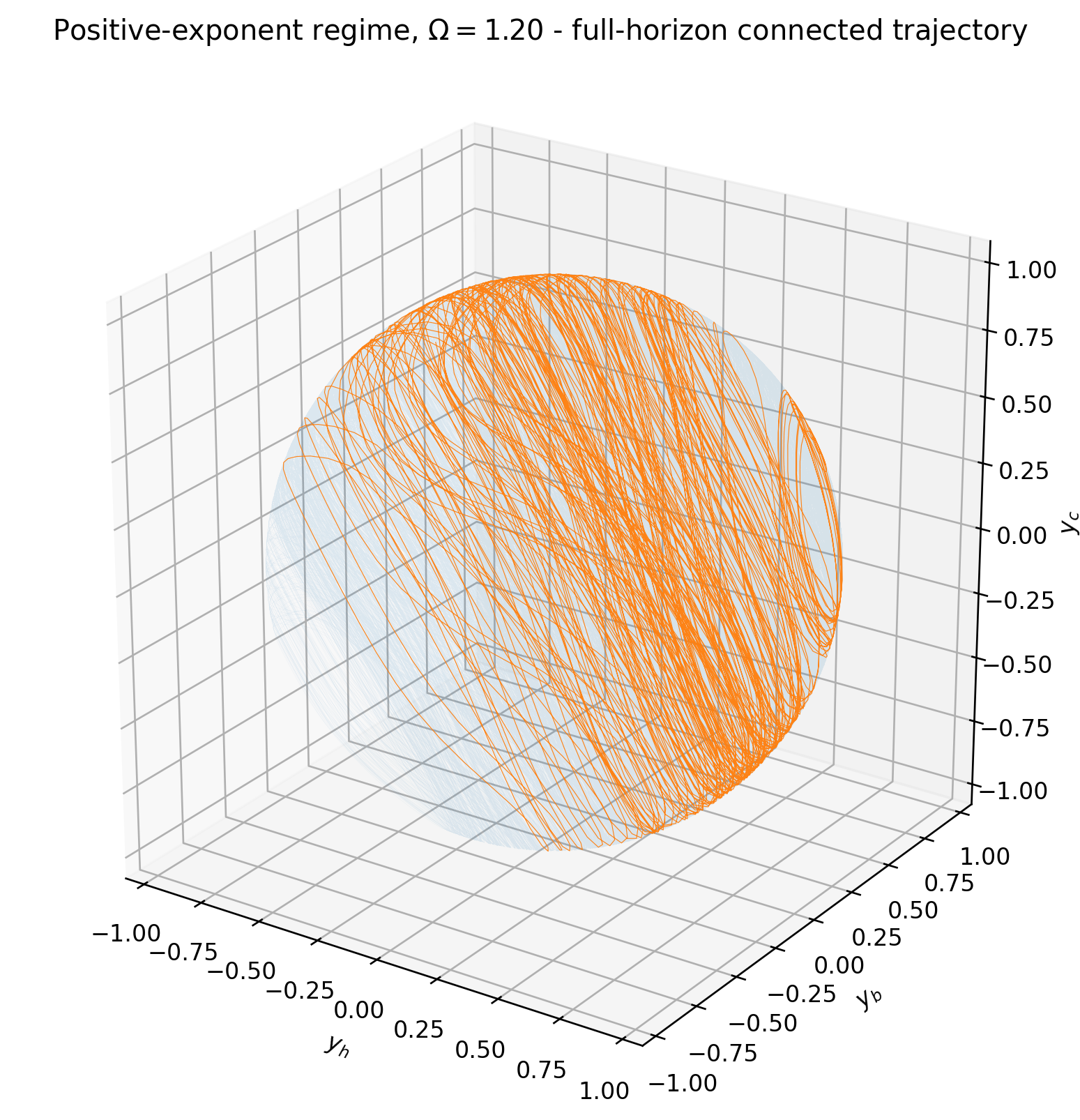}
\caption{Full-horizon connected moving-frame trajectory.}
\label{fig:spin-chaotic-trajectory}
\end{subfigure}\hfill
\begin{subfigure}[t]{0.48\textwidth}
\includegraphics[width=\textwidth]{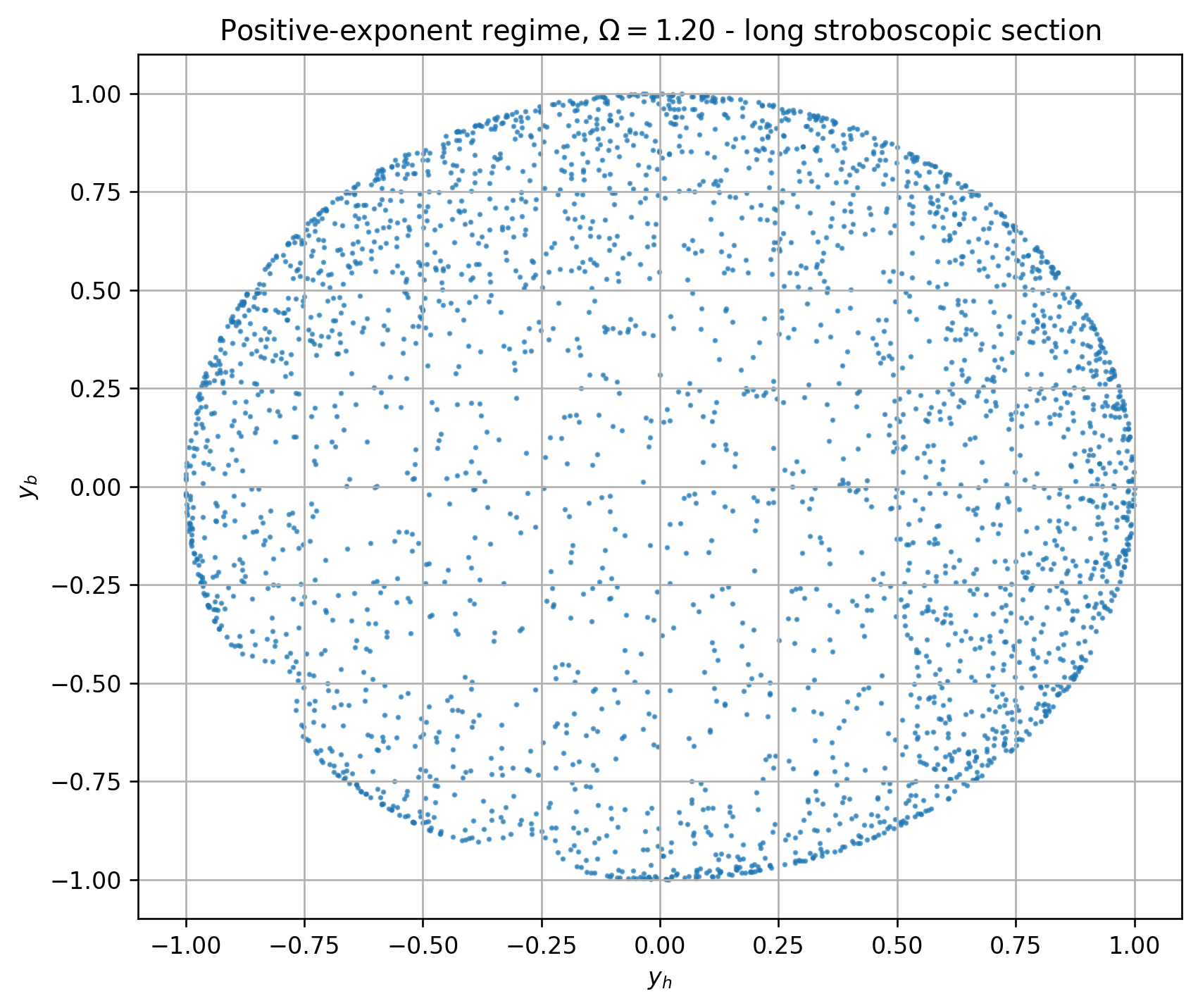}
\caption{Long stroboscopic section from $3{,}000$ periods.}
\label{fig:spin-chaotic-poincare}
\end{subfigure}
\caption{Trajectory and stroboscopic geometry in the positive-exponent regime.}
\end{figure}

The laboratory embedding confirms that the motion uses the fourth pure coordinate rather than remaining in the quaternionic core.

\begin{figure}[t]
\centering
\begin{subfigure}[t]{0.48\textwidth}
\includegraphics[width=\textwidth]{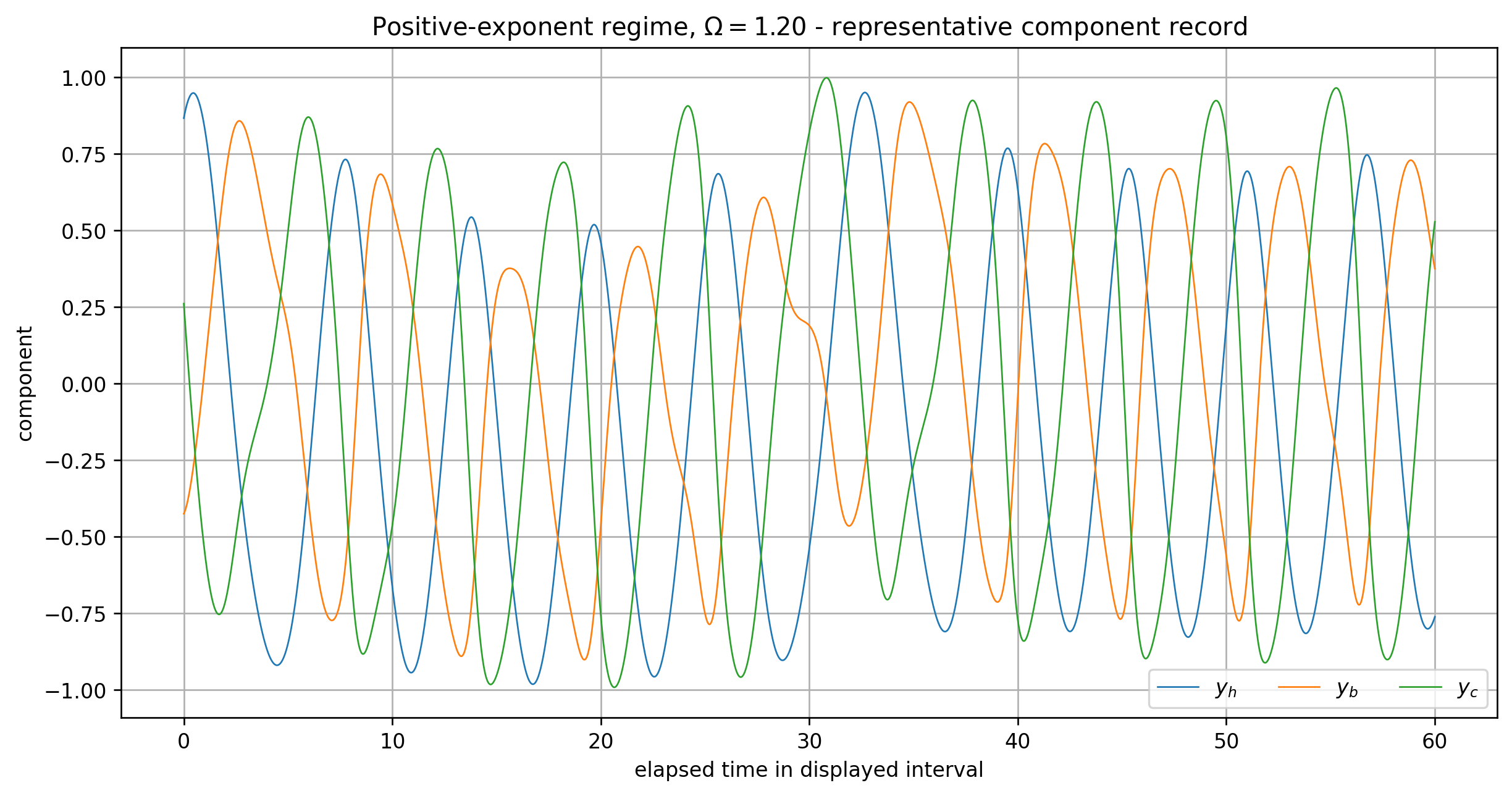}
\caption{Representative dense moving-frame signal.}
\label{fig:spin-chaotic-timeseries}
\end{subfigure}\hfill
\begin{subfigure}[t]{0.48\textwidth}
\includegraphics[width=\textwidth]{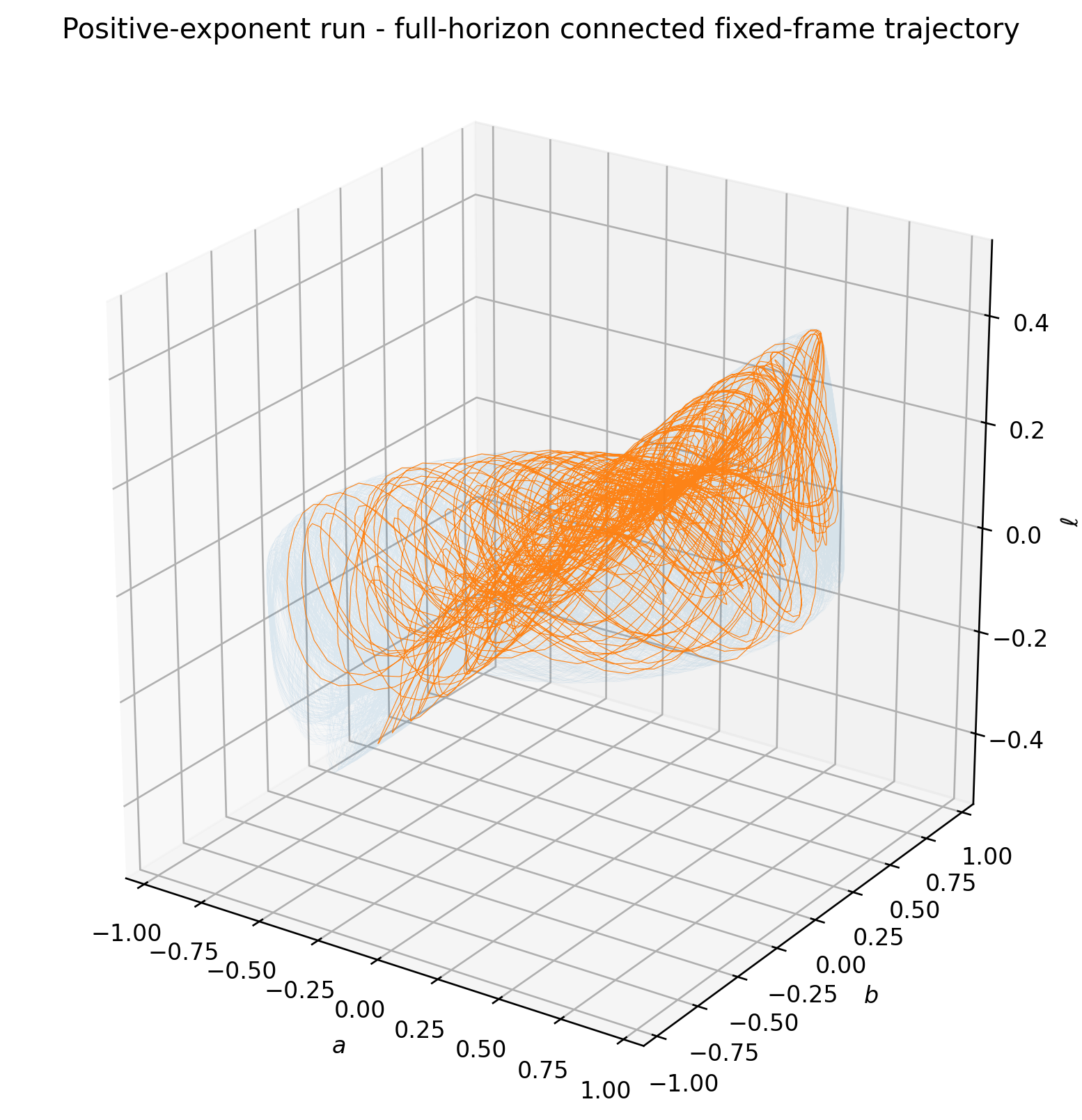}
\caption{Full-horizon connected fixed-frame projection $(a,b,\ell)$.}
\label{fig:spin-fixed-frame}
\end{subfigure}
\caption{Representative component behaviour and the full-horizon connected fixed-frame trajectory of the positive-exponent regime. The nonzero $\ell$ excursion is induced by the moving rotational plane.}
\end{figure}

The numerically accumulated parametric power agrees with the measured energy change; the maximum work-law residual in the long positive-exponent run is approximately $2.16\times10^{-6}$.

\begin{figure}[t]
\centering
\includegraphics[width=0.82\textwidth]{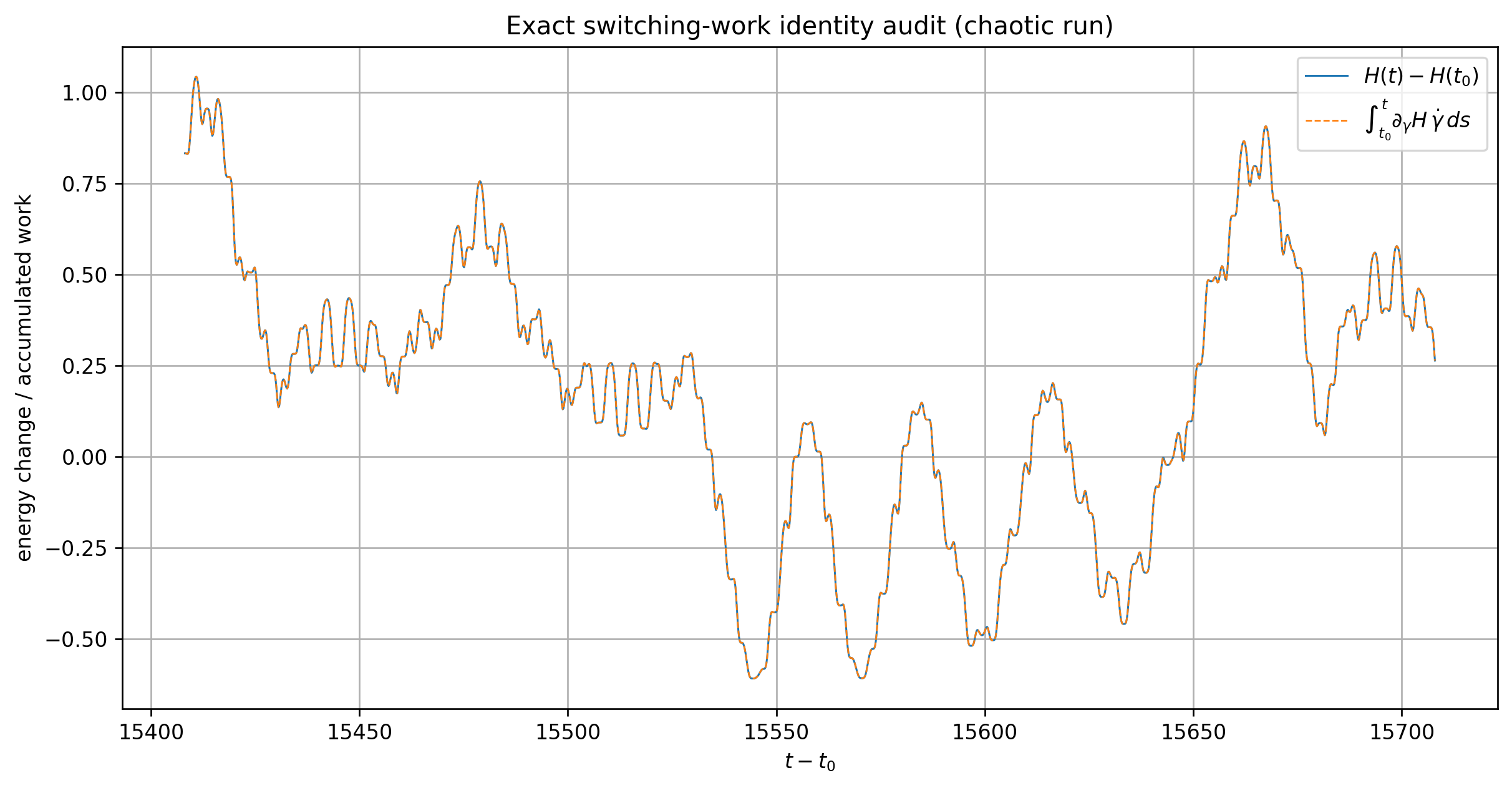}
\caption{Long-run audit of the exact work identity \cref{eq:work-law-main}.}
\label{fig:spin-work-audit}
\end{figure}

\begin{remark}
The results on $K_\gamma$, the equations of motion, the conserved quantities, the energy identity
\eqref{eq:work-law-main} and the energy bounds in Remark~\ref{thm:energy-enclosure} are proved. The evidence for
irregular motion is numerical: a positive Lyapunov exponent together with the scattered points of the
period map. A rigorous proof would require computer-assisted methods such as interval arithmetic for
the period map, which we do not attempt here.

The spin calculations use the Runge--Kutta and tangent-plane variational
methods described above. Time series of the coordinates are shown on
late-time intervals. The available code, supporting numerical results and
scope of reproducibility are described in
Appendix~\ref{app:reproducibility}.

\end{remark}

For the displayed dense trajectories, the regular run at $\Omega=0.30$
uses $800$ steps per forcing period and discards $300$ periods; the
positive-exponent run at $\Omega=1.20$ uses $910$ steps per period and
discards $500$ periods. The retained horizons are $900$ and $3{,}000$
periods, respectively. After each step the state is projected to the unit
sphere, and stroboscopic samples are taken at one fixed forcing phase.
The corresponding residuals in the work identity are
$8.37\times10^{-6}$ and $2.16\times10^{-6}$.
These settings concern the dense trajectory records, not the separate
Lyapunov-frequency scan.

\section{Open questions}
Several questions remain open.
For the algebras $\mathcal A_\gamma(a,b,c,d)$,
it would be useful to determine how the automorphism groups and nuclei  depend on $a,b,c,d$.
The numerical
dynamical examples also provide candidates for  computer-assisted proofs, for
example by rigorous interval arithmetic.
Finally, one
may allow $\gamma$ to depend on the state or spatial position,
leading to  state-dependent
or spatially varying  equations in which the
multiplication itself changes during the evolution.

\null
\FloatBarrier
\appendix
\section{Coordinate product and multiplication matrices}
\label{app:coordinate-matrices}

Throughout this appendix,
$\SwitchAlg_\gamma=\mathcal A_\gamma(1,1,1,1)$,
$\Pent=\SwitchAlg_1$,
$\Vpure=\Span_\R\{i,j,k,\ell\}$, $\kg=\frac{1}{\sqrt{1+\gamma^2}},$ and $\sg=\frac{\gamma}{\sqrt{1+\gamma^2}}$.
 When $\gamma$ is clear from the context,
we  use juxtaposition for $\star_\gamma$. Let
\[
 p=p_0+p_1i+p_2j+p_3k+p_4\ell,
 \qquad
 q=q_0+q_1i+q_2j+q_3k+q_4\ell.
\]
Then
\[
 p\star_\gamma q=C_0+C_1i+C_2j+C_3k+C_4\ell,
\]
where
\begin{align}
C_0={}&p_0q_0-p_1q_1-p_2q_2-p_3q_3-p_4q_4,\notag\\
C_1={}&p_0q_1+q_0p_1+\kg(p_2q_3-p_3q_2)
 +\frac{-p_2q_4-p_3q_4+p_4q_2+p_4q_3}{\sqrt2},\notag\\
C_2={}&p_0q_2+q_0p_2+\kg(-p_1q_3+p_3q_1)
 +\frac{-p_1q_4-p_3q_4+p_4q_1+p_4q_3}{\sqrt2},\notag\\
C_3={}&p_0q_3+q_0p_3+\kg(p_1q_2-p_2q_1)
 +\frac{-p_1q_4-p_2q_4+p_4q_1+p_4q_2}{\sqrt2},\notag\\
C_4={}&p_0q_4+q_0p_4
 +\sg(p_1q_2-p_1q_3-p_2q_1+p_2q_3+p_3q_1-p_3q_2).
\label{eq:complete-coordinate-product}
\end{align}

With respect to the canonical basis $1,i,j,k,\ell$, left multiplication is given by
\begin{equation}
L_p^{\gamma}=
\begin{pmatrix}
p_0&-p_1&-p_2&-p_3&-p_4\\
p_1&p_0&-\kg p_3+p_4/\sqrt2&\kg p_2+p_4/\sqrt2&-(p_2+p_3)/\sqrt2\\
p_2&\kg p_3+p_4/\sqrt2&p_0&-\kg p_1+p_4/\sqrt2&-(p_1+p_3)/\sqrt2\\
p_3&-\kg p_2+p_4/\sqrt2&\kg p_1+p_4/\sqrt2&p_0&-(p_1+p_2)/\sqrt2\\
p_4&\sg(-p_2+p_3)&\sg(p_1-p_3)&\sg(-p_1+p_2)&p_0
\end{pmatrix}.
\label{eq:left-multiplication-matrix}
\end{equation}
If $C=\diag(1,-1,-1,-1,-1)$, we have
\begin{equation}
R_p^{\gamma}=C L_{\bar p}^{\gamma}C.
\label{eq:right-multiplication-matrix}
\end{equation}

Fix $\gamma$ and a coefficient $A\in\SwitchAlg_\gamma$.
For the following six quadratic maps, write
$J_r=DF_r(p)$ for the Jacobian at $p$, with $A$ and $\gamma$
held fixed. Products of multiplication operators denote
composition of real-linear maps.

\begin{center}
\begin{tabular}{ll}
\toprule
Map & Jacobian\\
\midrule
$F_1(p)=A\star_\gamma(p\star_\gamma p)$
& $J_1=L_A^{\gamma}(R_p^{\gamma}+L_p^{\gamma})$\\
$F_2(p)=(A\star_\gamma p)\star_\gamma p$
& $J_2=R_p^{\gamma}L_A^{\gamma}+L_{A\star_\gamma p}^{\gamma}$\\
$F_3(p)=p\star_\gamma(A\star_\gamma p)$
& $J_3=R_{A\star_\gamma p}^{\gamma}+L_p^{\gamma}L_A^{\gamma}$\\
$F_4(p)=(p\star_\gamma A)\star_\gamma p$
& $J_4=R_p^{\gamma}R_A^{\gamma}+L_{p\star_\gamma A}^{\gamma}$\\
$F_5(p)=p\star_\gamma(p\star_\gamma A)$
& $J_5=R_{p\star_\gamma A}^{\gamma}+L_p^{\gamma}R_A^{\gamma}$\\
$F_6(p)=(p\star_\gamma p)\star_\gamma A$
& $J_6=R_A^{\gamma}(R_p^{\gamma}+L_p^{\gamma})$\\
\bottomrule
\end{tabular}
\end{center}
These formulas follow by differentiating with respect to $p$
and using the bilinearity of $\star_\gamma$. For example,
\[
\begin{aligned}
 DF_2(p)[h]
 &=(A\star_\gamma h)\star_\gamma p
   +(A\star_\gamma p)\star_\gamma h\\
 &=\left(R_p^\gamma L_A^\gamma
       +L_{A\star_\gamma p}^\gamma\right)h.
\end{aligned}
\]
The other identities follow in the same way, with the displayed
orderings and parentheses preserved.

For $A=i$, these formulas recover the Jacobians of the maps
defined in \cref{sec:bracket-geometry}. More precisely, with
$A=i$,
\[
 (F_1,F_2,F_3,F_4,F_5,F_6)
 =
 (M_{0,\gamma},M_{2,\gamma},M_{3,\gamma},
  M_{4,\gamma},M_{5,\gamma},M_{1,\gamma}).
\]

\section{Code and data availability}
\label{app:reproducibility}\label{app:D}

The code and supporting data for Part~I are available in
\path{quaternions-to-pentarnions/part-I/} of the repository
\url{https://github.com/poshtkohi/pentarnion-algebra}.
The repository includes symbolic verification routines, numerical
simulation programs, figure-generation scripts and supporting data.
Software requirements, data descriptions and instructions for
reproducing the results are provided in
\path{docs/reproducibility.md} within that folder.

\section{Declaration of generative AI and AI-assisted technologies in the manuscript preparation process}

The presented results and diagrams are our own. Claude 5.5 was
used in the last step to assist with proofreading the manuscript and to identify possible gaps in the proofs. We
independently checked all such suggestions, and we take full responsibility for the correctness and
content of the result. We support the Leiden Declaration on Artificial Intelligence and Mathematics.

\providecommand{\bysame}{\leavevmode\hbox to3em{\hrulefill}\thinspace}

\end{document}